%% file: igabem3d.tex
\documentclass[12pt,reqno]{amsart}

\usepackage{amsmath,amssymb,ifthen}
\usepackage{hyperref}
\usepackage{fullpage}
\usepackage{graphicx,psfrag,subfigure}
\usepackage{color}
\usepackage{moreverb}
\usepackage{amsthm}
\usepackage{mathrsfs}
\usepackage{hhline}
\usepackage{bbm}
\usepackage[english]{babel}
\usepackage{tikz}
\usepackage[utf8]{inputenc}
\usepackage{amsfonts}
\usepackage{enumerate}

\def\Nu{\boldsymbol{\rm N}}

\def\MM{\mathcal M}

\def\UU{\mathcal{U}}
\def\D{D}
\def\T{\mathbb{T}}
\def\R{{\mathbb R}}
\def\C{{\mathbb C}}

\def\N{{\mathbb N}}
\def\BB{{\mathcal B}}

\def\KK{{\mathcal K}}
\def\NN{{\mathcal N}}

\def\OO{{\mathcal O}}

\def\SS{{\mathcal S}}
\def\TT{{\mathcal T}}
\def\XX{{\mathcal X}}

\def\diam{{\rm diam}}
\def\q#1{q_{\rm #1}}

\def\ro#1{\rho_{\rm #1}}

\def\const#1{C_{\rm #1}}

\def\norm#1#2{\|#1\|_{#2}}
\def\seminorm#1#2{\vert #1\vert_{#2}}
\def\set#1#2{\big\{#1\,:\,#2\big\}}
\def\sprod#1#2{(#1\,;\,#2)}
\def\dual#1#2{\langle#1\,;\,#2\rangle}

\def\level{{\rm level}}

\def\Trunc{{\rm Trunc}}
\def\refine{{\tt refine}}
\def\supp{{\rm supp}}

\def\dist{{\rm dist}}

\def\uni#1{{\rm uni}(#1)}
\def\coarse{\bullet}
\def\fine{\circ}

\def\revision#1{{\color{black}{#1}}}
\numberwithin{equation}{section}
\numberwithin{figure}{section}
\newtheorem{theorem}{Theorem}[section]
\newtheorem{proposition}[theorem]{Proposition}
\newtheorem{lemma}[theorem]{Lemma}

\newtheorem{algorithm}[theorem]{Algorithm}

\newtheorem{remark}[theorem]{Remark}

\def\subsection#1
{
 \bigskip

 \refstepcounter{subsection}
 {\noindent\bf\arabic{section}.\arabic{subsection}.~#1.~}
}
\renewcommand{\subsection}[1]{\refstepcounter{subsection}\medskip{\bf\thesubsection.~#1.}}

\usepackage{fancyhdr}
\advance\footskip0.4cm
\advance\textheight-0.4cm
\calclayout
\newcommand*\patchAmsMathEnvironmentForLineno[1]{%
  \expandafter\let\csname old#1\expandafter\endcsname\csname #1\endcsname
  \expandafter\let\csname oldend#1\expandafter\endcsname\csname end#1\endcsname
  \renewenvironment{#1}%
     {\linenomath\csname old#1\endcsname}%
     {\csname oldend#1\endcsname\endlinenomath}}% 
\newcommand*\patchBothAmsMathEnvironmentsForLineno[1]{%
  \patchAmsMathEnvironmentForLineno{#1}%
  \patchAmsMathEnvironmentForLineno{#1*}}%
\AtBeginDocument{%
\patchBothAmsMathEnvironmentsForLineno{equation}%
\patchBothAmsMathEnvironmentsForLineno{align}%
\patchBothAmsMathEnvironmentsForLineno{flalign}%
\patchBothAmsMathEnvironmentsForLineno{alignat}%
\patchBothAmsMathEnvironmentsForLineno{gather}%
\patchBothAmsMathEnvironmentsForLineno{multline}%
}
\usepackage[mathlines]{lineno}

\begin{document}

\title{Adaptive BEM for elliptic PDE systems, \\ part II: Isogeometric analysis\\with hierarchical B-splines \\ for weakly-singular integral equations}

\author{Gregor Gantner}
\email{Gregor.Gantner@asc.tuwien.ac.at}

%\address{Korteweg-de Vries (KdV) Institute for Mathematics, University of Amsterdam, P.O. Box 94248, 1090 GE Amsterdam, The Netherlands.}

\author{Dirk Praetorius}
\email{Dirk.Praetorius@asc.tuwien.ac.at}

\address{TU Wien,
 Institute of Analysis and Scientific Computing,
 Wiedner Hauptstra\ss{}e 8-10,
 A-1040 Wien, Austria}

% Classification

% Abstract
\begin{abstract}
We formulate and analyze an adaptive algorithm for isogeometric analysis with hierarchical B-splines for weakly-singular boundary integral equations. 
We prove that the employed weighted-residual error estimator is reliable and converges at optimal algebraic rate. 
Numerical experiments with isogeometric boundary elements for the 3D Poisson problem confirm the theoretical results, which also cover general elliptic systems like linear elasticity.

\end{abstract}

% Make title
\date{\today}
\maketitle
%\thispagestyle{fancy}

% Contents
\input{01_introduction}

\input{02_preliminaries}
\input{04_definition_high}

\input{06_numeric_high}

\input{03_axioms}
%\input{05_proof_high}

% Acknowledgement
\section*{Acknowledgement} The authors acknowledge support through the Austrian Science Fund (FWF) under grant P29096,  %\textit{Optimal isogeometric boundary element method} and the FWF doctoral school \textit{Dissipation and Dispersion in Nonlinear PDEs} funded under 
grant W1245, and grant J4379.

% Bibliography
\bibliographystyle{alpha}
\bibliography{literature}

\end{document}

%% file: 01_introduction.tex
% !TEX encoding = MacOSRoman
% !TEX root = igabem3d.tex

\section{Introduction}

\subsection{State of the art}
The central idea of isogeometric analysis (IGA) is to use the same ansatz functions for the discretization of the partial differential equation (PDE) at hand, as are used for the representation of the problem geometry. 
Usually, the problem geometry $\Omega$ is represented in computer aided design (CAD) by means of non-uniform rational B-splines (NURBS), T-splines, or hierarchical splines. 
This concept, originally invented in~\cite{hcb05} for finite element methods (IGAFEM) has proved very fruitful in applications; see also the monograph \cite{bible}. 

Usually, CAD programs only provide  a parametrization of the boundary $\partial\Omega$ instead of the domain $\Omega$ itself.
In particular, for isogeometric FEM, the parametrization needs to be extended to the whole domain $\Omega$, which is non-trivial and still an active research topic.  
The \textit{boundary element method} (BEM) circumvents this difficulty by working only on  the CAD provided boundary  $\partial \Omega$. 
However, compared to the IGAFEM literature, only little is found for isogeometric BEM (IGABEM). 
Based on a collocation approach, the latter was first  considered   in \cite{igabem2d} for 2D and   in \cite{igabem3d} for 3D. 
We refer to the monograph~\cite{bmd20} for an introduction on IGABEM and to \cite{dhksw18,dhkmsw20} for an efficient IGABEM implementation based on the fast multipole method. 
However, to the best of our knowledge, {\sl {\sl a~posteriori}} error estimation and adaptive mesh-refinement for IGABEM, have only been considered for simple 2D model problems in the own works~\cite{fgp,resigabem,resigaconv,fgps19,gps20}, which employ B-splines on the one-dimensional boundary, and the recent work~\cite{fgkss19}, which employs hierarchical B-splines. 
In particular, \cite{fgp}  also appears to be the first work that considers Galerkin IGABEM. 

For standard BEM with (dis)continuous piecewise polynomials, {\sl {\sl a~posteriori}} error estimation and adaptive mesh-refinement are well understood. 
We refer to \cite{arcme} for  an overview on available {\sl a~posteriori} error estimation strategies. Moreover, optimal convergence of mesh-refining adaptive algorithms has been proved for polyhedral boundaries~\cite{part1,part2,fkmp} as well as smooth boundaries~\cite{gantumur}. 
The  work~\cite{invest} allows to transfer these results to piecewise smooth boundaries; see also the discussion in the review article~\cite{axioms}.

In the frame of adaptive IGAFEM, a rigorous error and convergence analysis is first  found in \cite{bg}, which  proves linear convergence for some adaptive IGAFEM with hierarchical splines for the Poisson equation.
Optimal algebraic rates have been proved independently from each other in \cite{bg17,igafem} for IGAFEM with (truncated) hierarchical B-splines.
In particular, \cite{igafem} especially identified certain abstract properties on the mesh-refinement and the used ansatz spaces that automatically guarantee optimal convergence. 
In \cite{gp20b}, we have recently verified these properties for IGAFEM with T-splines as well. 

Unlike adaptive IGAFEM, the rigorous error and convergence analysis of adaptive IGABEM is completely open beyond the 2D results mentioned above and thus the focus of the present work. In Part~1 of this work~\cite{gp20}, we followed the idea of \cite{igafem} and identified an abstract framework, which guarantees optimal convergence of adaptive BEM for weakly-singular integral equations.
In the present work (Part~2), we now show that IGABEM with hierarchical splines is indeed covered by this abstract framework.

\subsection{\revision{Outline}}
The remainder of this manuscript is roughly organized as follows:
\revision{Section~\ref{sec:preliminaries}  recalls conforming BEM for weakly-singular integral equations; see~\eqref{eq:strong} for the precise model problem.} 
Section~\ref{sec:hierarchical bem} defines hierarchical meshes and hierarchical splines on the boundary  $\Gamma$ and introduces some local mesh-refinement rule  (Algorithm~\ref{alg:refinement bem}) which preserves admissibility
in the sense of a certain mesh-grading property.
The main result of Section~\ref{sec:hierarchical bem} is Theorem~\ref{thm:main bem}, \revision{which states reliability~\eqref{eq:reliable bem} of the weighted-residual error estimator~\eqref{eq:eta bem} as well as linear convergence~\eqref{eq:linear bem} with optimal algebraic rates~\eqref{eq:optimal bem}  of a standard adaptive algorithm (Algorithm~\ref{alg:bem algorithm}) applied to the model problem at hand.}
Remark~\ref{rem:rational main bem} extends the result to rational hierarchical splines.
The proof of Theorem~\ref{thm:main bem} \revision{(and Remark~\ref{rem:rational main bem})} is given in \revision{Section~\ref{sec:abstract setting bem}, which}  recalls the abstract framework from~\cite{gp20} \revision{and subsequently verifies the corresponding properties in the  considered isogeometric setting}. 
Two numerical experiments in Section~\ref{sec:numerics for bem1} underpin the theoretical results, but also demonstrate the limitations  of hierarchical splines in the frame of adaptive BEM when the solution $\phi$ exhibits edge singularities.

%To prove Theorem~\ref{thm:main bem}, we verify the properties from Section~\ref{sec:abstract setting bem} for (rational) hierarchical splines.
%In particular, we derive the following inverse inequality (Sections \ref{subsec:S inverse bem} and \ref{sec:rational main proof bem})
%\begin{align}\label{eq:beminvX intro}
%\norm{h_\bullet^{1/2}\Psi_\coarse}{L^2(\Gamma)}\le \const{inv} \norm{\Psi_\coarse}{H^{-1/2}(\Gamma)} \quad \text{for all }\Psi_\coarse\in\XX_\coarse,
%\end{align} 
%where $\const{inv}>0$ and $\XX_\coarse$ denotes the space of all (rational) hierarchical splines on some mesh $\TT_\coarse$.
%Further, we construct a quasi-interpolation projection $J_{\bullet,\TT_\coarse}:L^2(\Gamma)^D\to\XX_\bullet$ which is locally $L^2$-stable (Section~\ref{subsec:scottbem} and Section~\ref{sec:rational main proof bem}).

%% file: 02_preliminaries.tex
% !TEX encoding = MacOSRoman
% !TEX root = igabem3d.tex

\section{Preliminaries}\label{sec:preliminaries}
In this section, we fix some general notation, recall Sobolev spaces on the boundary, and precisely state the considered problem. 
Throughout the work, let $\Omega\subset\R^d$ be a bounded Lipschitz domain as in \cite[Definition~3.28]{mclean} with boundary $\Gamma:= \partial\Omega$.

\subsection{General notation}
Throughout and without any ambiguity, $|\cdot|$ denotes the absolute value of scalars, the Euclidean norm of vectors in $\R^n$, and the Hausdorff measure of any $n$- or lower-dimensional set in $\R^n$. 
Let $B_\varepsilon(x):=\set{y\in\R^n}{|x-y|<\varepsilon}$ denote the open ball around $x$ with radius $\varepsilon>0$. 
For $\emptyset \neq\omega_1,\omega_2\subseteq \R^n$, let $B_\varepsilon(\omega_1):=\bigcup_{x\in\omega_1} B_\varepsilon(x)$. 
Moreover, let $\diam(\omega_1):=\sup\set{|x-y|}{x,y\in\omega_1}$ and $\dist(\omega_1,\omega_2):=\inf\set{|x-y|}{x\in\omega_1, y\in\omega_2}$.
We write $A\lesssim B$ to abbreviate $A\le CB$ with some generic constant $C>0$, which is clear from the context.
Moreover $A\simeq B$ abbreviates $A\lesssim B\lesssim A$.
Throughout, mesh-related quantities have the same index, e.g., $\XX_\coarse$ is the ansatz space corresponding to the mesh $\TT_\coarse$. 
The analogous notation is used for meshes $\TT_\fine,\TT_\star,\TT_\ell$, etc.

\subsection{Sobolev spaces}\label{subsec:sobolev}
For $\sigma\in[0,1]$, we define the Hilbert spaces $H^{\pm\sigma}(\Gamma)$ as in \cite[page~99]{mclean} by use of Bessel potentials on $\R^{d-1}$ and liftings via bi-Lipschitz mappings\footnote{For $\widehat\omega\subseteq \R^{d-1}$ and $\omega\subseteq \R^d$, a mapping $\gamma:\widehat\omega\to\omega$ is bi-Lipschitz if it is bijective and $\gamma$ as well as its inverse $\gamma^{-1}$ are Lipschitz continuous.} that describe $\Gamma$.
For $\sigma=0$, it holds that $H^0(\Gamma)=L^2(\Gamma)$ with equivalent norms.
We set $\norm{\cdot}{H^0(\Gamma)}:=\norm{\cdot}{L^2(\Gamma)}$.

For $\sigma\in(0,1]$,  any measurable subset $\omega\subseteq\Gamma$, and all $v\in H^\sigma(\Gamma)$, we define the associated   Sobolev--Slobodeckij norm
\begin{align}\label{eq:SS-norm}
 \norm{v}{H^{\sigma}(\omega)}^2
 := \norm{v}{L^2(\omega)}^2
 + |v|_{H^{\sigma}(\omega)}^2\text{ with }
 |v|_{H^{\sigma}(\omega)}^2 :=\begin{cases} \int_\omega\int_\omega\frac{|v(x)-v(y)|^2}{|x-y|^{d-1+2\sigma}}\,dx dy&\text{ if }\sigma\in(0,1),\\ \norm{\nabla_\Gamma v}{L^2(\omega)}^2&\text{ if }\sigma=1.\end{cases}
\end{align}
It is well-known that $\norm{\cdot}{H^\sigma(\Gamma)}$ provides an equivalent norm on $H^\sigma(\Gamma)$; see, e.g., \cite[Lemma~2.19]{s} and \cite[Theorem~3.30 and page 99]{mclean} for $\sigma\in(0,1)$ and \cite[Theorem~2.28]{gme} for $\sigma=1$.
Here, $\nabla_\Gamma(\cdot)$ denotes the usual (weak) surface gradient which can be  defined for almost all $x\in\Gamma$ as follows:
Since $\Gamma$ is a Lipschitz boundary, there exist an open cover  $(O_j)_{j=1}^J$ in $\R^d$ of $\Gamma$ such that each $\omega_j:=O_j\cap \Gamma$ can be parametrized by  a bi-Lipschitz mapping
 $\gamma_{\omega_j}: \widehat\omega_j\to \omega_j$, where  $\widehat \omega_j\subset\R^{d-1}$  is an  open set.
By Rademacher's theorem, $\gamma_{\omega_j}$ is  almost everywhere differentiable. 
\revision{With $D(\cdot)$ denoting the weak Jacobian matrix,}
the corresponding Gram determinant $\det(D\gamma_{\omega_j}^\top D\gamma_{\omega_j})$ is almost everywhere positive. %; see Lemma~\ref{lem:bi-Lipschitz} below.
Moreover, by definition of the space $H^1(\Gamma)$, $v\in H^1(\Gamma)$ implies that $v\circ\gamma_{\omega_j}\in H^1(\widehat\omega_j)$.
With the weak derivative $\nabla (v\circ\gamma_{\omega_j})\in L^2(\widehat\omega_j)^d$, we can  hence define 
\begin{align}\label{eq:surface gradient}
(\nabla_\Gamma v )|_{\omega_j} :=\big(D\gamma_{\omega_j}(D\gamma_{\omega_j}^\top D\gamma_{\omega_j})^{-1} \nabla (v\circ\gamma_{\omega_j})\big)\circ\gamma_{\omega_j}^{-1}
\quad\text{for all }v\in H^1(\Gamma).
\end{align}
This definition does not depend on the particular choice of the open sets $(O_j)_{j=1}^J$
and the corresponding parametrizations $(\gamma_{\omega_j})_{j=1}^J$; see, e.g., \cite[Theorem~2.28]{gme}.
With \eqref{eq:surface gradient}, we immediately obtain the chain rule 
\begin{align}\label{eq:chain rule}
\nabla(v\circ\gamma_{\omega_j})=D\gamma_{\omega_j}^\top((\nabla_\Gamma v)\circ\gamma_{\omega_j}(\cdot))\quad\text{for all }v\in H^1(\Gamma).
\end{align}

For $\sigma\in(0,1]$, $H^{-\sigma}(\Gamma)$ is a realization of the dual space of $H^{\sigma}(\Gamma)$ according to \cite[Theorem~3.30 and page~99]{mclean}.
With the dual bracket $\dual{\cdot}{\cdot}$, we define an equivalent norm 
\begin{align}
\norm{\psi}{H^{-\sigma}(\Gamma)}:=\sup \set{\dual{v}{\psi}}{v\in H^\sigma(\Gamma)\wedge\norm{v}{H^\sigma(\Gamma)}=1} \quad\text{for all } \psi\in H^{-\sigma}(\Gamma).
\end{align}
Moreover, we abbreviate
\begin{align}
\sprod{v}{\psi}:=\dual{\overline v}{\psi}\quad \text{for all } v\in H^\sigma(\Gamma),\psi\in H^{-\sigma}(\Gamma).
\end{align}

\cite[page~76]{mclean} states that $H^{\sigma_1}(\Gamma)\subseteq H^{\sigma_2}(\Gamma)$ for $-1\le \sigma_1\le\sigma_2\le 1$,  where the inclusion is continuous and dense.
In particular, $H^{\sigma}(\Gamma)\subset L^2(\Gamma)\subset H^{-\sigma}(\Gamma)$ forms a Gelfand triple in the sense of \cite[Section~2.1.2.4]{ss} for all $\sigma\in(0,1]$, where $\psi\in L^2(\Gamma)$ is interpreted as function in $H^{-\sigma}(\Gamma)$ via 
\begin{align}
\dual{v}{\psi}:=\sprod{\overline v}{\psi}_{L^2(\Gamma)}=\int_\Gamma v\,\psi \,dx \quad\text{for all }v\in H^\sigma(\Gamma),\psi\in L^2(\Gamma).
\end{align}

So far, we have only dealt with scalar-valued functions. 
For $D\ge1$, $\sigma\in[0,1]$, $v=(v_1,\dots,v_D)\in H^{\sigma}(\Gamma)^D$, we define  $\norm{v}{H^{\pm\sigma}(\Gamma)}^2:=\sum_{j=1}^D \norm{v_j}{H^{\pm\sigma}(\Gamma)}^2$. 
If $\sigma>0$, and $\omega\subseteq\Gamma$ is an arbitrary measurable set, we define $\norm{v}{H^\sigma(\omega)}$ and $\seminorm{v}{H^\sigma(\omega)}$ similarly.
With the definition
\begin{align}
\nabla_\Gamma v:=
\begin{pmatrix}
\nabla_\Gamma v_1\\
\vdots\\
\nabla_\Gamma v_D
\end{pmatrix}
\in L^2(\Gamma)^{D^2}\quad\text{for all }v\in H^1(\Gamma)^D,
\end{align}
it holds that
$\seminorm{v}{H^1(\omega)}=\norm{\nabla_\Gamma v}{L^2(\omega)}$.
Note that $H^{-\sigma}(\Gamma)^D$ with $\sigma\in(0,1]$ can be identified with the dual space of $H^{\sigma}(\Gamma)^D$, where we set
\begin{align}
\dual{v}{\psi}:=\sum_{j=1}^D\dual{v_j}{\psi_j}\quad\text{for all }v\in H^\sigma(\Gamma)^D,\psi\in H^{-\sigma}(\Gamma)^D.
\end{align}
As before we abbreviate
\begin{align}
\sprod{v}{\psi}:=\dual{\overline v}{\psi}\quad \text{for all } v\in H^\sigma(\Gamma)^D,\psi\in H^{-\sigma}(\Gamma)^D.
\end{align}
and set
\begin{align}
\dual{v}{\psi}:=\sprod{\overline v}{\psi}_{L^2(\Gamma)}=
%\sum_{j=1}^D\dual{v_j}{\psi_j}=
\sum_{j=1}^D\int_\Gamma v_j \psi_j\,dx\quad\text{for all }v\in H^\sigma(\Gamma)^D,\psi\in L^2(\Gamma)^D.
\end{align}

{The spaces $H^\sigma(\Gamma)$ can be also defined as trace spaces or via interpolation, where the resulting norms are always equivalent with constants which depend only on the dimension $d$ and the boundary $\Gamma$.}
More details and proofs are found, e.g., in the monographs  \cite{mclean,ss,s}.

\subsection{Model problem}\label{subsec:model problem bem}
We consider a general second-order linear  system of PDEs on the $d$-dimensional bounded Lipschitz domain $\Omega$ with partial differential operator  
\begin{align}\label{eq:PDE bem}
\begin{split}
\mathfrak{P}u:=
-\sum_{i=1}^d \sum_{i'=1}^d\partial_i (A_{ii'}\partial_{i'} u)+\sum_{i=1}^d b_i\partial_i u +c u,
\end{split}
\end{align}
where the coefficients $A_{ii'},b_i,c\in\C^{D\times D}$ are constant for some fixed dimension $D\ge 1$.
We suppose that $A_{ii'}^\top=\overline{A_{i'i}}$.
Moreover, we  assume that $\mathfrak{P}$ is coercive on $H_0^1(\Omega)^D$, i.e., the sesquilinear form 
\begin{align}\label{eq:PDE form}
\sprod{u}{v}_{\mathfrak{P}}:=\int_\Omega \sum_{i=1}^d \sum_{i'=1}^d(A_{ii'}\partial_{i'} u)\cdot \partial_i v+\sum_{i=1}^d (b_i\partial_i u)\cdot v +(c u)\cdot v\,dx
\end{align}
is  elliptic up to some compact perturbation.
This is equivalent to \textit{strong ellipticity}
%\footnote{Unfortunately, this name might be misleading. Indeed, strong ellipticity in the sense of \cite{mclean} does not necessarily  imply ellipticity.} 
of the matrices $A_{ii'}$ in the sense of \cite[page~119]{mclean}.

Let $G:\R^d\setminus\{0\}\to \C^{D\times D}$ be a corresponding (matrix-valued) fundamental solution in the sense of \cite[page~198]{mclean}, i.e., a distributional solution of $\mathfrak{P}G=\delta$, where $\delta$ denotes the Dirac delta function.
For $\psi\in L^\infty(\Gamma)^D$, we define the \textit{single-layer operator}
as
\begin{align}\label{eq:single layer operator integral}
({\mathfrak{V}}\psi)(x):=\int_{\Gamma} G(x-y) \psi(y) \,dy\quad\text{for all }x\in\Gamma.
\end{align}
According to \cite[page 209 and 219--220]{mclean} and \cite[%Theorem~3.32, 
Corollary~3.38]{mitrea}, 
%and Proposition~3.39
 this operator can be extended for arbitrary $\sigma\in(-1/2,1/2$] to a bounded linear operator 
\begin{align}\label{eq:single layer operator}
\mathfrak{V}:%=(\widetilde{\mathfrak{V}}\psi)|_{\Gamma}:
H^{-1/2+\sigma}(\Gamma)^D\to H^{1/2+\sigma}(\Gamma)^D.
\end{align}
For $\sigma=0$, \cite[Theorem~7.6]{mclean} states that $\mathfrak{V}$ is always elliptic up to some compact perturbation.
We assume that it is elliptic even without perturbation, i.e., 
\begin{align}\label{eq:ellipticity bem}
\mathrm{Re}\,\sprod{\mathfrak{V}\psi}{\psi}
\ge \const{ell}\norm{\psi}{H^{-1/2}(\Gamma)}^2\quad\text{for all }\psi\in H^{-1/2}(\Gamma)^D.
\end{align}
\revision{For instance, this is} satisfied for the Laplace problem or for the Lam\'e problem, where the case $d=2$ requires an additional scaling of the geometry $\Omega$; see, e.g., \cite[Chapter~6]{s}.
\revision{We stress that $\mathfrak V$ is not necessarily self-adjoint; see~\cite{mclean}.}
The sesquilinear form $\sprod{\mathfrak{V}\,\cdot}{\cdot}%_{L^2(\Gamma)}
$ is continuous due to \eqref{eq:single layer operator}, i.e.,  it holds with $\const{cont}:=\norm{\mathfrak{V}}{H^{-1/2}(\Gamma)^D\to H^{1/2}(\Gamma)^D}$ that 
\begin{align}\label{eq:continuity bem}
|\sprod{\mathfrak{V}\psi}{\xi}|%_{L^2(\Gamma)}
\le \const{cont}\norm{\psi}{H^{-1/2}(\Gamma)}\norm{\xi}{H^{-1/2}(\Gamma)} \quad\text{for all }\psi,\xi\in H^{-1/2}(\Gamma)^D.
\end{align}

Given a right-hand side $f\in H^{1}(\Gamma)^D$, 
we consider the boundary integral equation% in the abstract form
\begin{align}\label{eq:strong}
 \mathfrak{V}\phi = f.
% \quad\text{for all }x\in\Gamma_{},
\end{align}
Such equations arise from the solution of Dirichlet problems of the form $\mathfrak{P} u=0$ in $\Omega$ with $u=g$ on $\Gamma$ for some $g\in H^{1/2}(\Gamma)^D$; see, e.g., \cite[page 226--229]{mclean} for more details.
%where $\mathfrak{V}:H^{-1/2}(\Gamma_{})\to H^{1/2}(\Gamma_{})$ is an elliptic isomorphism.
%Here $H^{1/2}(\Gamma_{})$ is a fractional-order Sobolev space, and
%$H^{-1/2}(\Gamma_{})$ is its dual (see Section~\ref{subsec:sobolev} below).
 The Lax--Milgram lemma provides existence
and uniqueness of the solution $\phi\in H^{-1/2}(\Gamma_{})^D$ of the equivalent variational formulation of~\eqref{eq:strong}
\begin{align}\label{eq:weak}
\sprod{\mathfrak{V}\phi}{\psi}%_{L^2(\Gamma)}
=\sprod{f}{\psi}%_{L^2(\Gamma)}
 \quad\text{for all }\psi\in H^{-1/2}(\Gamma_{})^D.
\end{align}
In particular, we see that $\mathfrak{V}:H^{-1/2}(\Gamma)^D\to H^{1/2}(\Gamma)^D$ is an isomorphism.
In the Galerkin boundary element method, the test
space $H^{-1/2}(\Gamma_{})^D$ is replaced by some discrete subspace  $\XX_\bullet\subset {L^{2}(\Gamma_{})}^D\subset H^{-1/2}(\Gamma_{})^D$.
Again, the Lax--Milgram lemma guarantees existence and uniqueness of the solution
$\Phi_\bullet\in\XX_\bullet$ of the discrete variational formulation
\begin{align}\label{eq:discrete}
\sprod{\mathfrak{V}\Phi_\bullet}{\Psi_\bullet}%_{L^2(\Gamma)}
 = \sprod{f}{\Psi_\bullet}%_{L^2(\Gamma)}
 \quad\text{for all }\Psi_\bullet\in\XX_\bullet,
\end{align}
and $\Phi_\bullet$ can in fact be computed by solving a linear system of equations.
\revision{We note the Galerkin orthogonality
\begin{align}\label{eq:galerkin bem}
 \sprod{\mathfrak{V}(\phi-\Phi_\coarse)}{\Psi_\coarse} = 0
 \quad\text{for all }\Psi_\coarse\in\XX_\coarse,
\end{align}
along with the resulting C\'ea-type quasi-optimality
\begin{align}\label{eq:cea bem}
 \norm{\phi-\Phi_\coarse}{H^{-1/2}(\Gamma)}
 \le C_{\text{C\'ea}}\min_{\Psi_\coarse\in\XX_\coarse}\norm{\phi-\Psi_\coarse}{H^{-1/2}(\Gamma)}\quad
 \text{with}\quad
 C_{\text{C\'ea}} := \textstyle\frac{\const{cont}}{\const{ell}}.
\end{align}}

Note that \eqref{eq:single layer operator} implies that $\mathfrak{V} \Psi_\coarse\in H^1(\Gamma)^D$ for arbitrary $\Psi_\coarse\in \XX_\coarse$.
The additional regularity $f\in H^1(\Gamma)^D$ instead of $f\in H^{1/2}(\Gamma)^D$ is only needed to define the residual error estimator~\eqref{eq:eta bem} below.
For a more detailed introduction to boundary integral equations, the reader is referred to the monographs \cite{mclean,ss,s}.

%% file: 04_definition_high.tex
% !TEX encoding = MacOSRoman
% !TEX root = igabem3d.tex

\section{Boundary element method with hierarchical splines}\label{sec:hierarchical bem}

In this section, we consider $\Omega\subset\R^d$ with $d\ge 2$.
We introduce hierarchical splines on the boundary  $\Gamma$ and propose a local mesh-refinement strategy.
To this end, we  assume the existence of a mesh $\set{\Gamma_m}{m = 1,\dots,M}$ of $\Gamma$, \revision{i.e., $\Gamma=\bigcup_{m=1}^M \Gamma_m$ with interfaces $\Gamma_m\cap\Gamma_{m'}$ having $(d-1)$-dimensional measure zero for $m\neq m'$}, such that each surface $\Gamma_m$ can be parametrized over   $\widehat\Gamma_m:=[0,1]^{d-1}$.
%We use the notation from Section~\ref{sec:hierarchical splines} (with an additional index $m$ for the surface $\Gamma_m$), where we have already introduced hierarchical splines in the \textit{parameter domain} $\widehat\Gamma_m$.
\revision{We formulate an adaptive algorithm (Algorithm~\ref{alg:bem algorithm}) for conforming BEM discretizations of our model problem~\eqref{eq:strong}, where adaptivity is driven by the \textit{residual {\sl a~posteriori} error estimator} (see \eqref{eq:eta bem} below).}
The main result of this section is Theorem~\ref{thm:main bem}, which states that \revision{the employed estimator is reliable and converges linearly at optimal rate.}
%hierarchical splines  together with the proposed mesh-refinement strategy fit into the abstract setting of Section~\ref{sec:abstract setting bem} and are hence covered by Theorem~\ref{thm:abstract bem}.
The proof of Theorem~\ref{thm:main bem} is given in \revision{Section~\ref{sec:abstract setting bem}}.

\subsection{Parametrization of the boundary}
\label{subsec:gamma bem}
We assume that for all $m\in\{1,\dots,M\}$, the surface $\Gamma_m\revision{\subseteq\Gamma}$ can be parametrized via a bi-Lipschitz mapping 
\begin{align}\gamma_m:{\widehat{\Gamma}_m}\to\Gamma_m, \end{align} where $\widehat\Gamma_m=[0,1]^{d-1}$.
In particular, 
%Lemma~\ref{lem:bi-Lipschitz} (applied on the interior of $\widehat\Gamma_m$) shows that 
$\gamma_m$ is almost everywhere differentiable, and there exists a constant $C_\gamma\ge1$ such that 
\begin{subequations}\label{eq:gram bound}
\begin{align}
C_\gamma^{-1}|s-t|\le|\gamma_m(s)-\gamma_m(t)|\le C_\gamma |s-t|\quad\text{for all }s,t\in\widehat\Gamma_m,
\end{align}
and the Gram determinant satisfies  that
\begin{align}
C_\gamma^{-(d-1)}\le\sqrt{\det(D\gamma_m^\top(t)D\gamma_m(t))}\le C_\gamma^{d-1}\quad\text{for almost all }t\in\widehat\Gamma_m;
\end{align}
\end{subequations}
see, e.g., \cite[Lemma~5.2.1]{diss}.
We define  the set of nodes 
\begin{align}\label{eq:Ngamma}
\NN_\gamma:=\bigcup_{m=1}^M\set{\gamma_m(\widehat z)}{\widehat z\in \{0,1\}^{d-1}}.
\end{align}
We suppose that there are no (initial) hanging nodes, i.e., the intersection $\Gamma_m\cap\Gamma_{m'}$ with $m\neq m'$ is either empty or a common (transformed) lower-dimensional hyperrectangle $\gamma_m(\widehat E_m)=\gamma_{m'}(\widehat E_{m'})$, where $\widehat E_m=[0,1]^{d-1}\cap (v+V)$ and $\widehat E_{m'} = [0,1]^{d-1}\cap(v'+V')$ with $v,v'\in\{0,1\}^{d-1}$ and $V,V'$ the linear spans of at most $d-2$ of the unit vectors $e_1,\dots e_{d-1}\in \R^{d-1}$. 
Moreover, with the node patch $\pi_\gamma(z):=\bigcup\set{\Gamma_m}{z\in\Gamma_m \wedge m = 1,\dots,M }$ for $z\in\NN_\gamma$,  we suppose the following compatibility assumption for the different parametrizations:
For all nodes $z\in\NN_\gamma$, there exists a polytope
 $\overline\pi_\gamma(z)\subset\R^{d-1}$, i.e., an interval for $d=2$, a polygon for $d=3$, and a polyhedron for $d=4$,  and a bi-Lipschitz mapping 
\begin{align}
\gamma_z:\overline\pi_\gamma(z)\to\pi_\gamma(z)
\end{align}
 such that 
$\gamma_z^{-1}\circ\gamma_m$  is an affine bijection for all $m\in\{1,\dots,M\}$ with $\Gamma_m\subseteq\pi_\gamma(z)$.
Put into words, this means that each node patch $\pi_\gamma(z)$ can be flattened, where the corresponding bi-Lipschitz mapping restricted to any contained surface $\Gamma_m$ essentially coincides with the parametrization $\gamma_m$.
In particular, this prohibits the case $\pi_\gamma(z)=\Gamma$. 
The same assumption is also made in \cite[Assumption~4.3.25]{ss} for curvilinear triangulations.
It particularly implies that the parametrizations essentially coincide at the boundary of the surfaces, i.e., for all $m\neq m'$ with non-empty intersection $E:=\Gamma_m\cap \Gamma_{m'}\neq \emptyset$, it holds that
\begin{align}\label{eq:gammacomp}
\gamma_m|_{\widehat E_m}=\gamma_{m'}\circ\big(\gamma_{m'}^{-1}\circ\gamma_z\circ\gamma_z^{-1}\circ\gamma_m|_{\widehat E_m}\big)
\quad \text{with}\quad \widehat E_m:=\gamma_m^{-1}(E) \text{ and } \widehat E_{m'}:=\gamma_{m'}^{-1}(E),
\end{align}
where $\gamma_{m'}^{-1}\circ\gamma_z\circ\gamma_z^{-1}\circ\gamma_m|_{\widehat E_m}:\widehat E_m\to \widehat E_{m'}$ is an affine bijection.
By possibly enlarging $\const{\gamma}$ from \eqref{eq:gram bound}, we can assume that 
\begin{align}
{C}_{\gamma}^{-1}|s-t|\le|\gamma_z(s)-\gamma_z(t)|\le {C}_{\gamma} |s-t|\quad\text{for all }s,t\in\overline\pi_\gamma(z).
\end{align}

\begin{figure}[t] 
\psfrag{x1}[c][c]{\fontsize{5pt}{6pt}\selectfont $x_1$}
\psfrag{x2}[c][c]{\fontsize{5pt}{6pt}\selectfont $x_2$}
\psfrag{x3}[c][c]{\fontsize{5pt}{6pt}\selectfont $x_3$}
\psfrag{y1}[c][c]{\fontsize{5pt}{6pt}\selectfont $\overline x_1$}
\psfrag{y2}[c][c]{\fontsize{5pt}{6pt}\selectfont $\overline x_2$}
\psfrag{z1}[c][c]{\fontsize{5pt}{6pt}\selectfont $\widehat x_1$}
\psfrag{z2}[c][c]{\fontsize{5pt}{6pt}\selectfont $\widehat x_2$}

\begin{center}
\includegraphics[width=0.3\textwidth, clip=true]{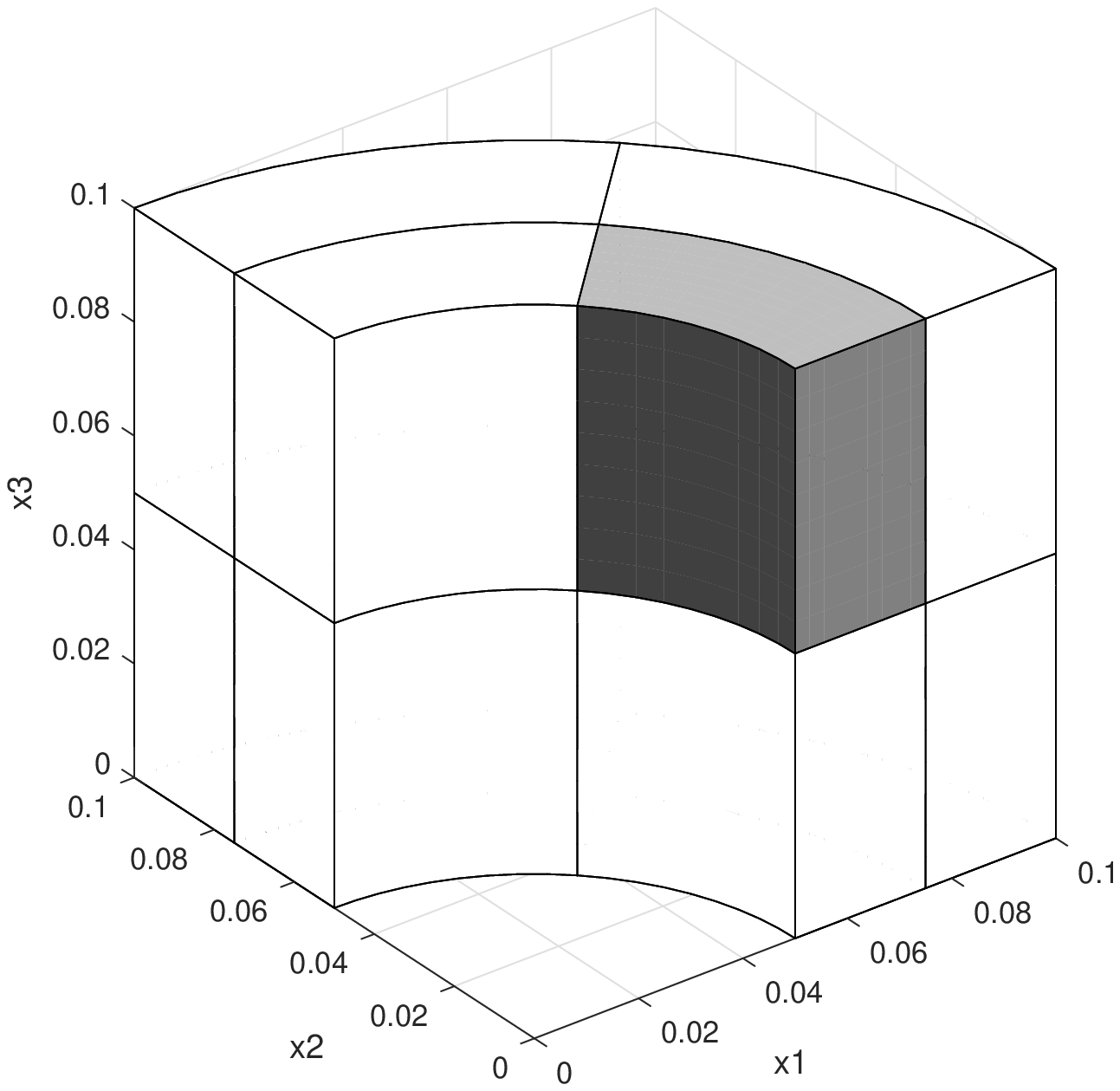}\quad
\includegraphics[width=0.3\textwidth, clip=true]{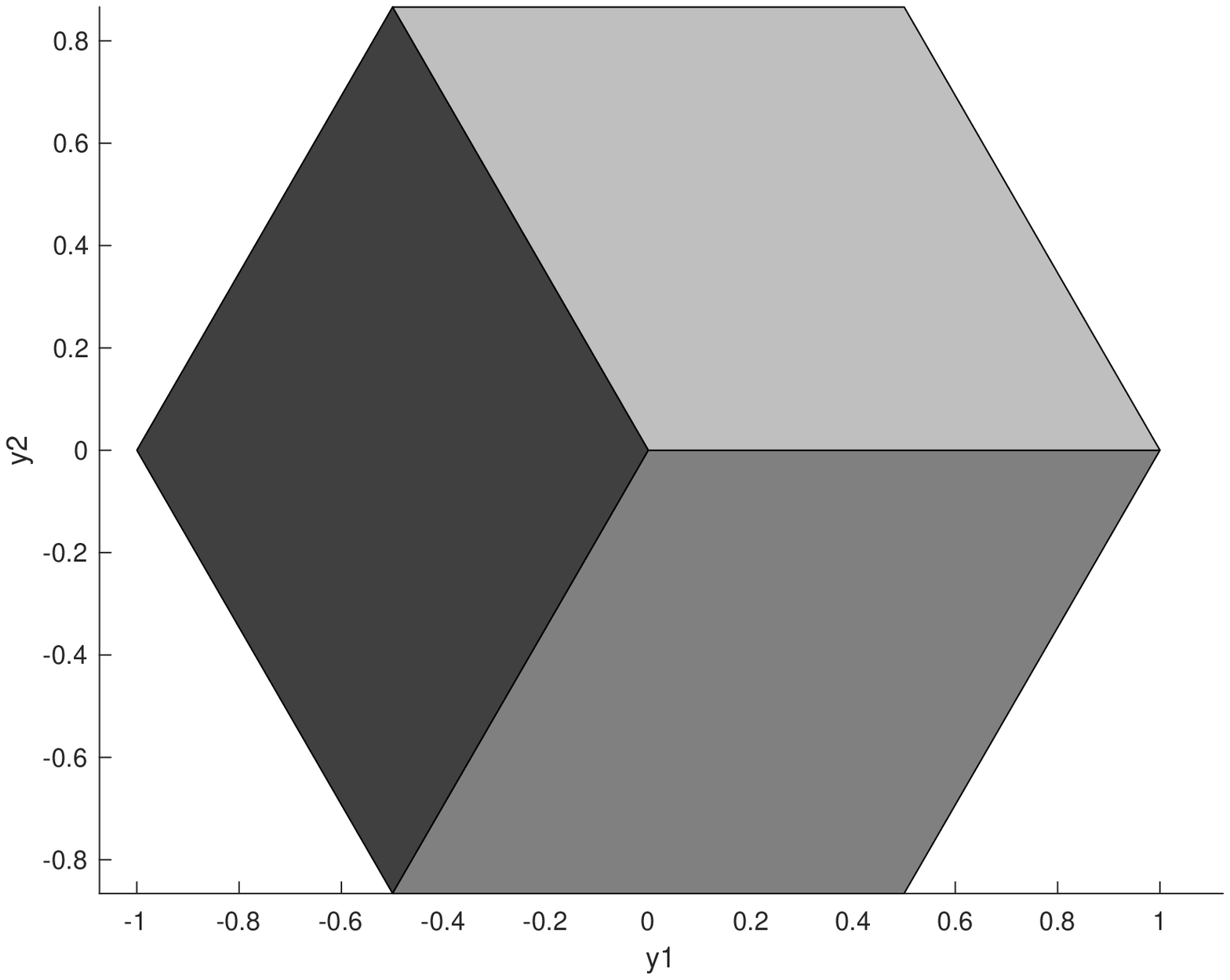}\quad
\includegraphics[width=0.25\textwidth, clip=true]{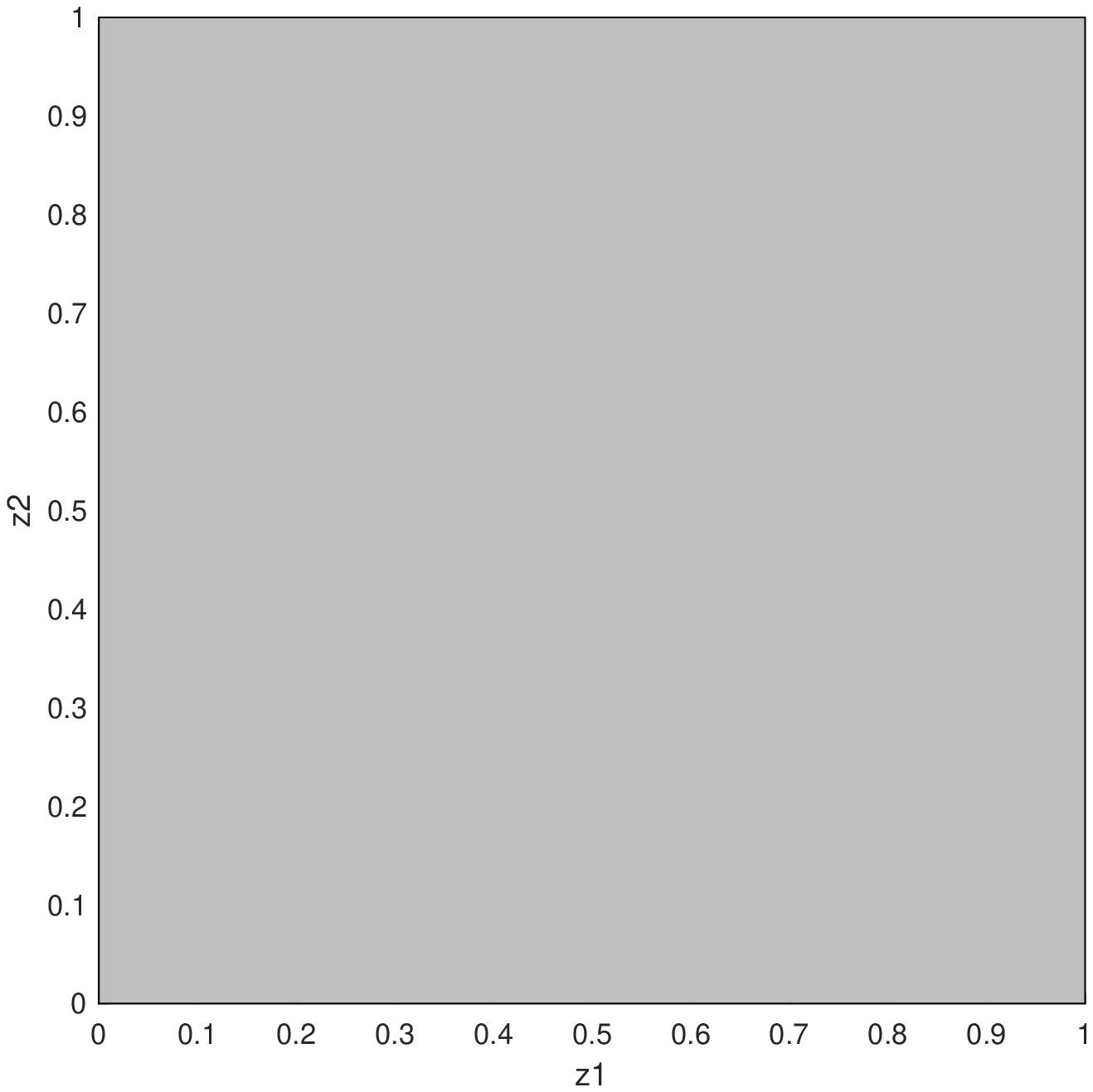}
\end{center}
\caption{\label{fig:patch}
A mesh $\set{\Gamma_m}{m = 1,\dots,24}$ of some boundary $\Gamma$ is depicted, where the patch $\pi_\gamma(z)$ for $z=(1/20,1/10)$ is highlighted in gray (left). 
The corresponding patch $\overline\pi_\gamma(z)$ is depicted (middle). 
The mapping $\gamma_z$ maps the light, medium, and dark gray affine elements to the respective elements on $\Gamma$.
The parameter domain $\widehat \Gamma_m$ of the light gray element is depicted (right). 
The mapping $\gamma_z^{-1}\circ\gamma_m$ maps $\widehat\Gamma_m$ affinely onto the light gray element in $\overline\pi_\gamma(z)$.
}
\end{figure}

\subsection{Uniformly refined tensor meshes and tensor-product splines}
For $m\in\{1,\dots,M\}$, let $(p_{1,m},\dots,p_{d-1,m})$  be a vector of fixed  polynomial degrees in $\N$, and set 
\begin{align}
p_{\max}:=\max\set{p_{i,m}}{i\in\{1,\dots,d\}\revision{, }  m\in\{1,\dots,M\}}.
\end{align}
Let 
\begin{align}{\widehat\KK}_{0,m}=(\widehat \KK_{1(0,m)},\dots,\widehat\KK_{{(d-1)}(0,m)})\end{align} be a fixed initial $(d-1)$-dimensional vector of \textit{$p_{i,m}$-open knot vectors}, i.e., 
\begin{align}
\widehat\KK_{i(0,m)}=(t_{i(0,m),j})_{j=0}^{N_{i(0,m)}+p_{i,m}}
\end{align}
for (in $j$) monotonously increasing knots $t_{i(0,m),j}\in[0,1]$ such that $0=t_{i(0,m),0}=\dots=t_{i(0,m),p_{i,m}}<t_{i(0,m),p_{i,m}+1}$ and $t_{i(0,m),N_{i,m}-1}<t_{i(0,m),N_{i,m}}=\dots=t_{i(0,m),N_{i,m}+p_{i,m}}=1$.
Additionally, we suppose for the multiplicity \revision{(within the knot vector $\widehat\KK_{i(0,m)}$)} of all interior knots $t_{i(0,m),j}\in(0,1)$  that
\begin{align}
\#_{i(0,m)} t_{i(0,m),j}\le p_{i,m}\quad\text{for all }i\in\{1,\dots,d-1\} \text{ and }  j\in\{2+p_{i,m},\dots,N_{i(0,m)}-1\}.
\end{align}
For each $i\in\{1,\dots,d-1\}$, this yields the corresponding \textit{mesh on $[0,1]$}
\begin{align}
\widehat\TT_{i(0,m)}:=\set{[t_{i(0,m),j},t_{i(0,m),j+1}]}{j\in\{0\dots,N_{i(0,m)}-1\}\wedge t_{i(0,m),j}<t_{i(0,m),j+1}}
\end{align}
and the corresponding \textit{splines on $[0,1]$}
\begin{align}
\begin{split}
\widehat \SS^{p_{i,m}}(\widehat\KK_{i(0,m)}):=\big\{\widehat S:[0,1]\to\C:\,\widehat S|_{\widehat T} \text{ is a polynomial of degree } p_{i,m} \text{ for }\widehat T\in\widehat \TT_{i(0,m)}\\
\text{and }\widehat S \text{ is }(\, p_{i,m}-\#_{i(0,m)} t_{i(0,m),j})\text{\revision{-times continuously differentiable}}\\ 
\text{ at }t_{i(0,m),j}\text{ for }j\in\{2+p_{i,m},\dots,N_{i(0,m)}-1\big\}.
\end{split}
\end{align}
The so-called \textit{B-splines} %, which are defined and analyzed, e.g., in \cite{boor},  
form a basis of the latter space
\begin{align}
\widehat \SS^{p_{i,m}}(\widehat\KK_{i(0,m)})=\mathrm{span}(\widehat \BB_{i(0,m)})\quad\text{with}\quad \widehat \BB_{i(0,m)}=\set{\widehat B_{i(0,m),j}}{j\in\{1,\dots,N_{(i(0,m)}\}};
\end{align}
\revision{see Figure~\ref{fig:bsplines} for an illustrative example.}
\revision{For convenience of the reader, we recall the following definition of B-splines via divided differences\footnote{\revision{For any (sufficiently smooth) function $F:\R\to\R$ and arbitrary points $x_0\le x_1\le\dots\le x_{p+1}$, divided differences are recursively defined via $[x_0]F:=F(x_0)$ and $[x_0,\dots,x_{j+1}]F:=\big([x_1,\dots,x_{j+1}]F - [x_0,\dots,x_j]F\big) / (x_{j+1}-x_0)$ if $x_0<x_{j+1}$ and $[x_0,\dots,x_{j+1}]F:=F^{(j+1)}(x_0)/(j+1)!$ else for $j=0,\dots,p$.}},
\begin{align*}
\widehat B_{i(0,m),j}(t) := (t_{i(0,m),j+p_{i,m}} - t_{i(0,m),j-1})  [t_{i(0,m),j-1},\dots,t_{i(0,m),j+p_{i,m}})] \big(\max\{(\cdot)-t,0\}^p\big);% \quad\text{for all }t\in\R;
\end{align*}
see also \cite{boor} for equivalent definitions and elementary properties.}
We only mention that each B-spline $\widehat B_{i(0,m),j}$ is positive on $(t_{i(0,m),j-1},t_{i(0,m),j+p_{i,m}})$ and vanishes outside of $[t_{i(0,m),j-1},t_{i(0,m),j+p_{i,m}}]$.
Finally, we define the \textit{tensor mesh on $\widehat\Gamma_m=[0,1]^{d-1}$}
\begin{align}
\widehat\TT_{0,m}:=\set{\widehat T_1\times\dots\times\widehat T_{d-1}}{\widehat T_{i}\in\widehat\TT_{i(0,m)}\text{ for }i\in\{1,\dots,d-1\}}
\end{align}
and the \textit{tensor-product splines} 
\begin{align}
\widehat \SS^{(p_{1,m},\dots,p_{d-1,m})}(\widehat\KK_{0,m}):=\set{\widehat S_1\otimes\dots\otimes\widehat S_{d-1}}{\widehat S_{i}\in\widehat \SS^{p_{i,m}}(\widehat\KK_{i(0,m)})\text{ for }i\in\{1,\dots,d-1\}}.
\end{align}
A corresponding basis is given by the set of \textit{tensor-product B-splines}
\begin{align}\label{eq:tensor spline basis}
\widehat \SS^{(p_{1,m},\dots,p_{d-1,m})}(\widehat\KK_{0,m})=\mathrm{span}(\widehat \BB_{0,m})\quad\text{with}\quad \widehat \BB_{0,m}:=\big\{&\widehat\beta_1\otimes\dots\otimes\widehat\beta_{d-1}:\\
&\,\widehat\beta_i\in\widehat \BB_{i(0,m)}\text{ for }i\in\{1,\dots,d-1\}\big\}.\notag
\end{align}

\begin{figure}[t] 
\begin{center}
\includegraphics[width=0.4\textwidth]{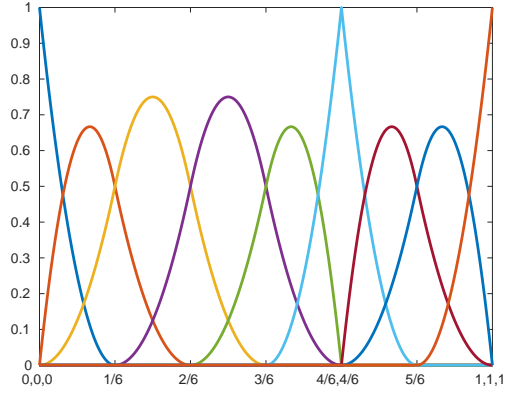}
\end{center}
\caption{\label{fig:bsplines}
\revision{The B-splines $\widehat\BB_{i(0,m)}$ for the polynomial degree $p_{i,m} = 2$ and the $p_{i,m}$-open knot vector $\widehat\KK_{i(0,m)}=(0, 0, 0, 1/6,2/6,3/6,4/6,4/6,5/6, 1, 1, 1)$ are depicted.}
}
\end{figure}

For $i\in\{1,\dots,d-1\}$, we set $\widehat\KK_{i(\uni{0},m)}:=\widehat\KK_{i(0,m)}$ and  recursively define  $\widehat\KK_{i(\uni{k+1},m)}$ for $k\in\N_0$ as the uniform $h$-refinement of 
\begin{align}
\widehat\KK_{i(\uni{k},m)}=(t_{i(\uni{k},m),j})_{j=0}^{N_{i(\uni{k},m)}+p_{i,m}},
\end{align} i.e., it is obtained by inserting the midpoint of any non-empty knot span $[t_{i(\uni{k},m),j},t_{i(\uni{k},m),j+1}]$ with multiplicity one to the knots $\widehat\KK_{i(\uni{k},m)}$. 
\revision{In particular, the old knots $\widehat\KK_{i(\uni{k},m)}$ are not changed in $\widehat\KK_{i(\uni{k+1},m)}$.}
Replacing the index $0$ by $\uni{k}$, one can define the corresonding meshes, splines, and B-splines $\widehat\TT_{i(\uni{k},m)}$, $\widehat \SS^{p_{i,m}}(\widehat\KK_{i(\uni{k},m)})$, $\widehat \BB_{i(\uni{k},m)}$, $\widehat\TT_{\uni{k},m}$, $\widehat \SS^{(p_{1,m},\dots,p_{d-1,m})}(\widehat\KK_{\uni{k},m})$, $\widehat \BB_{\uni{k},m}$ as before. 
This yields a  nested sequence of tensor-product spline spaces 
\begin{align}
\widehat\SS^{(p_1,\dots,p_d)}(\widehat\KK_{\uni{k},m})\subset\widehat\SS^{(p_1,\dots,p_d)}(\widehat\KK_{\uni{k+1},m})\subset C^0([0,1)^{d-1}),
\end{align}
where the last relation follows from the assumption for the multiplicity of the interior knots.

\begin{remark}\label{eq:different uni}
We define the  nested sequence of  tensor-product spline spaces in the most simple way. 
Another natural approach is to use uniform $h$-refinement with knots of some fixed  multiplicities $1\le q_{i,m}\le p_{i,m}$ for $i\in\{1,\dots,d-1\}$ as in \cite[Remark~3.4.1]{diss} instead of multiplicity one.
%With $q_m:=(q_{1,m},\dots,q_{d-1,m})$, we set $\widehat\KK_{\uni{0,q_m}}:=\widehat\KK_0$, and define $\widehat\KK_{\uni{k+1,q_m}}$ recursively for $k\in\N_0$ as the $(d-1)$-dimensional knot vector, that results from inserting the knot $(\widehat z_{i(\uni{k,q_m}),j-1}+\widehat z_{i(\uni{k,q_m}),j})/{2}$ with multiplicity $q_i$ to the knots $\widehat\KK_{i(\uni{k})}$, where $i\in\{1,\dots,d\}$ and $j\in\{1,\dots, n_{i(\uni{k,q})}\}$.
%If $q_{i_m}=1$ for all $i\in\{1,\dots,d-1\}$, we have that $\widehat\KK_{\uni{k}}=\widehat\KK_{\uni{k,q}}$.
The choice $q_{i,m}=1$ for all $i\in\{1,\dots,d-1\}$, that we consider, leads to the highest possible regularity of the splines at newly inserted mesh lines, whereas the maximal choice $q_{i,m}=p_{i,m}$ only leads to continuity at new mesh lines.
In particular, if all interior knots of the initial knots $\widehat \KK_{i(0,m)}$ already have multiplicity $p_{i,m}$, the latter choice leads to the space of all continuous $\widehat\TT_{\uni{k},m}$-piecewise tensor-product polynomials of degree $(p_{1,m},\dots,p_{d-1,m})$.
%Note that all following definitions of this section can be made similarly if $\uni{k}$ is replaced by $\uni{k,q}$.
%Also the corresponding results remain valid.
\end{remark}

\subsection{Hierarchical  meshes and splines on the parameter domains}\label{subsec:parameter hsplines bem}
Let $m\in\{1,\dots,M\}$.
We say that a set 
\begin{align}
\widehat\TT_{\coarse,m}\subseteq\bigcup_{k\in\N_0}\widehat \TT_{\uni{k},m}
\end{align}
 is a \textit{hierarchical mesh on the parameter domain $\widehat\Gamma_m=[0,1]^{d-1}$} if it is a partition of $[0,1]^{d-1}$ in the sense that $\bigcup\widehat \TT_{\coarse,m}=[0,1]^{d-1}$, where the intersection of  two different elements $\widehat T\neq \widehat T'$ with $\widehat T,\widehat T'\in\widehat\TT_{\coarse,m}$ has ($(d-1)$-dimensional) measure zero.
Since $\widehat\TT_{\uni{k},m}\cap\widehat\TT_{\uni{k'},m}=\emptyset$ for $k,k'\in\N_0$ with $k\neq k'$, we can define for an element $\widehat T\in\widehat\TT_{\coarse,m}$,
\begin{align}
\level(\widehat T):=k\in\N_0 \quad\text{with }\widehat T\in\widehat\TT_{\uni{k},m}.
\end{align}
 For an illustrative example of a hierarchical mesh, see Figure~\ref{fig:hiermesh}.
In particular, any uniformly refined tensor mesh $\widehat\TT_{\uni{k},m}$ with $k\in\N_0$ is a hierarchical mesh.

For a hierarchical mesh $\widehat\TT_{\coarse,m}$, we define a corresponding nested sequence $(\widehat\Omega^k_{\coarse,m})_{k\in\N_0}$ of closed subsets of $[0,1]^{d-1}$ by
\begin{align}\label{eq:omega k}
\widehat\Omega_{\coarse,m}^k:=\bigcup\set{\widehat T\in\widehat\TT_{\coarse,m}}{\level(\widehat T)\ge k}.
\end{align}
It holds that
\begin{align}\label{eq:parameter mesh}
\widehat\TT_{\coarse,m}=\bigcup_{k\in\N_0}\set{\widehat T\in\widehat\TT_{\uni{k},m}}{\widehat T\subseteq \widehat \Omega_{\coarse,m}^k\wedge \widehat T\not\subseteq\widehat\Omega_{\coarse,m}^{k+1}}.
\end{align}
Indeed, in the literature, one usually assumes that one is given the sequence $(\widehat\Omega^k_{\coarse,m})_{k\in\N_0}$ and defines the corresponding hierarchical mesh via \eqref{eq:parameter mesh}.
Note that, for $\widehat T\in\widehat\TT_{\coarse,m}$, $\level(\widehat T)$ is  the unique integer $k\in\N_0$ with $\widehat T\subseteq \widehat \Omega_{\coarse,m}^k$ and $\widehat T\not\subseteq\widehat\Omega_{\coarse,m}^{k+1}$.
With this, we can define the corresponding \textit{hierarchical B-splines on the parameter domain $\widehat\Gamma_m=[0,1]^{d-1}$}
\begin{align}\label{eq:short cHH}
\widehat\BB_{\coarse,m}:=\bigcup_{k\in\N_0}\set{\widehat\beta\in\widehat\BB_{\uni{k},m}}{ \supp(\widehat\beta)\subseteq\widehat\Omega_{\coarse,m}^k\wedge\supp(\widehat\beta)\not\subseteq\widehat\Omega_{\coarse,m}^{k+1}}.
\end{align}
The space of ($D$-dimensional) \textit{hierarchical splines on the parameter domain $\widehat\Gamma_m=[0,1]^{d-1}$} is just defined as the product of the resulting span
\begin{align}
%\widehat\XX_{\coarse,m}
%\widehat\SS^{(p_{1,m},\dots,p_{d-1,m})}(\widehat\KK_{0,m},\widehat\TT_{\coarse,m})
\widehat\XX_{\coarse,m}:=\mathrm{span}(\widehat\BB_{\coarse,m})^D.
\end{align}
It is easy to check that if $\widehat \TT_{\coarse,m}$ is a tensor mesh and hence coincides with some $\widehat\TT_{\uni{k},m}$, then the hierarchical basis and the standard tensor-product B-spline basis are the same.
Thus, the notation is consistent with the notation from \eqref{eq:tensor spline basis}.
One can prove that $\widehat\BB_{\coarse,m}$ is linearly independent; see, e.g., \cite[Lemma~2]{juttler}.
Since $\widehat\BB_{\uni{k}}\cap\widehat\BB_{\uni{k'}}=\emptyset$ for $k\neq k'\in\N_0$, we can define for a basis function $\widehat \beta\in\widehat\BB_{\coarse,m}$
\begin{align}
\level(\widehat \beta):=k\in\N_0 \quad\text{with }\widehat \beta\in\widehat\BB_{\uni{k}}.
\end{align}
Note that $\level(\widehat\beta)$ is  the unique integer $k\in\N_0$ with $\supp(\widehat\beta)\subseteq\widehat\Omega_{\coarse,m}^k$ and $\supp(\widehat\beta)\not\subseteq\widehat\Omega_{\coarse,m}^{k+1}$.

\begin{figure}[t] 
\psfrag{z}[r][r]{z}
\psfrag{T}[r][r]{T}
\begin{center}
\includegraphics[width=0.3\textwidth]{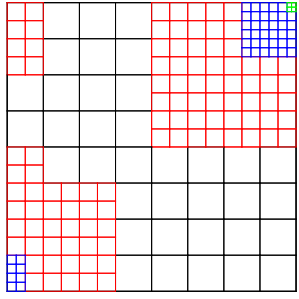}
\end{center}
\caption{\label{fig:hiermesh}
A two-dimensional hierarchical mesh $\widehat\TT_{\coarse,m}$ is depicted, where $\widehat\Omega_{\coarse,m}^4=\emptyset$.
The corresponding  domains $\widehat\Omega^0_{\coarse,m}\supseteq\widehat\Omega^1_{\coarse,m}\supseteq\widehat\Omega^2_{\coarse,m}\supseteq\widehat\Omega^3_{\coarse,m}$ are highlighted in black, red, blue and green.
}
\end{figure}

The hierarchical basis  $\widehat\BB_{\coarse,m}$ and the mesh $\widehat\TT_{\coarse,m}$ are compatible in the sense that 
\begin{align}\label{eq:supp elements}
\supp(\widehat\beta)=\bigcup_{k\ge \level(\widehat\beta)}\set{\widehat T\in\widehat\TT_{\coarse,m}\cap\widehat\TT_{{\rm uni}(k),m}}{\widehat T\subseteq\supp(\widehat\beta)};
\end{align}
see, e.g., \cite[Section~3.2]{igafem} for the elementary argumentation.

Finally, we say that a hierarchical mesh $\widehat\TT_{\fine,m}$ is \textit{finer} than $\widehat\TT_{\coarse,m}$ if $\widehat\TT_{\fine,m}$ is obtained from $\widehat\TT_{\coarse,m}$ via iterative dyadic bisection.
Formally, this can be stated as $\widehat\Omega_{\coarse,m}^k\subseteq\widehat\Omega_{\fine,m}^k$ for all $k\in\N_0$.
In this case, the corresponding hierarchical spline spaces are nested according to \cite[Proposition~4]{juttler}, i.e.,
\begin{align}\label{eq:hierarchical nested pardom bem}
\widehat\XX_{\coarse,m}\subseteq\widehat\XX_{\fine,m}.
\end{align}
For a more detailed introduction to hierarchical meshes and splines, we refer to, e.g., \cite{juttler,juttler2,garau,speleers}.

\subsection{Admissible hierarchical meshes on the parameter domains} 
We recall the definition and some properties of \textit{admissible meshes} introduced in \cite[Section~3.5]{igafem}.
Let $m\in\{1,\dots,M\}$ and $\widehat\TT_{\coarse,m}$ be an arbitrary hierarchical mesh on $\widehat\Gamma_m$.
We define the set of all \textit{neighbors}  of an element $\widehat T\in\widehat\TT_{\coarse,m}$ as
\begin{align}\label{eq:neigbors}
\Nu_{\coarse,m}(\widehat T):=\set{\widehat T'\in\widehat\TT_{\coarse,m}}{\exists\widehat\beta\in\widehat\BB_{\coarse,m}\quad|\widehat T\cap\supp(\widehat\beta)|, |\widehat T'\cap\supp(\widehat \beta)|>0},
\end{align}
The following lemma from \cite[Lemma~5.2]{igafem} provides a relation between the set of neighbors and the patch of an element $\widehat T\in\widehat\TT_{\coarse,m}$.

\begin{lemma}\label{lem:patch2neigbor}
%Let $\widehat\TT_{\coarse,m}$ be an arbitrary hierarchical mesh.
There holds that 
\begin{align}\label{eq:patch2neigbor}
\Pi_{\coarse,m}(\widehat T)\subseteq \Nu_{\coarse,m}(\widehat T)\quad\text{for all } \widehat T\in\widehat\TT_{\coarse,m},
\end{align}
where $\Pi_{\coarse,m}(\widehat T):=\set{\widehat T'}{\widehat T\cap \widehat T\neq \emptyset}$.
\hfill$\square$
\end{lemma}
%According to  \eqref{eq:supp elements},  the condition   $\widehat T,\widehat T'\subseteq\supp(\widehat\beta)$ is equivalent to $|\widehat T\cap\supp(\widehat\beta)|\neq0\neq |\widehat T'\cap\supp(\widehat \beta)|$.
We call  $\widehat\TT_{\coarse,m}$  \textit{admissible} if 
\begin{align}\label{def:admissible}
|\level(\widehat T)-\level(\widehat T')|\le 1\quad\text{for all }\widehat T,\widehat T'\in\widehat \TT_{\coarse,m} \text{ with }\widehat T'\in\Nu_{\coarse,m}(\widehat T).
\end{align}
Let $\widehat\T_m$ be the set of all admissible hierarchical meshes on $\widehat\Gamma_m$.
Clearly, all tensor meshes $\widehat\TT_{\uni{k},m}$, $k\in\N_0$, belong to $\widehat\T_m$.
 Moreover, admissible meshes satisfy the following interesting properties, which are also  important for an efficient implementation of finite or boundary element methods with hierarchical splines; see \cite[Proposition~3.2]{igafem} for a proof.
 
\begin{proposition}\label{prop:bounded number}
For each $\widehat \TT_{\coarse,m}\in\widehat\T_m$, 
the support of any hierarchical  B-spline $\widehat \beta\in\widehat \BB_{\coarse,m}$ is the union  of at most  $2^d (p_{\max}+1)^d$ elements $\widehat T'\in\widehat \TT_{\coarse,m}$.
Moreover, for any $\widehat T\in\widehat \TT_{\coarse,m}$, there are at most $2(p_{\max}+1)^d$ basis functions $\widehat \beta'\in\widehat \BB_{\coarse,m}$ with support on $\widehat T$, i.e., $|\supp(\widehat \beta')\cap \widehat T|>0$.\hfill$\square$
\end{proposition}

\begin{remark}\label{rem:connected}
{\rm (a)}
 Since the support of any $\widehat \beta\in\widehat \BB_{\coarse,m}$ is connected, Proposition~\ref{prop:bounded number} particularly shows that $|\widehat T' \cap\supp(\widehat\beta)|>0$ for some element $\widehat T'\in\widehat\TT_{\coarse,m}$ implies that $\supp(\widehat\beta)\subseteq\pi_{\coarse,m}^{2(p_{\max}+1)}(\widehat T')$, \revision{where $\pi_{\coarse,m}^{2(p_{\max}+1)}(\widehat T')$ is recursively defined via 
$\pi_{\bullet,m}^0(\widehat T') := \widehat T'$ and $\pi_{\bullet,m}^q(\widehat T') := \bigcup\set{\widehat  T\in\widehat \TT_{\bullet,m}}{ \widehat{T}\cap \pi_{\bullet,m}^{q-1}(\widehat T')\neq \emptyset}$ for $q>0$.}
 
{\rm (b)}
\cite{bg} studied  related  admissible hierarchical meshes. 
Whereas our admissibility is chosen such that Proposition~\ref{prop:bounded number} holds for hierarchical B-splines, their weaker definition of admissibility  only guarantees Proposition~\ref{prop:bounded number} for \emph{truncated hierarchical B-splines}; see \cite[Remark~3.2]{igafem}.
%Truncated hierarchical B-splines have a smaller  but also more complicated support than the corresponding hierarchical B-splines.
%Indeed, the support is not even necessarily connected, which is why the implication from {\rm (a)}, which is an important ingredient for the verification of the properties \eqref{S:local bem}--\eqref{S:stab bem},  is not clear on admissible meshes in the sense of \cite{bg}.
We also note that our admissibility guarantees Proposition~\ref{prop:bounded number} for both standard and truncated hierarchical B-splines.
This follows immediately from \eqref{eq:bounds for Trunc} below.
Moreover, the computation of the truncated hierarchical B-splines simplifies greatly on our admissible meshes; see \cite[Proposition~5.2]{igafem}.
\end{remark}

\begin{remark}\label{rem:Nu for p0}
For $\widehat\TT_{\coarse,m}\in\widehat\T_m$, Lemma~\ref{lem:patch2neigbor} immediately yields that 
\begin{align}\label{def:admissible2}
|\level(\widehat T)-\level(\widehat T')|\le 1\quad\text{for all }\widehat T,\widehat T'\in\widehat \TT_\coarse \text{ with }\widehat T'\in\Pi_{\coarse,m}(\widehat T).
\end{align}
We note that in the proof of Lemma~\ref{lem:patch2neigbor}, one exploits that all interior knot multiplicities are smaller or equal than the corresponding polynomial degree $p_{i,m}$. 
Actually, this is the only place, where we need this assumption, since we do not require continuous ansatz functions to discretize $H^{-1/2}(\Gamma)^D$.
However, this lemma is of course essential as it implies for example local quasi-uniformity \eqref{def:admissible2} for admissible meshes. 
If one drops the additional assumption on the knot multiplicities and allows multiplicities up to $p_{i,m}+1$ as well as lowest-order polynomial degrees $p_{i,m}=0$, one could define the neighbors of an element $\widehat T$ in an arbitrary hierarchical mesh $\widehat\TT_{\coarse,m}$ differently as
\begin{align}
\Nu_{\coarse,m}(\widehat T):=\set{\widehat T'\in\widehat\TT_{\coarse,m}}{\big(\exists\widehat\beta\in\widehat\BB_{\coarse,m}\quad \widehat T,\widehat T'\subseteq\supp(\widehat\beta)\big)\vee \big(\widehat T\cap\widehat T'\neq \emptyset\big)}
\end{align}
and adapt the definition of admissibility.
Then, Lemma~\ref{lem:patch2neigbor} and Proposition~\ref{prop:bounded number} remain valid.
Moreover, newly inserted knots can also have multiplicity $p_{i,m}+1$, i.e., the choice $q_{i,m}=p_{i,m}+1$ in Remark~\ref{eq:different uni} is possible.
\end{remark}

%The next proposition shows that for an admissible mesh $\widehat\TT_\coarse\in\widehat\TT_\coarse$, the full truncation $\Trunc_\coarse$ reduces to  $\trunc_\coarse^{\level(\widehat\beta)+1}$.
%\begin{proposition}\label{prop:trunc}
%Let $\widehat\TT_\coarse\in\widehat\T$ and $\widehat\beta\in\widehat\BB_\coarse$.
%Then, it holds that
%\begin{align}\label{eq:trunc one}
%\Trunc_\coarse(\widehat\beta)=\trunc^{\level(\widehat\beta)+1}_\coarse(\widehat\beta).
%\end{align}
%\end{proposition}

\subsection{(Admissible) hierarchical  meshes and splines on the boundary}\label{subsec:physical hsplines bem}
In order to transform the definitions from the parameter domain $\widehat\Gamma_m$ to the boundary part $\Gamma_m$, we use the parametrizations from Section~\ref{subsec:gamma bem}.  
All previous definitions can now also be made on each part $\Gamma_m$, just by pulling them from the parameter domain via the bi-Lipschitz mapping $\gamma_m$.
For these definitions, we drop the symbol~$\widehat\cdot$.
 If $\widehat\TT_{\coarse,m}$ is a hierarchical mesh in the parameter domain $\widehat\Gamma_m$,  we  define the corresponding mesh on $\Gamma_m$ as  $\TT_{\coarse,m}:=\set{\gamma_m(\widehat T)}{\widehat T\in\widehat\TT_{\coarse,m}}$.
In particular, we have that $\TT_{0,m}=\set{\gamma_m(\widehat T)}{\widehat T\in\widehat\TT_{0,m}}$.
Moreover, let  $\T_m:=\set{\TT_{\coarse,m}}{\widehat\TT_{\coarse,m}\in\widehat\T_m}$ denote the  set of  all admissible meshes on $\Gamma_m$, where $\widehat\T_m$ is the set of all admissible hierarchical meshes on $\widehat\Gamma_m$. % in the sense of Section~\ref{sec:first admissible}.
For an arbitrary hierarchical mesh $\TT_{\coarse,m}$ on $\Gamma_m$, 
we introduce the corresponding hierarchical spline space 
\begin{align}
\XX_{\coarse,m}:=\set{ \widehat\Psi_{\coarse,m}\circ\gamma_m^{-1}}{ \widehat\Psi_{\coarse,m}\in\widehat\XX_{\coarse,m}}\subset C^0(\Gamma_m)^D\subset L^2(\Gamma_m)^D\subset H^{-1/2}(\Gamma_m)^D
\end{align}

Finally, it remains to define hierarchical meshes and splines on $\Gamma$ itself.
This can be done by gluing the previous definitions for the single surfaces $\Gamma_m$ together. 
For all $m\in\{1,\dots,M\}$, let $\TT_{\coarse,m}$ be a hierarchical mesh on $\Gamma_m$. 
We define the corresponding hierarchical mesh on $\Gamma$ by
$\TT_\coarse:=\bigcup_{m=1}^M \TT_{\coarse,m}$. 
%Clearly, $\TT_\coarse$ is a mesh in the sense of Section~\ref{subsec:boundary discrete bem}, where 
%\revision{We abbreviate}
%\begin{align}
%\widehat T:=\gamma_m^{-1}(T)\quad \text{and}\quad \gamma_T:=\gamma_m|_{\widehat T} \quad \text{for } T\in\TT_{\coarse,m} \text{ with }m\in\{1,\dots,M\}. 
%\end{align} 
%and we can use the notation from there.
For $T\in\TT_{\coarse,\revision{m}}$, we denote 
$\level(T):=\level(\revision{\gamma_m^{-1}(T)})$.
We call $\TT_\coarse$ \textit{admissible} if the mesh satisfies the following two properties:
\begin{itemize}
\item
All partial meshes are admissible, i.e.,  $\TT_{\coarse,m}\in\T_m$ for all $m\in\{1,\dots,M\}$.
\item 
There are no hanging nodes on the boundary  of the surfaces $\Gamma_m$, i.e., the intersection $T\cap T'$ for  $T\in\TT_{\coarse,m},T'\in\TT_{\coarse,m'}$ with $m\neq m'$ is either empty or a common (transformed) lower-dimensional hyperrectangle.
\end{itemize}
We define the set of all admissible hierarchical meshes on $\Gamma$ as $\T$, and suppose that the initial mesh on $\Gamma$ is admissible, i.e.,  
\begin{align}\TT_0=\bigcup_{m=1}^M \TT_{0,m} \in\T.\end{align}
For an arbitrary hierarchical mesh $\TT_\coarse$ on $\Gamma$, the corresponding hierarchical splines read
\begin{align}
\XX_\coarse:=\set{\Psi_{\coarse}: \Gamma\to \C^D}{\Psi_\coarse|_{\Gamma_m}\in \XX_{\coarse,m}\text{ for all }m\in\{1,\dots,M\}}\subset L^2(\Gamma)^D\subset H^{-1/2}(\Gamma)^D.
\end{align}

\begin{remark}\label{rem:admissible bem}
{\rm (a)}
The property that there are no hanging nodes implies local quasi-uniformity at the boundaries $\partial\Gamma_m$, i.e., if $T\in\TT_{\coarse,m},T'\in\TT_{\coarse,m'}$ with $m\neq m'$ have non-empty intersection $T\cap T'\neq\emptyset$, then $|\level( T)-\level( T')|\le 1$. 
Indeed, the intersection $T\cap T'$ of $T\in\TT_{\coarse,m},T'\in\TT_{\coarse,m'}$ is either empty or a common (transformed) lower-dimensional hyperrectangle. 
Thus, if ${\rm dim}(T\cap T')\ge 1$, then $\level( T)=\level( T')$.
If ${\rm dim}(T\cap T')= 0$, i.e., if $T\cap T'$ is only a point, there exists a sequence of elements $T_1\in\TT_{\coarse,m_1},\dots, T_J\in \TT_{\coarse,m_J}$ with $T_1=T$ and $T_J=T'$ such that $m_j\neq m_{j+1}$ and ${\rm dim}(T_j\cap T_{j+1}) \ge 1$ for all $j\in \{1,\dots,J-1\}$.
The previous argumentation yields that $\level( T)=\level( T')$.

{\rm (b)}
Since, the ansatz space  only needs to be a subset of $H^{-1/2}(\Gamma)^D$, the  property that there are no hanging nodes can actually be replaced by local quasi-uniformity at the boundaries $\partial\Gamma_m$ as in {\rm (a)}, which is sufficient for the following analysis.
However, for the {\rm hyper-singular integral equation} which appears for the Neumann problem $\mathfrak{P} u=0$ in $\Omega$ with conormal derivative $\mathfrak{D}_\nu u=\phi$ on $\Gamma$ for some given $\phi\in H^{-1/2}(\Gamma)^D$ (see, e.g., \cite[pages 229--231]{mclean}), the ansatz functions must be in $H^{1/2}(\Gamma)^D$.
In this case, the natural choice for the ansatz space is $\XX_\coarse\cap C^0(\Gamma)^D$, which is even a subset of $H^1(\Gamma)^D$.
If one supposes that there  are no hanging nodes, one can define a local basis with the help of \cite[Proposition~3.1]{igafem}, which states that hierarchical B-splines on $[0,1]^{d-1}$ restricted onto any hyperface of $[0,1]^{d-1}$ are hierarchical B-splines on $[0,1]^{d-2}$.
The knowledge of such a basis is not only essential for an efficient implementation, but is also needed, e.g., for the definition of a quasi-interpolation  operator. 
\end{remark}

\subsection{Refinement  of hierarchical meshes}\label{subsec:concrete refinement bem}
In this section, we present a concrete refinement algorithm to specify the setting of Section~\ref{subsec:general refinement bem}.
We start in the parameter domain.
Recall that a hierarchical mesh $\widehat\TT_{\circ,m}$ is finer than another hierarchical mesh $\widehat\TT_{\coarse,m}$ 
 if  $\widehat\Omega_{\coarse,m}^k\subseteq \widehat\Omega_{\circ,m}^k$ for all $k\in\N_0$.
Then, $\widehat\TT_{\circ,m}$ is obtained from $\widehat\TT_{\coarse,m}$ by iterative dyadic bisections of the elements in $\widehat\TT_{\coarse,m}$.
To bisect an element $\widehat T\in\widehat\TT_{\coarse,m}$, one just has to add it to the set $\widehat\Omega_{\coarse,m}^{\level(\widehat T)+1}$; see \eqref{eq:bisection bem} below.
In this case, the corresponding spaces are nested according to  \eqref{eq:hierarchical nested pardom bem}.
To transfer this definition onto the surface $\Gamma_m$ for $m\in\{1,\dots,M\}$, we essentially just drop the symbol $\widehat\cdot$.
We say that a hierarchical mesh $\TT_{\fine,m}$ on $\Gamma_m$ is \textit{finer} than another hierarchical mesh $\TT_{\coarse,m}$ on $\Gamma_m$, if the corresponding meshes in the parameter domain satisfy this relation, i.e., if $\widehat\TT_{\fine,m}$ is finer than  $\widehat\TT_{\coarse,m}$.
In this case, it holds that
\begin{align}\label{eq:hierarchical nested pydom bem}
\XX_{\coarse,m}\subseteq\XX_{\fine,m}.
\end{align}
Finally, we call a hierarchical mesh $\TT_\fine$ on $\Gamma$ \textit{finer} than another  hierarchical mesh $\TT_\coarse$ on $\Gamma$, if the corresponding partial meshes satisfy this relation, i.e., if $\TT_{\fine,m}$ is finer than $\TT_{\coarse,m}$ for all $m\in\{1,\dots,M\}$.
In this case, it holds that
\begin{align}\label{eq:nested bem}
\XX_{\coarse}\subseteq\XX_{\fine}.
\end{align}

Let $ \TT_\coarse$ be a hierarchical mesh and $ T\in\TT_\coarse$, i.e., $T\in\TT_{\coarse,m}$ for some $m\in\{1,\dots,M\}$.
Let $\widehat T=\gamma_m^{-1}(T)$ be the corresponding element in the parameter domain.
We define the sets of \textit{neighbors}
\begin{align}
\Nu_{\coarse,m}(T):=\set{\gamma_m(\widehat T')}{\widehat T'\in \Nu_{\coarse,m}(\widehat T)}, 
\end{align}
where 
 $\Nu_{\coarse,m}(\widehat T)=\set{\widehat T'\in\widehat\TT_{\coarse,m}}{\exists\widehat\beta\in\widehat\BB_{\coarse,m}\quad\widehat T,\widehat T'\subseteq\supp(\widehat\beta)}$ is the set from \eqref{eq:neigbors}, and
\begin{align}\label{eq:neighbors bem}
\Nu_\coarse(T):=\Nu_{\coarse,m}(T)\cup\bigcup_{m'\neq m}\set{T'\in\TT_{\coarse,m}}{T\cap T'\neq \emptyset}.
\end{align}
Further, we 
 define the sets of \textit{bad neighbors}
 \begin{align}
\Nu_{\coarse,m}^{\rm bad}(T):=\set{\gamma_m(\widehat T')}{\widehat T'\in \Nu^{\rm bad}_{\coarse,m}(\widehat T)}, 
\end{align}
where $\Nu_{\coarse,m}^{\rm bad}(\widehat T):=\set{\widehat T'\in\Nu_{\coarse,m}(\widehat T)}{\level(\widehat T')=\level(\widehat T)-1}$, 
and
\begin{align}\label{eq:bad neighbors bem}
\begin{split}
\Nu^{\rm bad}_\coarse( T)&:=\Nu_{\coarse,m}^{\rm bad}(T)\cup \bigcup_{m'\neq m} \set{T'\in\TT_{\coarse,m'}}{{\rm dim}(T\cap T')>0}
\end{split}
\end{align}
 \begin{algorithm}\label{alg:refinement bem}
\textbf{Input:} Hierarchical mesh $\TT_\coarse$ , marked elements $\MM_\coarse=:\MM_\coarse^{(0)}\subseteq\TT_\coarse$.
\begin{enumerate}[\rm (i)]
\item \label{item:refinement 1 bem} Iterate the following steps {\rm (a)--(b)} for $i=0,1,2,\dots,$ until $\UU_\coarse^{(i)}=\emptyset$:
\begin{itemize}
\item[\rm (a)]Define $\UU_\coarse^{(i)}:=\bigcup_{ T\in\MM_\coarse^{(i)}}\set{ T'\in\TT_\coarse\setminus\MM_\coarse^{(i)}}{ T'\in\Nu^{\rm bad}_\coarse(T)}$.
\item[\rm(b)] 
Define  $\MM_\coarse^{(i+1)}:=\MM_\coarse^{(i)}\cup\UU_\coarse^{(i)}$.
\end{itemize}
\item
Dyadically bisect all $ T\in \MM_\coarse^{(i)}$ in the parameter domain by adding the corresponding  $\widehat T\in\widehat\TT_{\coarse,m}$ to the set $\widehat \Omega_{\coarse,m}^{\level(\widehat T)+1}$ and obtain a finer hierarchical mesh $\TT_\circ$, where for all $m\in\{1,\dots,M\}$ and  all $k\in\N$
\begin{align}\label{eq:bisection bem}
\widehat\Omega_{\circ,m}^k=\widehat\Omega_{\coarse,m}^k\cup\bigcup\set{\widehat T\in\widehat\TT_{\coarse,m}}{\gamma_m(\widehat T)\in \MM_{\coarse}^{(i)}\wedge\level(\widehat T)=k-1}.
\end{align}
\end{enumerate}
\textbf{Output:} Refined mesh $\TT_\fine=\refine(\TT_\coarse,\MM_\coarse)$.
\end{algorithm}
Clearly, $\refine(\TT_\coarse,\MM_\coarse)$ is finer than $\TT_\coarse$. 
For any hierarchical mesh $\TT_\coarse$ on $\Gamma$, we define $\refine(\TT_\coarse)$ as the set of all hierarchical meshes $\TT_\circ$ on $\Gamma$ such that there exist hierarchical meshes $\TT_{(0)},\dots,\TT_{(J)}$ and marked elements $\MM_{(0)},\dots,\MM_{(J-1)}$ with $\TT_\circ$ $=\TT_{(J)}$ $=\refine(\TT_{(J-1)},\MM_{(J-1)}),\dots,\TT_{(1)}$ $=\refine(\TT_{(0)},\MM_{(0)})$, and $\TT_{(0)}=\TT_\coarse$. 
Note that $\refine(\TT_\coarse,\emptyset)=\TT_\coarse$, which is why $\TT_\coarse\in\refine(\TT_\coarse)$. 
The following  proposition characterizes the set $\refine(\TT_\coarse)$.
In particular, it shows that $\refine(\TT_0)=\T$.
The proof works as for IGAFEM in \cite[Proposition~5.1]{igafem}, where the assertion is stated for $\Gamma=\Gamma_1\subset\R^{d-1}$; see \cite[Proposition~5.4.3]{diss} for further details. 

\begin{proposition}\label{prop:refineT subset T bem}
If $\TT_\coarse\in\T$, then $\refine(\TT_\coarse)$ coincides with the set of all admissible hierarchical meshes $\TT_\circ$ that are finer than $\TT_\coarse$.
\hfill $\square$
\end{proposition}

\subsection{Error estimator}\label{subsec:estimator bem}
%%%%%%%%%%%%%%%%%%%%%%%%%%%%%%%%%%%%%%%%%%%%%%%%%%%%%%%%%%%%%%%%%%%%%%%%%%%%%%%%%%%%%%%%%%%%%
Let $\TT_\coarse\in\T$.
Due to the regularity assumption $f\in H^1(\Gamma)^D$, the mapping property \eqref{eq:single layer operator}, and $\XX_\coarse\subset L^2(\Gamma)^D$, the residual satisfies that $f-\mathfrak{V}\Psi_\coarse\in H^1(\Gamma)^D$ for all $\Psi_\coarse\in\XX_\coarse$.
This allows to employ the  weighted-residual {\sl a~posteriori} error estimator
\begin{subequations}\label{eq:eta bem}
\begin{align}
 \eta_\coarse := \eta_\coarse(\TT_\coarse)
 \quad\text{with}\quad 
 \eta_\coarse(\SS)^2:=\sum_{T\in\SS} \eta_\coarse(T)^2
 \text{ for all }\SS\subseteq\TT_\coarse,
\end{align}
where, for all $T\in\TT_\coarse$, the local refinement indicators read
\begin{align}
\eta_\coarse(T):=\revision{|T|^{\frac{1}{2(d-1)}}} \seminorm{f-\mathfrak{V}\Phi_\coarse}{H^1(T)}.
\end{align}
\end{subequations}
This estimator goes back to the works \cite{cs96,c97}, where reliability \eqref{eq:reliable bem} is proved for standard 2D BEM with piecewise polynomials on polygonal geometries, while the corresponding result for 3D BEM is first found in \cite{cms}.
%We refer, e.g., to the monographs~\cite{ainsworth-oden,verfuerth} for the analysis of the residual {\sl a~posteriori} error estimator~\eqref{eq:eta bem} in the frame of standard FEM with piecewise polynomials of fixed order.

%%%%%%%%%%%%%%%%%%%%%%%%%%%%%%%%%%%%%%%%%%%%%%%%%%%%%%%%%%%%%%%%%%%%%%%%%%%%%%%%%%%%%%%%%%%%%
\subsection{Adaptive algorithm}
%%%%%%%%%%%%%%%%%%%%%%%%%%%%%%%%%%%%%%%%%%%%%%%%%%%%%%%%%%%%%%%%%%%%%%%%%%%%%%%%%%%%%%%%%%%%%
We consider the following concrete realization of the abstract algorithm from, e.g.,  \cite[Algorithm~2.2]{axioms}.
\begin{algorithm}
\label{alg:bem algorithm}
\textbf{Input:} 
\hspace{-1.2mm}D\"orfler parameter $\theta\in(0,1]$ and marking constant $\const{min}\in[1,\infty]$.\\
\textbf{Loop:} For each $\ell=0,1,2,\dots$, iterate the following steps:
\begin{itemize}
\item[\rm(i)] Compute Galerkin approximation $\Phi_\ell\in\XX_\ell$.
\item[\rm(ii)] Compute refinement indicators $\eta_\ell({T})$
for all elements ${T}\in\TT_\ell$.
\item[\rm(iii)] Determine a set of marked elements $\MM_\ell\subseteq\TT_\ell$ which has up to the multiplicative constant $\const{min}$  minimal cardinality, such that the following D\"orfler marking is satisfied
\begin{align}
 \theta\,\eta_\ell^2 \le \eta_\ell(\MM_\ell)^2.
 \end{align}
\item[\rm(iv)] Generate refined mesh $\TT_{\ell+1}:=\refine(\TT_\ell,\MM_\ell)$. 
\end{itemize}
\textbf{Output:} Refined meshes $\TT_\ell$ and corresponding Galerkin approximations $\Phi_\ell$ with error estimators $\eta_\ell$ for all $\ell \in \N_0$.
\end{algorithm}

%%%%%%%%%%%%%%%%%%%%%%%%%%%%%%%%%%%%%%%%%%%%%%%%%%%%%%%%%%%%%%%%%%%%%%%%%%%%%%%%%%%%%%%%%%%%%
\subsection{Optimal convergence}
%%%%%%%%%%%%%%%%%%%%%%%%%%%%%%%%%%%%%%%%%%%%%%%%%%%%%%%%%%%%%%%%%%%%%%%%%%%%%%%%%%%%%%%%%%%%%
We set
\begin{align}
 \T(N):=\set{\TT_\coarse\in\T}{\#\TT_\coarse-\#\TT_0\le N}
 \quad\text{for all }N\in\N_0
\end{align}
and, for all $s>0$,
\begin{align}
\norm{\phi}{\mathbb{A}_s^{\rm est}}
 := \sup_{N\in\N_0}\min_{\TT_\coarse\in\T(N)}(N+1)^s\,\eta_\coarse\in[0,\infty].
\end{align}
We say that the solution $\phi\in H^{-1/2}(\Gamma)^D$ lies in the \textit{approximation class $s$ with respect to the estimator} if
\begin{align}
\norm{\phi}{\mathbb{A}_s^{\rm est}} <\infty.
\end{align}
By definition, $\norm{\phi}{\mathbb{A}^{\rm est}_s}<\infty$  implies that the error estimator $\eta_\coarse$
on the optimal meshes $\TT_\coarse$ decays at least with rate $\OO\big((\#\TT_\coarse)^{-s}\big)$. The following main theorem \revision{of our work} states that each possible rate $s>0$ will indeed be realized by Algorithm~\ref{alg:bem algorithm}.
\revision{The proof is  given in  Section~\ref{sec:abstract setting bem}.
There, we will verify certain abstract properties to employ our recent abstract counterpart \cite[Theorem~3.4]{gp20} of Theorem~\ref{thm:main bem}, which itself builds on the so-called \textit{axioms of adaptivity} from \cite{axioms}.}
Such an optimality result was first proved in \cite{fkmp} for the Laplace operator $\mathfrak{P}=-\Delta$ on a polyhedral domain $\Omega$.
As ansatz space, \cite{fkmp} considered piecewise constants on shape-regular triangulations.
\cite{part1} in combination with \cite{invest} extends the assertion to piecewise polynomials on shape-regular curvilinear triangulations  of some piecewise smooth boundary $\Gamma$.
Independently, \cite{gantumur} proved the same result for globally smooth $\Gamma$ and arbitrary symmetric and elliptic boundary integral operators.

\begin{theorem}\label{thm:main bem}%{thm:abstract bem}
Let $(\TT_\ell)_{\ell\in\N_0}$ be the sequence of meshes generated by Algorithm~\ref{alg:bem algorithm}.
Then, there hold the following statements {\rm(i)--(iii)}:
\begin{enumerate}[\rm (i)]
\item\label{item:qabstract reliable bem}
\revision{The} residual error estimator \revision{\eqref{eq:eta bem}}satisfies reliability, i.e., there exists a constant $\const{rel}>0$ such that
\begin{align}\label{eq:reliable bem}
 \norm{\phi-\Phi_\coarse}{H^{-1/2}(\Gamma)}\le \const{rel}\eta_\coarse\quad\text{for all }\TT_\coarse\in\T.
\end{align}
\item\label{item:qabstract linear convergence bem}
\revision{For} arbitrary $0<\theta\le1$ and $\const{min}\in[1,\infty]$, the   estimator converges linearly, i.e., there exist constants $0<\ro{lin}<1$ and $\const{lin}\ge1$ such that
\begin{align}\label{eq:linear bem}
\eta_{\ell+j}^2\le \const{lin}\ro{lin}^j\eta_\ell^2\quad\text{for all }j,\ell\in\N_0.
\end{align}
\item\label{item:qabstract optimal convergence bem}
\revision{There} exists a constant $0<\theta_{\rm opt}\le1$ such that for all $0<\theta<\theta_{\rm opt}$ and $\const{min}\in[1,\infty)$, the estimator converges at optimal rate, i.e., for all $s>0$, there exist constants $c_{\rm opt},\const{opt}>0$ such that
\begin{align}\label{eq:optimal bem}
 c_{\rm opt}\norm{\phi}{\mathbb{A}_s^{\rm est}}
 \le \sup_{\ell\in\N_0}{(\# \TT_\ell-\#\TT_0+1)^{s}}\,{\eta_\ell}
 \le \const{opt}\norm{\phi}{\mathbb{A}_s^{\rm est}}.
\end{align}
%where the lower bound relies only on \eqref{R:sons bem}.
\end{enumerate}
\noindent All involved constants $\const{rel},\const{lin},q_{\rm lin},\theta_{\rm opt}$, and $\const{opt}$ depend only on the dimensions $d,D$, the coefficients of the differential operator $\mathfrak{P}$,  \revision{the boundary} $\Gamma$, {\color{black} the (fixed) number $M$ of boundary parts $\Gamma_m$, the parametrizations $\gamma_m$ and $\gamma_z$, the initial meshes 
 $\widehat \TT_{0,m}$, and the polynomial orders $(p_{1,m},\dots,p_{d-1,m})$ for $m\in\{1,\dots,M\}$ and $z\in\NN_\gamma$ \revision{(see~\eqref{eq:Ngamma})}}, 
 while $\const{lin},\ro{lin}$ depend additionally on $\theta$ and the sequence $(\Phi_\ell)_{\ell\in\N_0}$, and $\const{opt}$ depends furthermore on $\const{min}$ and $s>0$.
The constant $c_{\rm opt}$ depends only on \revision{$d$}, $\#\TT_0$,  $s$, and if there exists $\ell_0$ with $\eta_{\ell_0}=0$, then also on $\ell_0$ and $\eta_0$.
\end{theorem}

\begin{remark}
If the sesquilinear form $\sprod{\mathfrak{V}\,\cdot}{\cdot}$ is Hermitian, then $\const{lin}$, $\ro{lin},$ and  $\const{opt}$ are  independent of $(\Phi_\ell)_{\ell\in\N_0}$; see \cite[Remark~4.14]{gp20}.
\end{remark}

\begin{remark}
%Let 
%$\XX_\coarse$ contain all componentwise constant functions, i.e.,  
%\begin{align}\label{eq:alexass}
%x\in\XX_\coarse
%\quad\text{for all }x\in\C^D.
%\end{align}
\revision{Under} the assumption that $\norm{h_{\ell}}{L^\infty(\Omega)}\to0$ as $\ell\to\infty$, one can show  that
$\XX_\infty:=\overline{\bigcup_{\ell\in\N_0}\XX_\ell}=H^{-1/2}(\Gamma)^D$; see \cite[Remark~3.7]{gp20} for the elementary proof.
This observation allows to follow the ideas of~\cite{helmholtz} and to show that the adaptive algorithm yields \revision{optimal} convergence even if the sesquilinear form $\sprod{\mathfrak{V}\,\cdot}{\cdot}$ is only elliptic up to some compact perturbation, provided that the continuous problem is well-posed. This includes, e.g., adaptive BEM for the Helmholtz equation; see~\cite[Section~6.9]{s}. For details, the reader is referred  to~\cite{helmholtz,bbhp19}.%
\revision{We note that $\norm{h_{\ell}}{L^\infty(\Omega)}\to0$ can easily be algorithmically enforced (without influencing the optimality result) by additionally marking the largest elements in Algorithm~\ref{alg:bem algorithm}~{\rm (iii)}; see \cite[Proposition~16]{helmholtz}.}
\end{remark}

\begin{remark}\label{rem:rational main bem}
{\rm (a)}
Theorem~\ref{thm:main bem} is still valid if one replaces the ansatz space $\XX_\coarse$ by {\rm rational hierarchical splines}, i.e., by the set
\begin{align}
\XX_\coarse^{W_0}:=\Big\{{W_0^{-1}\Psi_\coarse}:\Psi_\coarse\in\XX_\coarse\Big\},
\end{align}
where $\widehat W_{0,m}:=W_0\circ\gamma_m$ is a fixed positive weight function in the initial space of hierarchical splines $\widehat\SS^{(p_{1,m},\dots,p_{d-1,m})}(\widehat\KK_{0,m},\widehat\TT_{0,m})$ for all $m\in\{1,\dots,M\}$.
With the B-spline basis $\widehat\BB_{0,m}$ on $\widehat\TT_{0,m}$, we even suppose that $\widehat W_{0,m}$ can be written as 
\begin{align}\label{eq:W0 representation}
\widehat W_{0,m}=\sum_{\widehat\beta\in\widehat\BB_{0,m}} w_{0,m,\widehat\beta}\,  \widehat \beta \quad\text{with }w_{0,m,\widehat\beta}\ge 0.
\end{align}
We will prove this generalization in Section~\ref{sec:rational main proof bem}.
In this case, the constants depend additionally on $W_0$.

{\rm (b)}
Moreover, Theorem~\ref{thm:main bem} still holds true if newly inserted knots have a multiplicity higher than one, i.e., if one uses the uniformly refined knots of Remark~\ref{eq:different uni} with $1\le q_{i,m}\le p_{i,m}$ to define (rational) hierarchical splines.
The corresponding proof works verbatim.

{\rm (c)}
Finally, if one defines for an element $\widehat T$ of a hierarchical mesh $\widehat\TT_{\coarse,m}$ its neighbours $\Nu_{\coarse,m}(\widehat T)$  as in Remark~\ref{rem:Nu for p0}, and adapts the definition of admissibility and $\refine(\cdot,\cdot)$ accordingly, one can also allow for lowest-order polynomial degrees $p_{i,m}\in\N_0$ as well as full knot multiplicities $q_{i,m}=p_{i,m}+1$.
\end{remark}

%\subsection{Optimal convergence for hierarchical splines}
%Altogether, we have  specified the abstract framework of Section~\ref{sec:abstract setting bem} to hierarchical meshes and splines.
%The following theorem is the main result of this work. It states that all assumptions of Theorem~\ref{thm:abstract bem} are satisfied for the present hierarchical approach.
%The proof is  given in Section \ref{sec:proof bem}.

%\begin{theorem}\label{thm:main bem}
%Hierarchical splines on admissible meshes satisfy the abstract assumptions  \eqref{M:patch bem}--\eqref{M:semi bem}, \eqref{R:sons bem}--\eqref{R:overlay bem}, and \eqref{S:inverse bem}--\eqref{S:stab bem} from Section~\ref{sec:abstract setting bem}, where the constants depend only on the dimensions $d,D$, the (fixed) number $M$ of boundary parts $\Gamma_m$, the parametrizations $\gamma_m$ and $\gamma_z$, the initial meshes 
% $\widehat \TT_{0,m}$, and the polynomial orders $(p_{1,m},\dots,p_{d-1,m})$ for $m\in\{1,\dots,M\}$ and $z\in\NN_\gamma$ \revision{(see~\eqref{eq:Ngamma})}.
%By Theorem~\ref{thm:abstract bem}, this implies reliability \eqref{eq:reliable bem}  of the error estimator, and  linear convergence \eqref{eq:linear bem} with optimal  algebraic rate \eqref{eq:optimal bem} for the adaptive strategy from Algorithm~\ref{alg:bem algorithm}.
%\end{theorem}

%% file: 06_numeric_high.tex
% !TEX encoding = MacOSRoman
% !TEX root = igabem3d.tex
\section{Numerical experiments with hierarchical splines}\label{sec:numerics for bem1}
In this section, we empirically investigate the performance of Algorithm \ref{alg:bem algorithm} in two typical situations: In Section \ref{subsec:cube bem1}, the solution is generically  singular at the edges of $\Gamma=\partial\Omega$.
In  Section~\ref{subsec:bendcube bem1}, the solution is nearly singular at one point.
%In Section \ref{subsec:slit problem}, we consider a slit problem.
%In all examples, the exact solution is known.
%This allows to analyse  the reliability and efficiency of the proposed estimator.
We consider the 3D Laplace operator $\mathfrak{P}:=-\Delta$.
The  corresponding fundamental solution reads
\begin{align}
G(z):=\frac{1}{4\pi}\frac{1}{|z|}\quad\text{for all }z\in \R^3\setminus\{0\}.
\end{align}
As already mentioned in Section~\ref{subsec:model problem bem}, the corresponding single-layer operator $\mathfrak{V}:H^{-1/2}(\Gamma)\to H^{1/2}(\Gamma)$ is elliptic.
For experiments in 2D, we refer to the recent work~\cite{fgkss19} for univariate hierarchical splines and to~\cite{resigabem,fgps19,gps20} for univariate standard and rational splines.

In the first example (Section~\ref{subsec:cube bem1}), we consider the exterior Laplace--Dirichlet problem
\begin{subequations}\label{eq:Laplace exterior bem}
\begin{align}
\begin{split}
-\Delta u&=0\quad\text{in }{\R^3\setminus\overline{\Omega}},\\ u&=g\quad\text{on } \Gamma,\end{split}
\end{align}
for given Dirichlet data $g\in {H}^{1/2}(\Gamma_{})$, together with the radiation condition
\begin{align}
u(x)=\mathcal{O}\Big(\frac{1}{|x|}\Big)\quad\text{as }|x|\to\infty.
\end{align}
\end{subequations}
%The (exterior) normal derivative $\phi:=\partial_{-\nu} u$ of the unique weak solution $u$ satisfies the integral equation \eqref{eq:strong} with 
Then, \eqref{eq:Laplace exterior bem} can be equivalently rewritten as integral equation \eqref{eq:strong}; see, e.g., \cite[Theorem~7.15 and Theorem~8.9]{mclean}. %, \cite[Section~7.5]{s}, or \cite[Section~3.4.2.2]{ss}.
Indeed, the (exterior) normal derivative $\phi:=-\partial_{\nu} u$ of the  weak solution $u$ of  \eqref{eq:Laplace exterior bem} satisfies the integral equation \eqref{eq:strong} with $f:=(\mathfrak{K}-1/2)g$, i.e., 
\begin{equation}\label{eq:Symmy}
\mathfrak{V}\phi =(\mathfrak{K}-1/2) g, 
\end{equation}
where 
\begin{align}\label{eq:double layer mapping}
\mathfrak{K}: H^{1/2}(\Gamma)\to H^{1/2}(\Gamma)
\end{align}
denotes the \textit{double-layer operator}. 
According to \cite[Corollary~3.3.12 and Theorem~3.3.13]{ss}, if $\Gamma$ is piecewise smooth and if $g\in L^\infty(\Gamma)$, there holds the representation 
%For $g\in L^\infty(\Gamma)$
\begin{align}\label{eq:double-layer}
\begin{split}
\mathfrak{K}g(x)= \int_{\Gamma_{}} g(y)\partial_{\nu(y)} G(x,y) \,dy \quad\text{if }\Gamma \text{ is smooth in } x\in\Gamma \text{ and }g \text{ is continuous at }x.
\end{split}
\end{align}

In the second example (Section~\ref{subsec:bendcube bem1}), we consider the interior Laplace--Dirichlet problem
\begin{align}\label{eq:Laplace interior bem}
\begin{split}
-\Delta u&=0\quad\text{in }{\Omega},\\ u&=g\quad\text{on } \Gamma,\end{split}
\end{align}
for given Dirichlet data $g\in {H}^{1/2}(\Gamma_{})$.
Then, \eqref{eq:Laplace interior bem} can be equivalently rewritten as integral equation \eqref{eq:strong}; see, e.g., \cite[Theorem~7.6]{mclean}. %\cite[Section~7.1]{s}, or \cite[Section~3.4.2.1]{ss}.
Indeed, the normal derivative $\phi:=\partial_\nu u$ of the  weak solution $u$ of  \eqref{eq:Laplace interior bem} satisfies the integral equation \eqref{eq:strong} with $f:=(\mathfrak{K}+1/2)g$ and $\mathfrak{K}$  the double-layer operator~\eqref{eq:double layer mapping}, i.e.,
\begin{equation}\label{eq:Symmy interior}
\mathfrak{V}\phi =(\mathfrak{K}+1/2) g. 
\end{equation}

The integral representation \eqref{eq:double-layer} is satisfied for both examples.
Indeed, the surfaces $\Gamma_m$ of the boundary $\Gamma=\bigcup_{m=1}^M\Gamma_m$  are  parametrized via rational splines, i.e., for each $m\in\{1,\dots,M\}$, 
there exist polynomial orders $p_{1(\gamma,m)},p_{2(\gamma,m)}\in\N$, a   two-dimensional vector $\widehat\KK_{\gamma,m}=(\widehat\KK_{1(\gamma,m)},\widehat\KK_{2(\gamma,m)})$ of $p_{i(\gamma,m)}$-open knot vectors with multiplicity smaller or equal to $p_{i(\gamma,m)}$ for the interior knots, and a positive spline weight function $\widehat W_{\gamma,m}\in\widehat\SS^{(p_{1(\gamma,m)},p_{2(\gamma,m)})}(\widehat \KK_{\gamma,m})$ such that the parametrization $\gamma_m:\widehat \Gamma_m\to\Gamma_m$ satisfies   that
\begin{align}
\gamma_m\in \set{\widehat W_{\gamma,m}^{-1}\,\widehat S}{\widehat S\in\widehat\SS^{(p_{1(\gamma,m)},p_{2(\gamma,m)})}(\widehat\KK_{\gamma,m})^3}.
\end{align}

Based on the knots $\widehat\KK_{\gamma,m}$ for the geometry, we choose the initial knots $\widehat\KK_{0,m}$ for the discretization.  
%such that (at least) the corresponding nodes coincide, i.e., $\widehat\NN_0=\widehat \NN_\gamma$.
As basis for the considered  ansatz spaces of (non-rational) hierarchical splines, we use the basis given in \eqref{eq:hierarchical basis bem}.
%To (approximately) calculate the Galerkin matrix, the right-hand side vector, and the weighted-residual error  estimator, we use tensor Gauss quadrature. 
%Recall that Lemma~\ref{lem:properties for B-splines} \eqref{item:derivative of splines} provides a formula for the derivative of B-splines.
To (approximately) calculate the Galerkin matrix and the right-hand side vector, we proceed as in \cite[Chapter~5]{ss} where all  singular integrals are transformed via Duffy transformations and then computed with tensor Gauss quadrature.
%For the (dense) Galerkin matrix, we do not apply any matrix compression techniques such as wavelet methods \cite{wave,dhr,cad2wave}, fast multipole methods \cite{fmm,igafmm,harbnew}, or $\mathcal{H}$-matrix methods \cite{hmat,igahmat}.
To calculate the weighted-residual error estimator\footnote{To ease the computation, we replace $|T|^{1/2}$ in \eqref{eq:eta bem} by the equivalent term $\diam(\Gamma)\,|\widehat T|^{1/2}$.
Here, $\widehat T$ denotes the corresponding element of $T\in\TT_{\ell,m}$ in the parameter domain $\widehat\Gamma_m$.} \eqref{eq:eta bem}, we employ  formula \eqref{eq:surface gradient} for the surface gradient and use again tensor Gauss quadrature.
%\begin{align}
To this end, we approximate $\nabla ((f-\mathfrak{V}\Phi_\ell)\circ\gamma_m)$ on an element $\widehat T\in\widehat\TT_{\ell,m}$ by the gradient of the polynomial interpolation of the residual $f-\mathfrak{V}\Phi_\ell$ as in \cite[Section~7.1.5]{karkulik}.
In particular, we have to evaluate the residual  at some quadrature points which can be done (approximately) using appropriate Duffy transformations and tensor Gauss quadrature as in \cite[Sections~5.1--5.2]{diplomarbeit}.
%\end{align}
 
%Then, we use adapted Gauss quadrature to compute the resulting integrals with appropriate accuracy;
% see \cite[Section 5]{diplomarbeit} for details.
%The computation of the collocation solution and the weighted-residual error estimator essentially relies on the numerical evaluation of $V\psi_h$ for $\psi_h\in\XX_h$  and $Kg$ for $g\in H^{1/2}(\Gamma)$ with the double-layer operator $K$ as in \eqref{eq:double-layer}, which is found as well in \cite[Section 5]{diplarbeit}.
%For the weighted-residual error estimator \eqref{eq:residual}, we replace $|\omega_h(z)|$ by the length $|\gamma^{-1}(\omega_h(z))|$, since this eases the calculation.
%Note that $|\omega_h(z)|$ $\simeq|\gamma^{-1}(\omega_h(z))|$, where the hidden constants  depend only on the parametrization $\gamma$.

To (approximately) calculate the energy error, we proceed as follows:
Let $\Phi_{\ell}\in \XX_\ell$ be the Galerkin approximation of the $\ell$-th step with  the corresponding coefficient vector $\boldsymbol{c_\ell}$.
%Let  $\phi_{h}^{\rm col}\in \XX_h$ be  the collocation approximation  with $\boldsymbol{c_h^{\rm col}}$ the corresponding coefficient vector.
Further, let $\boldsymbol{V_\ell}$ be the Galerkin matrix. % of the $h$-th step.
With  Galerkin orthogonality \eqref{eq:galerkin bem} and the energy norm $\norm{\phi}{\mathfrak{V}}^2=\dual{\mathfrak{V}\phi}{\phi}$ obtained by Aitken's $\Delta^2$-extrapolation, we can compute the energy error as
\begin{align}\label{eq:error calc gal bem1}
\begin{split}
\norm{\phi-\Phi_\ell}{\mathfrak{V}}^2&=\norm{\phi}{\mathfrak{V}}^2-\norm{\Phi_\ell}{\mathfrak{V}}^2=\norm{\phi}{\mathfrak{V}}^2-\boldsymbol{V_\ell}\boldsymbol{c_\ell}\cdot \boldsymbol{c_\ell}.
\end{split}
\end{align}
%%%%%%%%%%%%%%%%%%%%%%%%%%%%%%%%%%%%%%%%%%%%%%%%%%%%%%%%%%%%%%%%%%%%%%%%%%%%%%%%
\subsection{Solution with edge singularities on cube}
\label{subsec:cube bem1}
In the first experiment, we consider 
\begin{align}
\Omega := (0,1/10)^3.
\end{align}
Each of the six faces $\Gamma_m$ of $\Omega$ can be parametrized by non-rational splines of degree $p_{1(\gamma,m)}:=p_{2(\gamma,m)}:=1$ corresponding to the knot vectors  $\widehat\KK_{1(\gamma,m)}:=\widehat\KK_{2(\gamma,m)} := (0,0,1,1)$; see \cite[Section~6.1]{igafem}.
%We set the control points
%\begin{align*}
%c^\gamma_{1,1} = (0,0),~ c^\gamma_{2,1} = (1,0),~ c^\gamma_{1,2} = (0,1),~ c^\gamma_{2,2} = (1,1)
%\end{align*} and all weights equal to $1$. 
We choose the right-hand side $f:=1$ in \eqref{eq:strong}.
Note that the constant function $1$ satisfies the Laplace problem, which is why \eqref{eq:Symmy interior} implies that  $\mathfrak{K}1=-1/2$.
We conclude that 
\begin{align}
f=(\mathfrak{K}-1/2) g\quad\text{with }g:=-1.
\end{align}
This means that the considered integral equation stems from an exterior Laplace--Dirichlet problem~\eqref{eq:Laplace exterior bem}.
In particular, we expect singularities at the non-convex edges of $\R^3\setminus\overline{\Omega}$, i.e., at all edges of the cube $\Omega$.

\begin{figure}[h!]%cube
\psfrag{x1}[c][c]{\fontsize{5pt}{6pt}\selectfont $x_1$}
\psfrag{x2}[c][c]{\fontsize{5pt}{6pt}\selectfont $x_2$}
\psfrag{x3}[c][c]{\fontsize{5pt}{6pt}\selectfont $x_3$}
\begin{center}
\includegraphics[width=0.24\textwidth,clip=true]{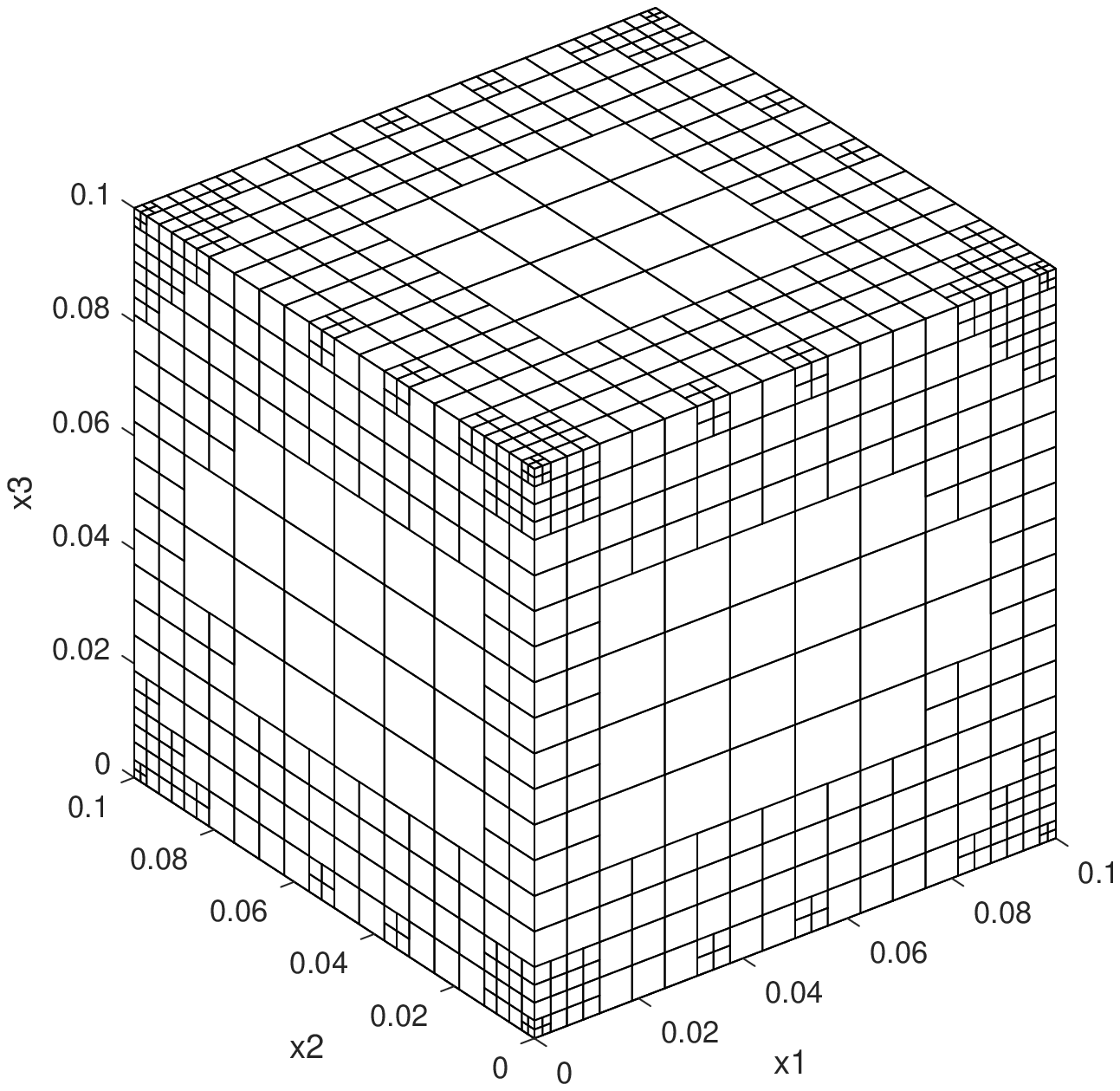}
\includegraphics[width=0.24\textwidth,clip=true]{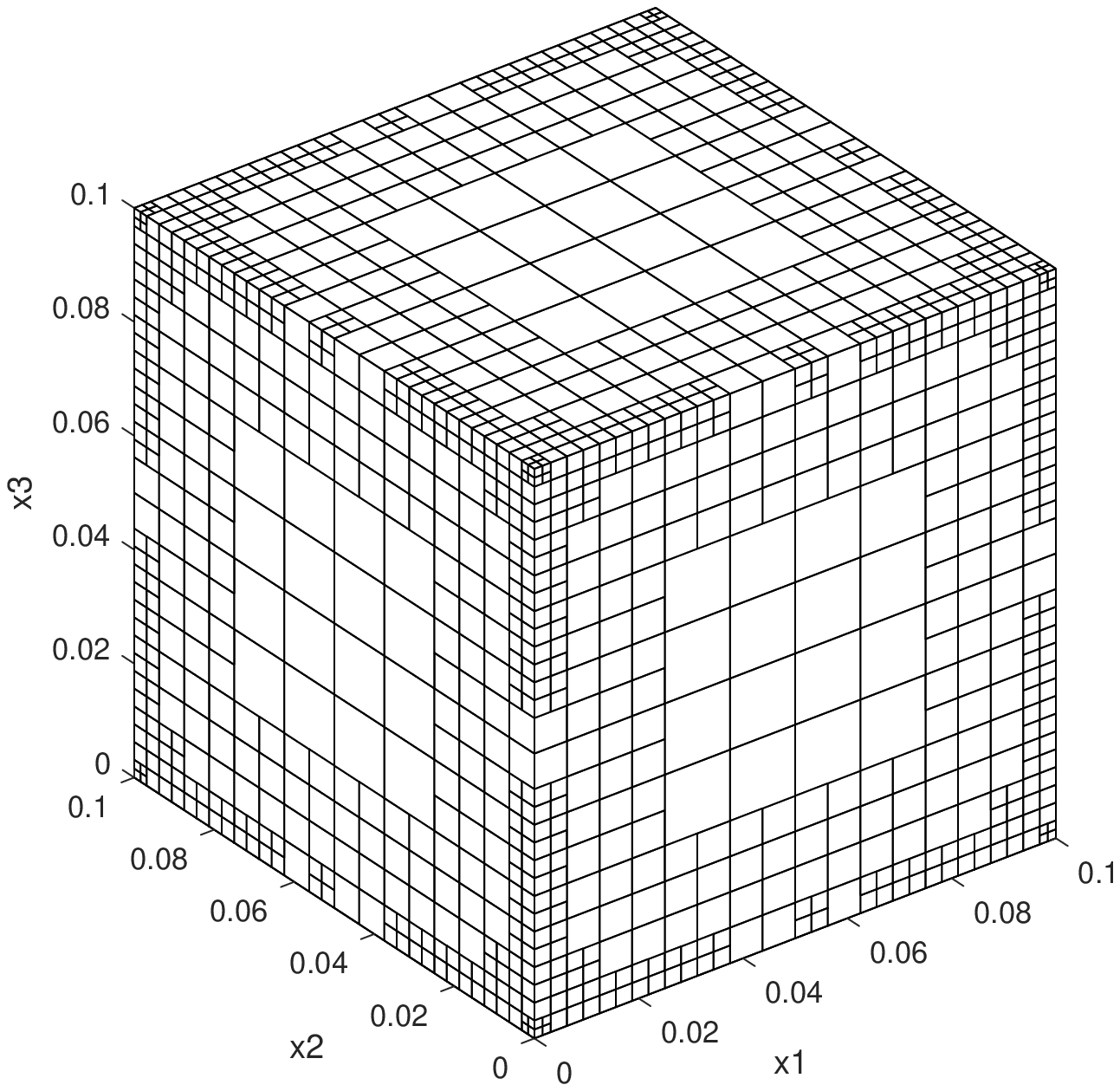}
%
%\vspace{2mm}
%
\includegraphics[width=0.24\textwidth,clip=true]{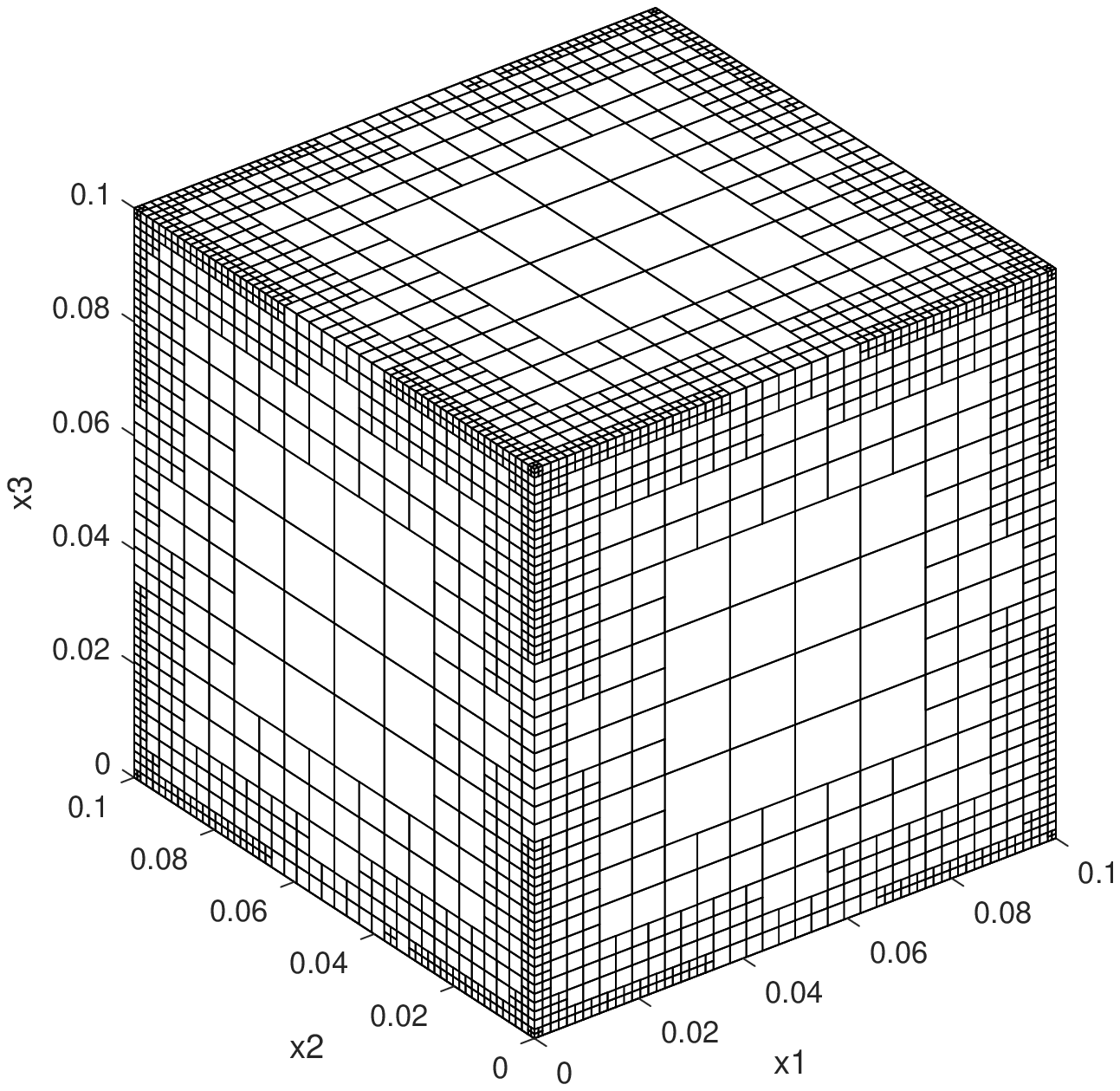}
\includegraphics[width=0.24\textwidth,clip=true]{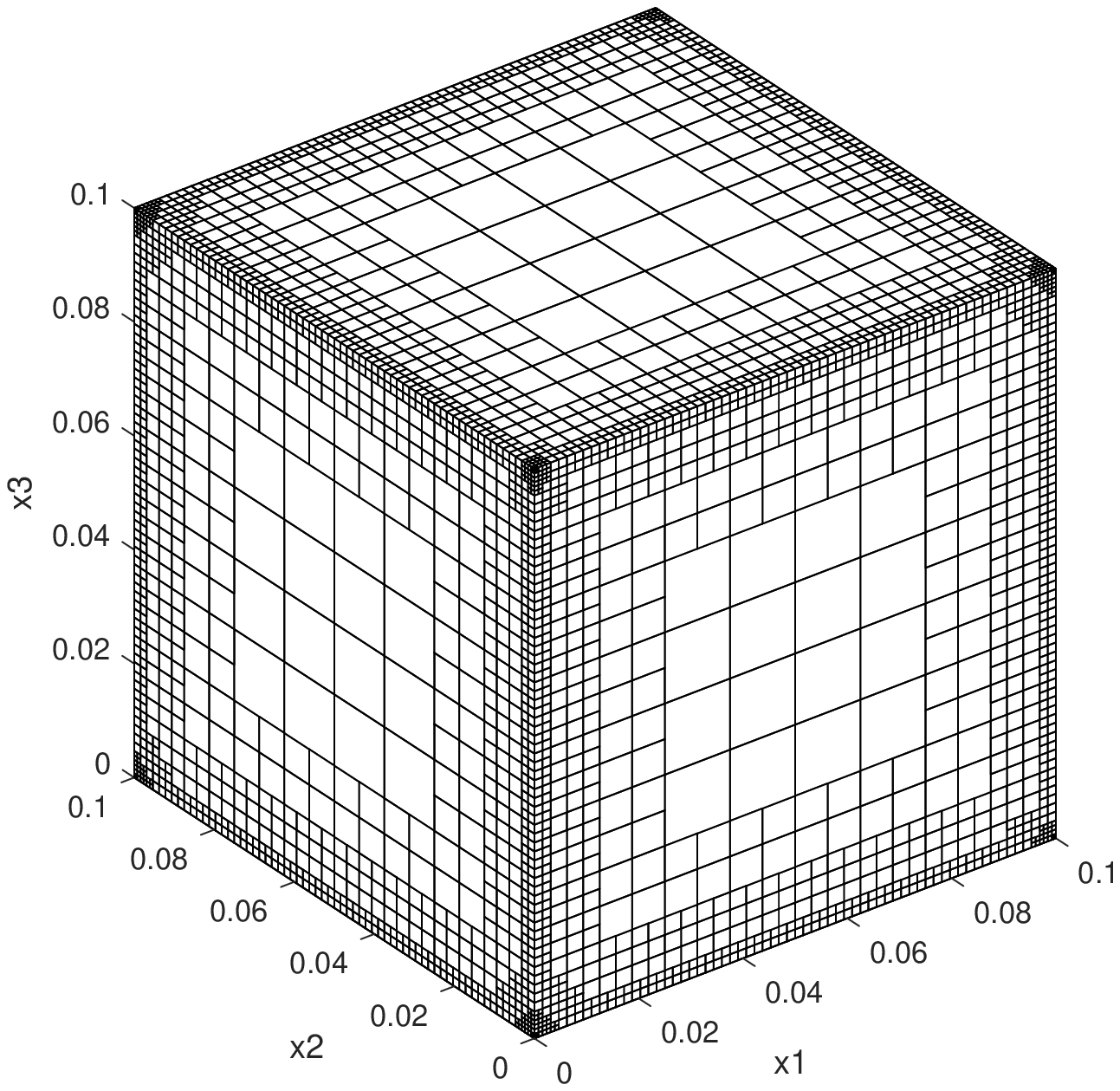}
\end{center}
%\begin{minipage}{0.32\textwidth}
%\begin{center}$\#\TT_4=265$\end{center}
%\end{minipage}
%\begin{minipage}{0.32\textwidth}
%\begin{center}$\#\TT_5=718$\end{center}
%\end{minipage}
%\begin{minipage}{0.32\textwidth}
%\begin{center}$\#\TT_{6}=1777$\end{center}
%\end{minipage}
\caption{Experiment with solution with edge singularities 
on cube  of Section~\ref{subsec:cube bem1}. 
Hierarchical meshes $\TT_8,\TT_{10},\TT_{11},\TT_{13}$ generated by  Algorithm~\ref{alg:bem algorithm} (with $\theta=0.5$) for hierarchical splines of degree $p=1$. 
}
\label{fig:cube_mesh}
\end{figure}

We consider polynomial degrees $p\in\{0,1,2\}$.
For the initial ansatz space with spline degree $p_{1,m}:=p_{2,m}:=p$ for all $m\in\{1,\dots 6\}$, we choose the initial knot vectors $\widehat\KK_{1(0,m)}:=\widehat\KK_{2(0,m)}:=(0,\dots 0,1,\dots,1)$ for all $m\in\{1,\dots,6\}$, where the multiplicity of $0$ and $1$ is $p+1$.
%However, we raise the multiplicity of the inner knot $0.5$ in $\widehat\KK^0_1$ to the maximal multiplicity, that is $p_1$. 
%As a consequence, the ansatz functions are only continuous at $\gamma(\{0.5\}\times[0,1])$,  but not continuously differentiable.
We choose the parameters of Algorithm~\ref{alg:bem algorithm} as $\theta=0.5$ and $\const{min}=1$, where we use the refinement strategy of Remark~\ref{rem:rational main bem} {\rm(c)} in the lowest-order case $p=0$.
For comparison, we also consider uniform refinement, where we mark all elements in each step, i.e., $\MM_\ell=\TT_\ell$ for all $\ell\in\N_0$. 
This leads to uniform bisection of all elements.
 In Figure~\ref{fig:cube_mesh}, one can see some adaptively generated hierarchical meshes.
 In  Figure~\ref{fig:cube_p} and Figure~\ref{fig:cube_pcomp}, 
we plot   the energy error $\norm{\phi-\Phi_\ell}{\mathfrak{V}}$ and the error estimator $\eta_\ell$ against the  number of elements $\#\TT_\ell$.  
%The energy error is computed via \eqref{eq:error calc gal fem}, where $\norm{\nabla u}{L^2(\Omega)}^2$ can even be computed exactly due to the simple form of $u$.
All values are plotted in a double-logarithmic scale.  The experimental convergence rates are thus visible as the slope of the corresponding curves.
%In all cases, the lines of the error and the error estimator are parallel, which numerically indicates reliability \eqref{eq:reliable bem}
% and efficiency \eqref{eq:efficient}.
 Although we only proved reliability \eqref{eq:reliable bem} of the employed estimator, the curves for the error and the estimator are parallel in each case, which numerically indicates reliability and efficiency.
The uniform approach always leads to the suboptimal  convergence rate $\mathcal{O}((\#\TT_\ell)^{-1/3})$  due to the edge singularities.
Independently on the chosen polynomial degree $p$, the adaptive approach leads approximately to the rate $\mathcal{O}((\#\TT_\ell)^{-1/2})$.
% at $\{0\}\times [0,1]$.
% However, it seems that the adaptive strategy converges at rate $\mathcal{O}((\# \TT_\ell)^{-\min(\tau-1/2,p/2)})$, i.e., (if possible) at double rate.
% The speed of convergence remains unchanged if one decreases the adaptivity parameter $\theta$; see, e.g., Figure~\ref{fig:square_pcomp1} for $\theta=0.1$.
For smooth solutions $\phi$, one would expect the rate  $\mathcal{O}((\#\TT_\ell)^{-3/4-p/2})$; see \cite[Corollary~4.1.34]{ss}.
However, according to Theorem~\ref{thm:main bem}, the achieved rate is optimal if one uses the proposed refinement strategy and the resulting hierarchical splines.
The reduced optimal convergence rate is probably due to the edge singularites.
% which would  actually require anisotropic refinement.
A similar convergence behavior is also witnessed in \cite[Section~5.2]{ferraz} for the lowest-order case $p=0$.
\cite{ferraz} additionally considers anisotropic refinement which recovers the optimal convergence rate $\mathcal{O}((\#\TT_\ell)^{-3/4})$.
\revision{We mention that for graded meshes the theoretical convergence rates are explicitly known on polyhedral domains; see \cite{ps90} or \cite[Chapter~7]{gs18}.}
%In the following Section~\ref{subsec:iso vs aniso}, we give a heuristic argumentation 
%%in \cite[Section~7.3]{cmps} 
%which suggests that the optimal  convergence rate with respect to the number of elements for isotropic refinement is bounded by $\min(\tau-1/2,p/2)$.
%In our numerical examples, it seems that this rate is attained exactly.
%Finally, in Figure~\ref{fig:square_scomp}, we consider $\tau\in\{5/4,7/4\}$ with $\theta=0.5$.

\begin{figure}[h!]%cube convergence 
\psfrag{p=0, unif., est.}[l][l]{\fontsize{5pt}{6pt}\selectfont $p=0$, unif., est.}
\psfrag{p=0, unif., err.}[l][l]{\fontsize{5pt}{6pt}\selectfont $p=0$, unif., err.}
\psfrag{p=0, adap., est.}[l][l]{\fontsize{5pt}{6pt}\selectfont  $p=0$, adap., est.}
\psfrag{p=0, adap., err.}[l][l]{\fontsize{5pt}{6pt}\selectfont  $p=0$, adap., err.}
\psfrag{p=1, unif., est.}[l][l]{\fontsize{5pt}{6pt}\selectfont $p=1$, unif., est.}
\psfrag{p=1, unif., err.}[l][l]{\fontsize{5pt}{6pt}\selectfont $p=1$, unif., err.}
\psfrag{p=1, adap., est.}[l][l]{\fontsize{5pt}{6pt}\selectfont  $p=1$, adap., est.}
\psfrag{p=1, adap., err.}[l][l]{\fontsize{5pt}{6pt}\selectfont  $p=1$, adap., err.}
\psfrag{p=2, unif., est.}[l][l]{\fontsize{5pt}{6pt}\selectfont $p=2$, unif., est.}
\psfrag{p=2, unif., err.}[l][l]{\fontsize{5pt}{6pt}\selectfont $p=2$, unif., err.}
\psfrag{p=2, adap., est.}[l][l]{\fontsize{5pt}{6pt}\selectfont  $p=2$, adap., est.}
\psfrag{p=2, adap., err.}[l][l]{\fontsize{5pt}{6pt}\selectfont  $p=2$, adap., err.}
\psfrag{p=3, unif., est.}[l][l]{\fontsize{5pt}{6pt}\selectfont $p=3$, unif., est.}
\psfrag{p=3, unif., err.}[l][l]{\fontsize{5pt}{6pt}\selectfont $p=3$, unif., err.}
\psfrag{p=3, adap., est.}[l][l]{\fontsize{5pt}{6pt}\selectfont  $p=3$, adap., est.}
\psfrag{p=3, adap., err.}[l][l]{\fontsize{5pt}{6pt}\selectfont  $p=3$, adap., err.}
\psfrag{error and estimator}[c][c]{\fontsize{5pt}{6pt}\selectfont error and estimator}
\psfrag{number of elements}[c][c]{\fontsize{5pt}{6pt}\selectfont number of elements}
\psfrag{O(2.5)}[r][r]{\fontsize{5pt}{6pt}\selectfont $\mathcal{O}(N^{-5/2})$}
\psfrag{O(13)}[l][l]{\fontsize{5pt}{6pt}\selectfont $\mathcal{O}(N^{-1/3})$}
\psfrag{O(1)}[l][l]{\fontsize{5pt}{6pt}\selectfont $\mathcal{O}(N^{-1})$}
\psfrag{O(1.5)}[r][r]{\fontsize{5pt}{6pt}\selectfont $\mathcal{O}(N^{-5/2})$}
\psfrag{O(1.5)}[r][r]{\fontsize{5pt}{6pt}\selectfont $\mathcal{O}(N^{-3/2})$}
\psfrag{O(2)}[r][r]{\fontsize{5pt}{6pt}\selectfont $\mathcal{O}(N^{-2})$}
\psfrag{O(12)}[r][r]{\fontsize{5pt}{6pt}\selectfont $\mathcal{O}(N^{-1/2})$}
\psfrag{O(54)}[r][r]{\fontsize{5pt}{6pt}\selectfont $\mathcal{O}(N^{-5/4})$}
\psfrag{O(74)}[r][r]{\fontsize{5pt}{6pt}\selectfont $\mathcal{O}(N^{-7/4})$}
\psfrag{O(92)}[r][r]{\fontsize{5pt}{6pt}\selectfont $\mathcal{O}(N^{-9/2})$}
\psfrag{O(1)}[l][l]{\fontsize{5pt}{6pt}\selectfont $\mathcal{O}(N^{-1})$}
\begin{center}
\includegraphics[width=0.32\textwidth,clip=true]{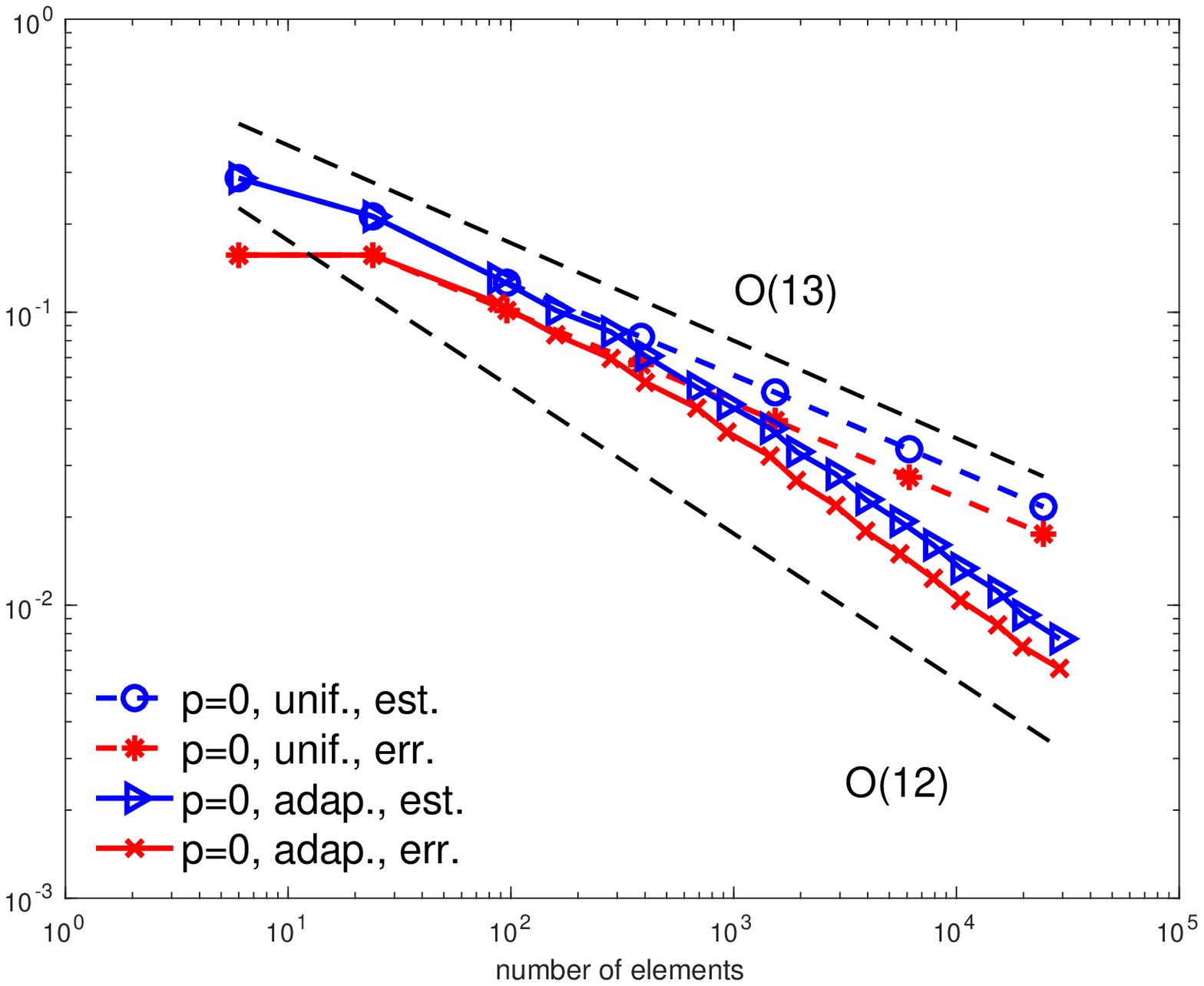}
%
%\vspace{2mm}
%
\includegraphics[width=0.32\textwidth,clip=true]{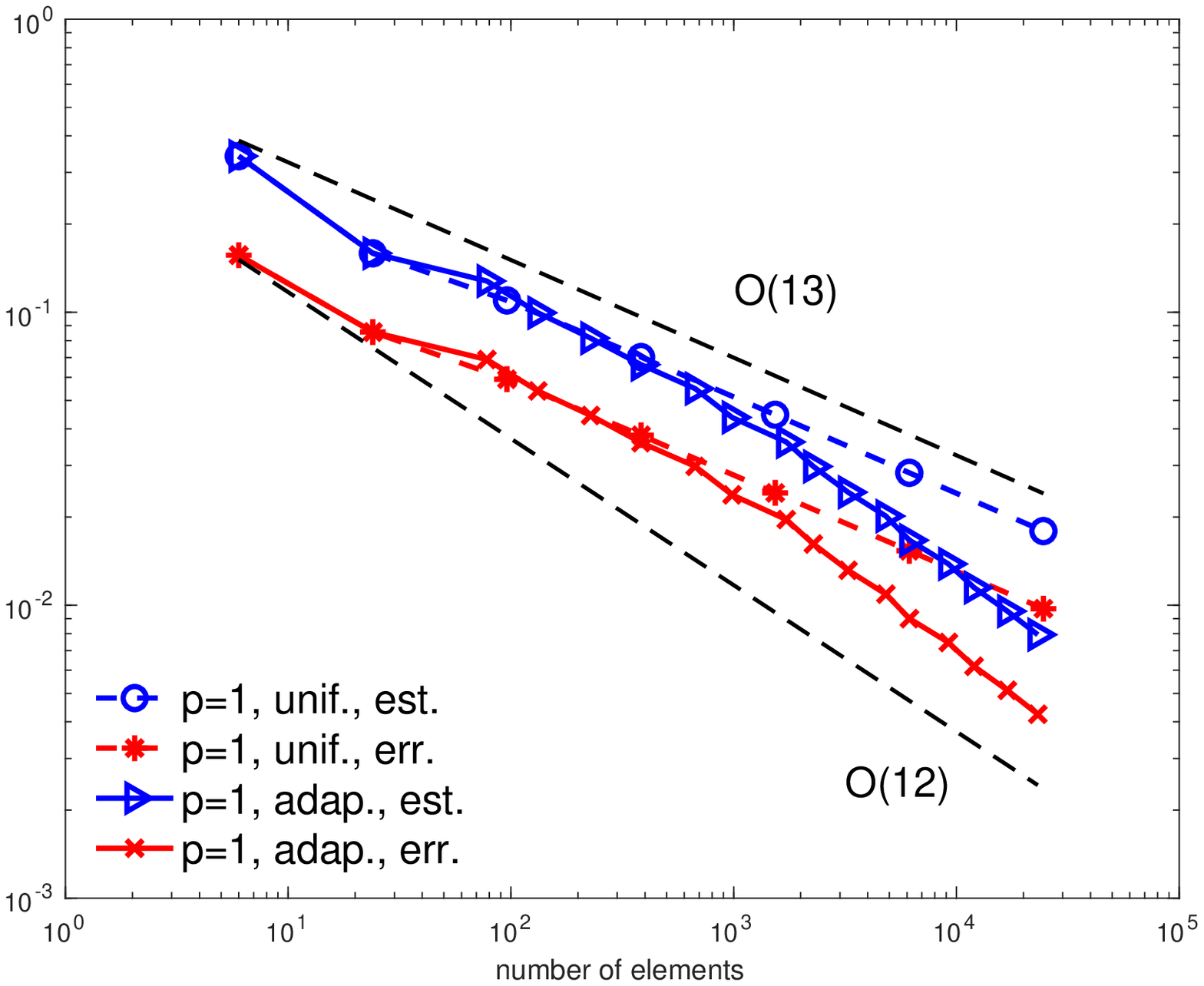}
%
%\vspace{2mm}
%
\includegraphics[width=0.32\textwidth,clip=true]{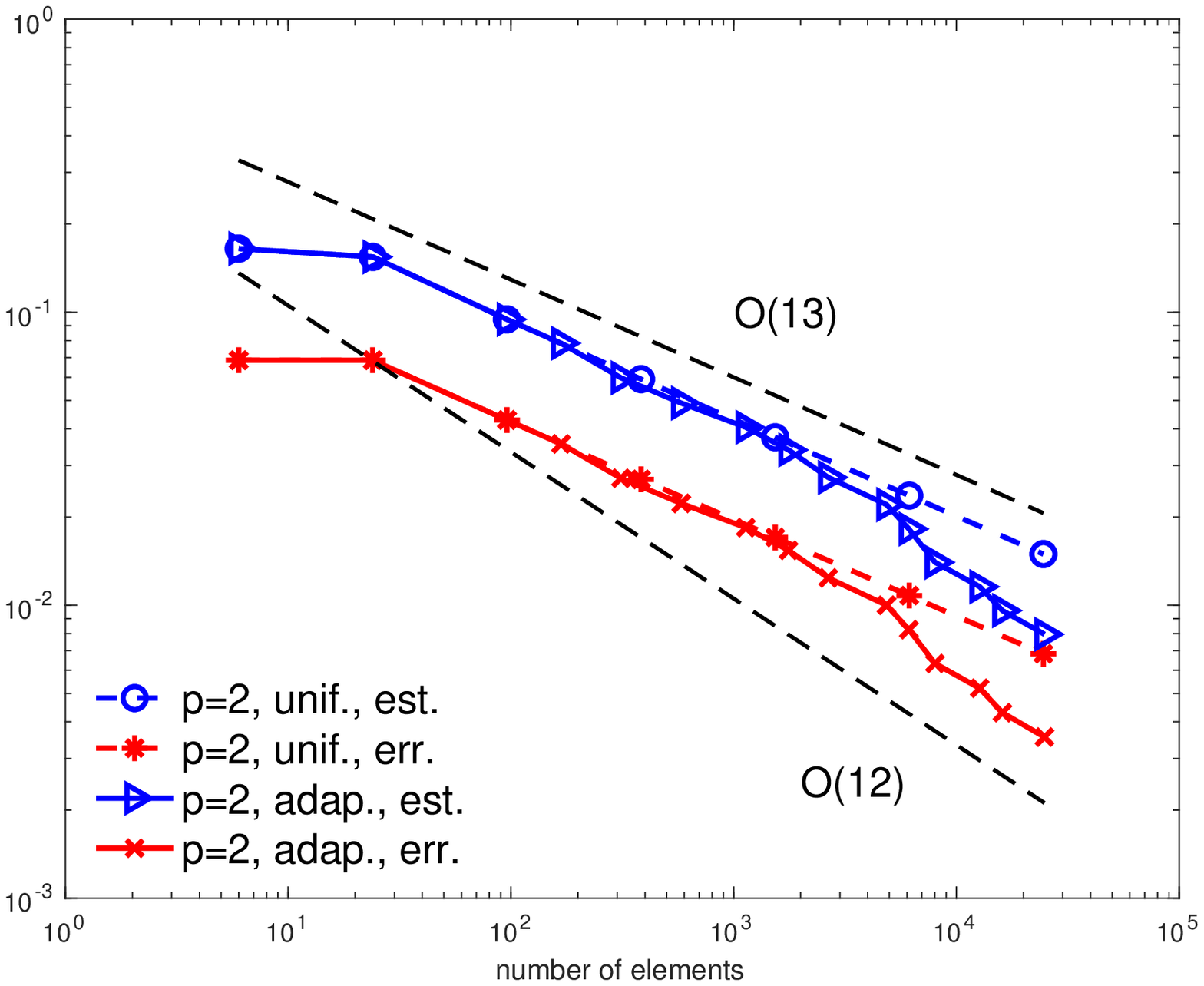}
%\qquad
%\includegraphics[width=0.45\textwidth]{FEM_2D/bend-Shape_p5}
\end{center}
\caption{Experiment with solution with edge singularities on cube
 of  Section~\ref{subsec:cube bem1}. 
Energy error $\norm{ \phi- \Phi_\ell}{\mathfrak{V}}$ and estimator $\eta_\ell$ of Algorithm~\ref{alg:bem algorithm} for hierarchical splines of degree $p\in\{0,1,2\}$ are plotted versus the number of elements $\#\TT_\ell$.
Uniform  and adaptive ($\theta=0.5$) refinement is considered.
}
\label{fig:cube_p} 
\end{figure}

 \begin{figure}[h!]%cube convergence 2
\psfrag{p=0, unif., est.}[l][l]{\fontsize{5pt}{6pt}\selectfont  $p=0$, unif., est.}
\psfrag{p=0, unif., err.}[l][l]{\fontsize{5pt}{6pt}\selectfont  $p=0$, unif., err.}
\psfrag{p=0, adap., est.}[l][l]{\fontsize{5pt}{6pt}\selectfont   $p=0$, adap., est.}
\psfrag{p=0, adap., err.}[l][l]{\fontsize{5pt}{6pt}\selectfont   $p=0$, adap., err.}
\psfrag{p=1, unif., est.}[l][l]{\fontsize{5pt}{6pt}\selectfont  $p=1$, unif., est.}
\psfrag{p=1, unif., err.}[l][l]{\fontsize{5pt}{6pt}\selectfont  $p=1$, unif., err.}
\psfrag{p=1, adap., est.}[l][l]{\fontsize{5pt}{6pt}\selectfont   $p=1$, adap., est.}
\psfrag{p=1, adap., err.}[l][l]{\fontsize{5pt}{6pt}\selectfont   $p=1$, adap., err.}
\psfrag{p=2, unif., est.}[l][l]{\fontsize{5pt}{6pt}\selectfont  $p=2$, unif., est.}
\psfrag{p=2, unif., err.}[l][l]{\fontsize{5pt}{6pt}\selectfont  $p=2$, unif., err.}
\psfrag{p=2, adap., est.}[l][l]{\fontsize{5pt}{6pt}\selectfont   $p=2$, adap., est.}
\psfrag{p=2, adap., err.}[l][l]{\fontsize{5pt}{6pt}\selectfont   $p=2$, adap., err.}
\psfrag{p=3, unif., est.}[l][l]{\fontsize{5pt}{6pt}\selectfont  $p=3$, unif., est.}
\psfrag{p=3, unif., err.}[l][l]{\fontsize{5pt}{6pt}\selectfont  $p=3$, unif., err.}
\psfrag{p=3, adap., est.}[l][l]{\fontsize{5pt}{6pt}\selectfont   $p=3$, adap., est.}
\psfrag{p=3, adap., err.}[l][l]{\fontsize{5pt}{6pt}\selectfont   $p=3$, adap., err.}
\psfrag{error and estimator}[c][c]{\fontsize{5pt}{6pt}\selectfont  error and estimator}
\psfrag{number of elements}[c][c]{\fontsize{5pt}{6pt}\selectfont  number of elements}
\psfrag{O(2.5)}[r][r]{\fontsize{5pt}{6pt}\selectfont $\mathcal{O}(N^{-5/2})$}
\psfrag{O(rh)}[l][l]{\fontsize{5pt}{6pt}\selectfont $\mathcal{O}(N^{-1})$}
\psfrag{O(13)}[l][l]{\fontsize{5pt}{6pt}\selectfont $\mathcal{O}(N^{-1/3})$}
\psfrag{O(12)}[r][r]{\fontsize{5pt}{6pt}\selectfont $\mathcal{O}(N^{-1/2})$}
\psfrag{O(1.5)}[r][r]{\fontsize{5pt}{6pt}\selectfont $\mathcal{O}(N^{-3/2})$}
\psfrag{O(2)}[r][r]{\fontsize{5pt}{6pt}\selectfont $\mathcal{O}(N^{-2})$}
\psfrag{O(34)}[r][r]{\fontsize{5pt}{6pt}\selectfont $\mathcal{O}(N^{-3/4})$}
\psfrag{O(54)}[r][r]{\fontsize{5pt}{6pt}\selectfont $\mathcal{O}(N^{-5/4})$}
\psfrag{O(74)}[r][r]{\fontsize{5pt}{6pt}\selectfont $\mathcal{O}(N^{-7/4})$}
\psfrag{O(92)}[r][r]{\fontsize{5pt}{6pt}\selectfont $\mathcal{O}(N^{-9/2})$}
\psfrag{O(1)}[l][l]{\fontsize{5pt}{6pt}\selectfont $\mathcal{O}(N^{-1})$}
\centering 
\includegraphics[width=0.32\textwidth]{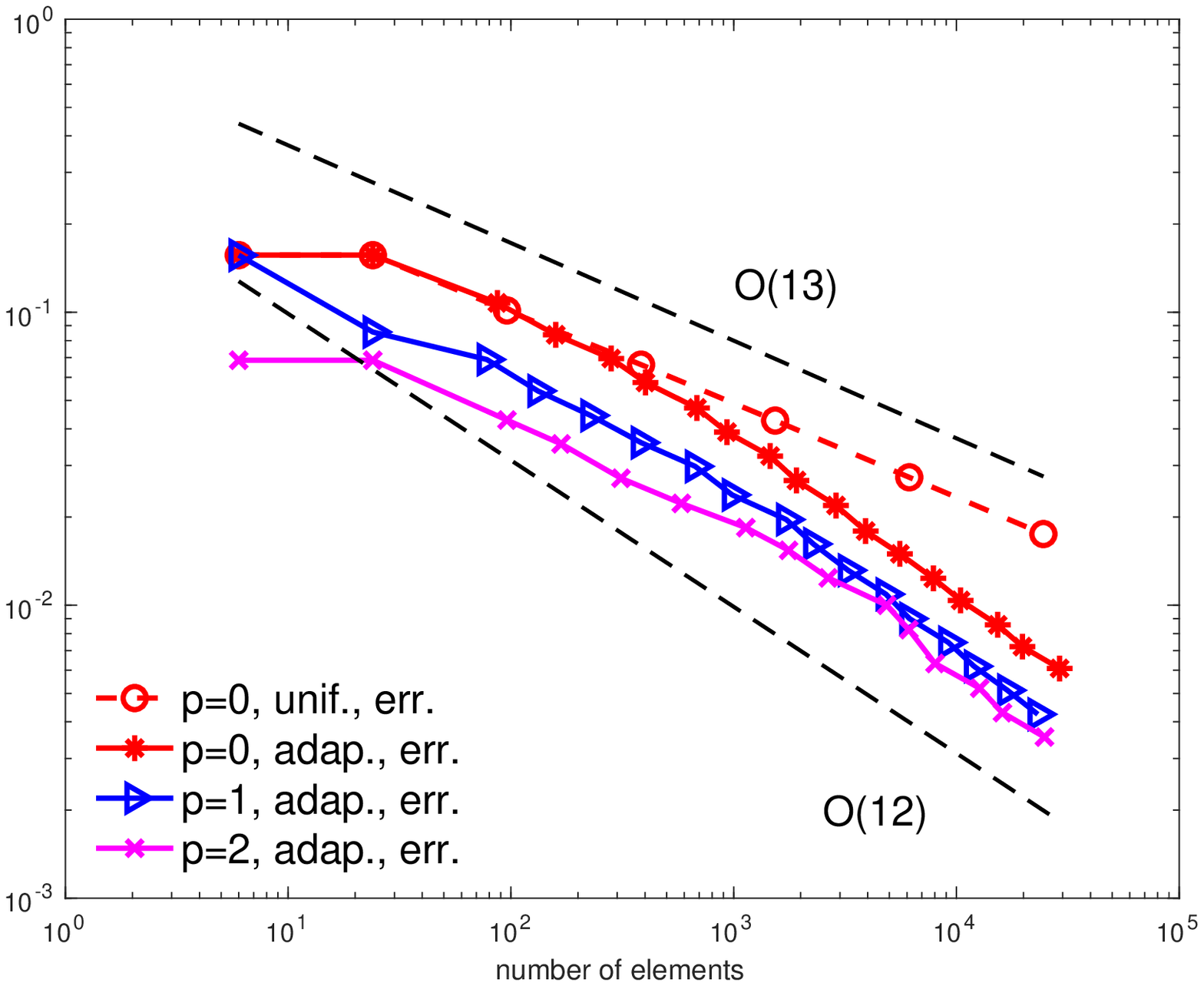}
\caption{Experiment with solution with edge singularities on cube of Section~\ref{subsec:cube bem1}. 
The energy errors $\norm{\phi- \Phi_\ell}{\mathfrak{V}}$ of Algorithm~\ref{alg:bem algorithm} for hierarchical splines of degree $p\in\{0,1,2\}$ are plotted versus the number of elements $\#\TT_\ell$.
Uniform (for $p=0$) and adaptive ($\theta=0.5$ for $p\in\{0,1,2\}$) refinement is considered.}
\label{fig:cube_pcomp} 
\end{figure}

\subsection{Nearly singular solution on quarter pipe}
\label{subsec:bendcube bem1}
We consider the quarter pipe 
\begin{align}
\Omega: = \big\{10^{-1}(r\cos(\beta),r\sin(\beta),z):~ r \in(1/2, 1)\wedge\beta\in(0,{\pi}/{2})\wedge z\in(0,1)\big\};
\end{align}
see Figure~\ref{fig:bendcube_mesh}.
We split the boundary $\Gamma$ into the six surfaces
\begin{align*}
\Gamma_1&:=\set{10^{-1}(\cos(\beta)/2,\sin(\beta)/2,z)}{\beta\in(0,{\pi}/{2})\wedge z\in(0,1)}\\
\Gamma_2&:=\set{10^{-1}(r,0,z)}{r\in(1/2,1)\wedge z\in(0,1)}\\
\Gamma_3&:=\set{10^{-1}(\cos(\beta),\sin(\beta),z)}{\beta\in(0,{\pi}/{2})\wedge z\in(0,1)}\\
\Gamma_4&:=\set{10^{-1}(0,r,z)}{r\in(1/2,1)\wedge z\in(0,1)}\\
\Gamma_5&:=\set{10^{-1}(r\cos(\beta),r\sin(\beta),0)}{r\in(1/2,1)\wedge\beta\in(0,{\pi}/{2})}\\
\Gamma_6&:=\set{10^{-1}(r\cos(\beta),r\sin(\beta),1)}{r\in(1/2,1)\wedge\beta\in(0,{\pi}/{2})}
\end{align*}
$\Gamma_1,\Gamma_3,\Gamma_5$, and $\Gamma_6$ can be parametrized by rational splines of degree $p_{1(\gamma,m)}:=2, p_{2(\gamma,m)}:=1$ corresponding to the knot vectors  $\widehat\KK_{1(\gamma,m)}:=(0,0,0,1,1,1), \widehat\KK_{2(\gamma,m)} := (0,0,1,1)$; see \cite[Chapter~8]{NURBSbook}.
The affine surfaces $\Gamma_2$ and $\Gamma_4$ can be parametrized by non-rational splines of degree $p_{1(\gamma,m)}:= p_{2(\gamma,m)}:=1$ corresponding to the knot vectors  $\widehat\KK_{1(\gamma,m)}:= \widehat\KK_{2(\gamma,m)} := (0,0,1,1)$; see \cite[Section~6.1]{igafem}.

\begin{figure}[h!]%quarter ring 
\psfrag{x1}[c][c]{\fontsize{5pt}{6pt}\selectfont $x_1$}
\psfrag{x2}[c][c]{\fontsize{5pt}{6pt}\selectfont $x_2$}
\psfrag{x3}[c][c]{\fontsize{5pt}{6pt}\selectfont $x_3$}

\begin{center}
\includegraphics[width=0.24\textwidth,clip=true]{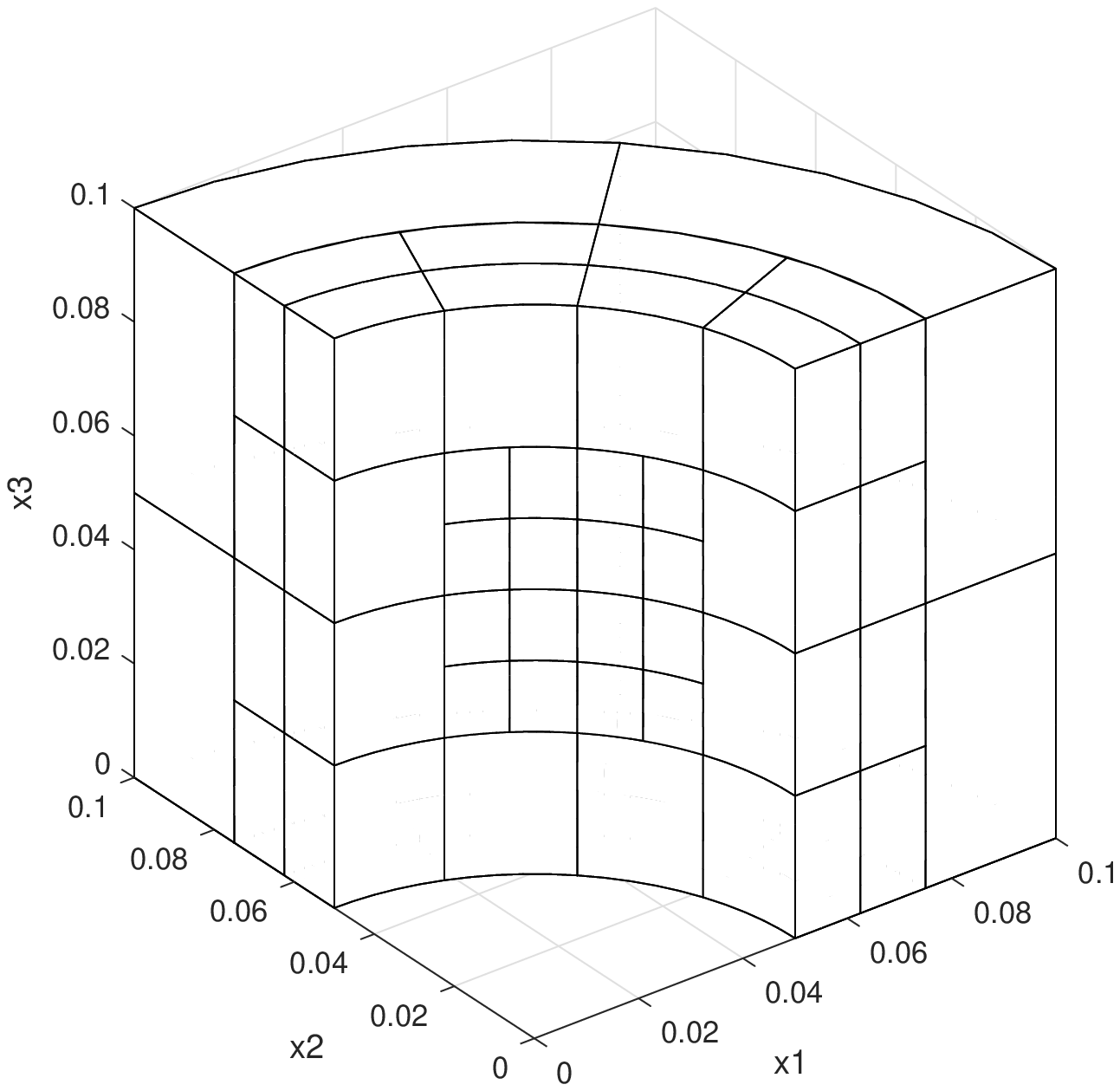}
\includegraphics[width=0.24\textwidth,clip=true]{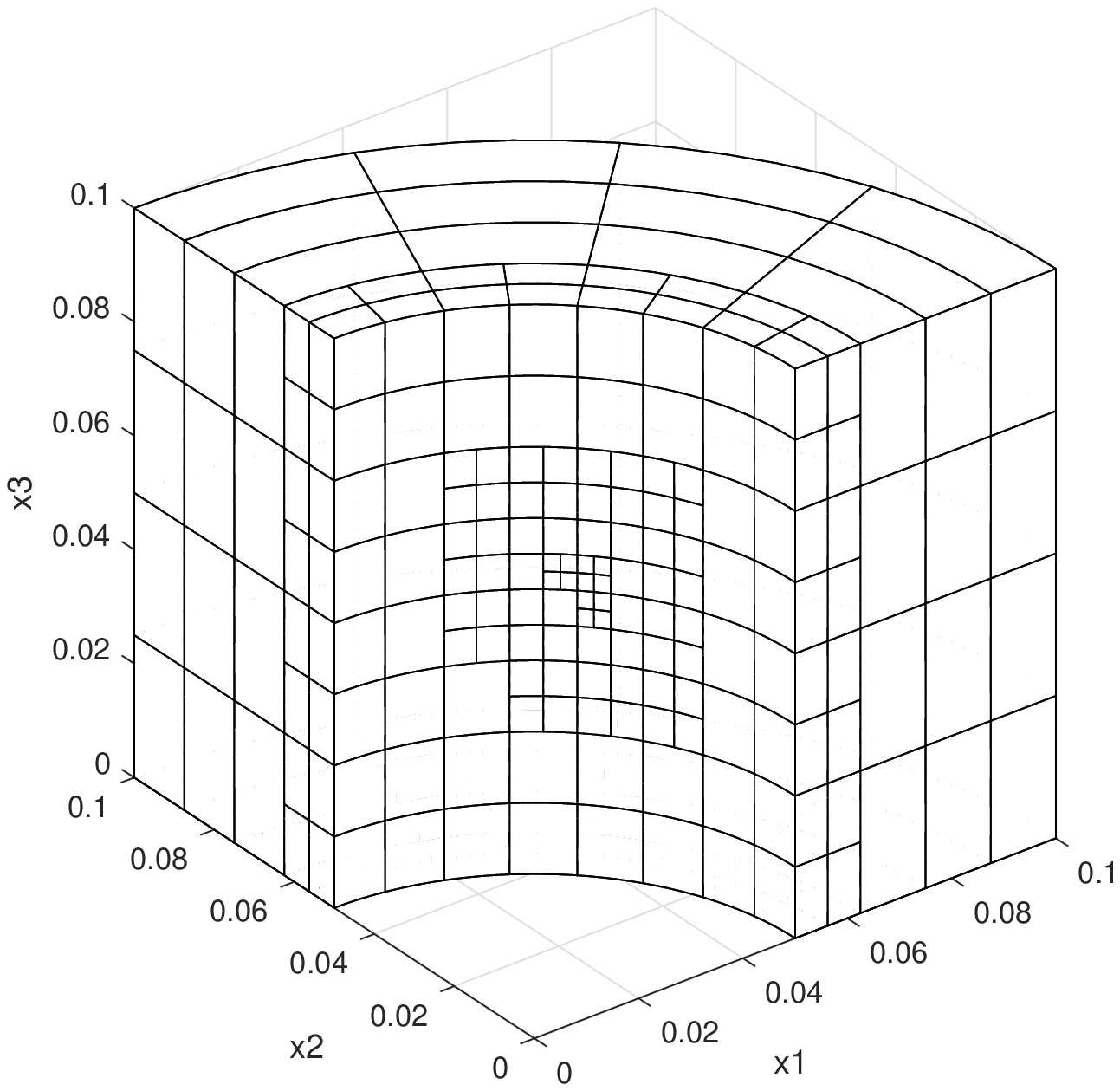}
%
%\vspace{2mm}
%
\includegraphics[width=0.24\textwidth,clip=true]{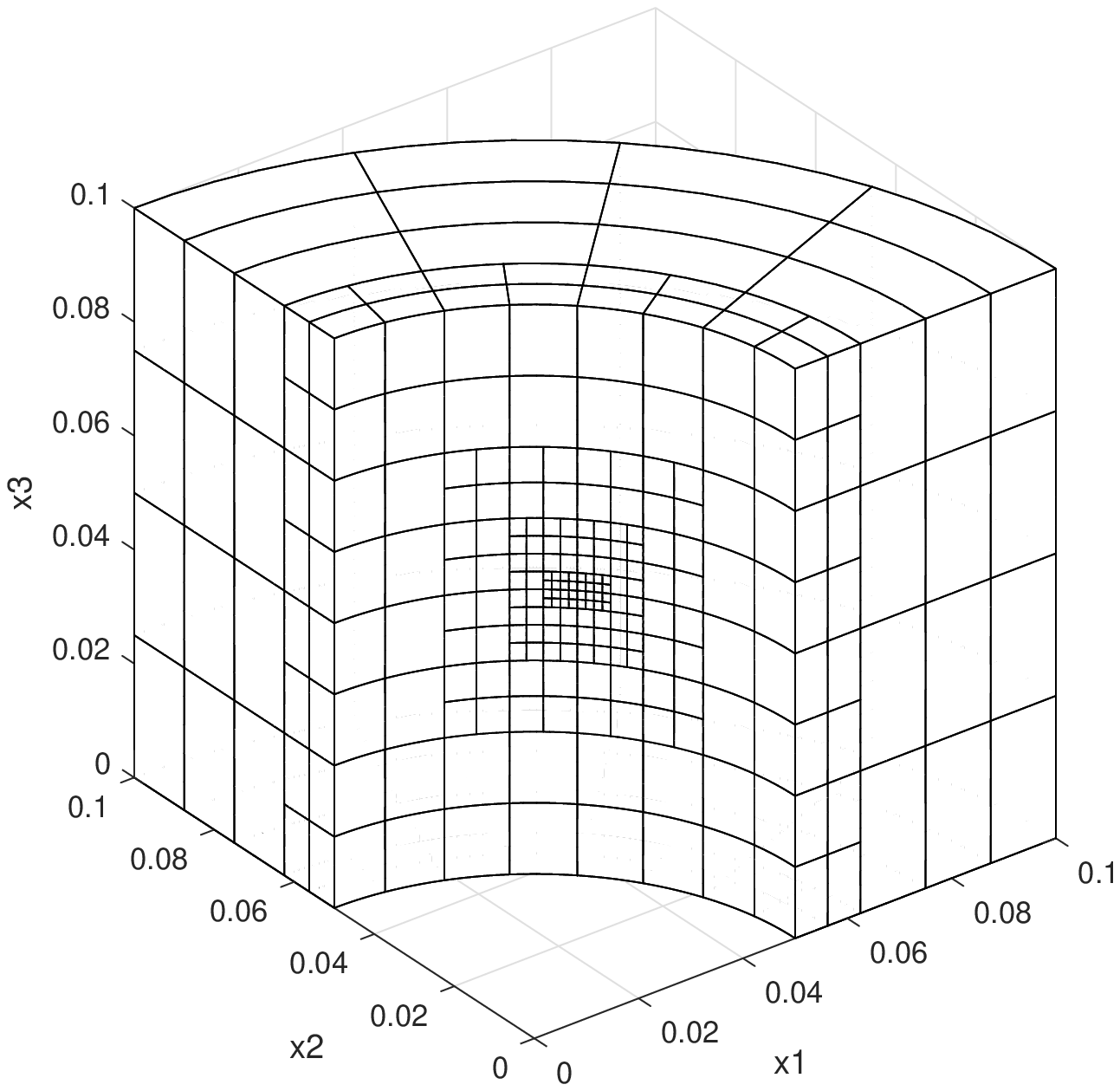}
\includegraphics[width=0.24\textwidth,clip=true]{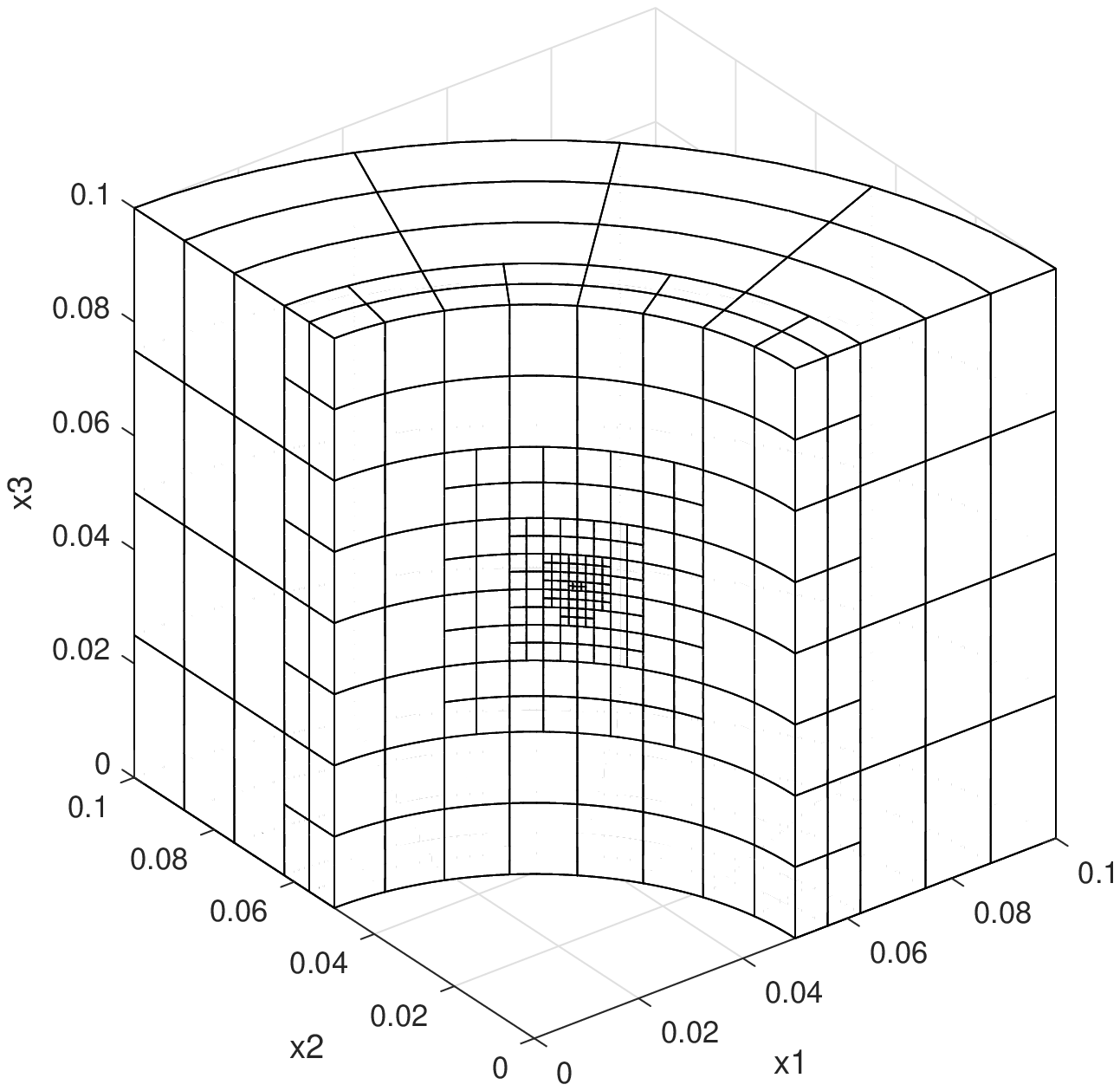}
\end{center}
%\begin{minipage}{0.32\textwidth}
%\begin{center}$\#\TT_4=265$\end{center}
%\end{minipage}
%\begin{minipage}{0.32\textwidth}
%\begin{center}$\#\TT_5=718$\end{center}
%\end{minipage}
%\begin{minipage}{0.32\textwidth}
%\begin{center}$\#\TT_{6}=1777$\end{center}
%\end{minipage}
\caption{Experiment with nearly singular solution on quarter pipe of Section~\ref{subsec:bendcube bem1}. 
Hierarchical meshes $\TT_4,\TT_7,\TT_{9},\TT_{10}$ generated by  Algorithm~\ref{alg:bem algorithm} (with $\theta=0.5$) for hierarchical splines of degree $p=1$. 
}
\label{fig:bendcube_mesh}
\end{figure}

% we choose 
% $p_{1(\gamma_1)}:=p_{2(\gamma)}=1$ and $\widehat\KK_{1(\gamma)} := (0,0,0.5,1,1),~\widehat\KK_{2(\gamma)}: = (0,0,1,1)$;  and $\widehat W_\gamma:=1$; see \cite[Section~6.2]{igafem}.
% Moreover, we choose the control points 
% \begin{align*}
%c^\gamma_{1,1} &= (0,0.5),~ c^\gamma_{2,1} = (0.5,0.5),~
%c^\gamma_{3,1} = (0.5,0),\\ c^\gamma_{1,2} &= (0,1),~ c^\gamma_{2,2} = (1,1),~
%c^\gamma_{3,2} = (1,0),
%\end{align*}
%and set all weights equal to $1$. 
We prescribe the exact solution of the interior Laplace--Dirichlet problem \eqref{eq:Laplace interior bem} as the shifted fundamental solution
\begin{equation}
u(x):=G(x-y_0)=\frac{1}{4\pi}\frac{1}{|x-y_0|}, 
\end{equation}
with $y_0:=10^{-1}(0.95\cdot 2^{-3/2},0.95\cdot 2^{-3/2},1/2)\in\R^3\setminus\overline\Omega$.
Although $u$ is smooth on $\overline\Omega$, it is nearly singular at the midpoint  $\widetilde y_0:=10^{-1}( 2^{-3/2},2^{-3/2},1/2)$ of $\Gamma_1$.
%in polar coordinates $(x,y)=r(\cos \beta,\sin \beta)$.
We consider the corresponding integral equation \eqref{eq:Symmy interior}.
The normal derivative $\phi=\partial_\nu u$ of $u$ reads 
\begin{equation}
\phi(x)=-\frac{1}{4\pi}\frac{x-y_0}{|x-y_0|^3}\cdot\nu(x).
%\begin{pmatrix} \cos(\beta)\cos\left(\tau\beta\right)+\sin(\beta)\sin\left(\tau\beta\right)\\ \sin(\beta)\cos\left(\tau\beta\right)-\cos(\beta)\sin\left(\tau\beta\right)\end{pmatrix}\cdot \nu(x,y) \cdot \tau \cdot r^{\tau-1}
\end{equation}
%and has a generic singularity at the origin.
%In Figure \ref{fig:pacman solution}, the solution $\phi$ is plotted over the parameter domain.
%The singularity is located at $t=1/2$ and two jumps are located at $t=1/3$ resp. $t=2/3$.

\begin{figure}[h!]%quarter ring convergence
\psfrag{p=0, unif., est.}[l][l]{\fontsize{5pt}{6pt}\selectfont $p=0$, unif., est.}
\psfrag{p=0, unif., err.}[l][l]{\fontsize{5pt}{6pt}\selectfont $p=0$, unif., err.}
\psfrag{p=0, adap., est.}[l][l]{\fontsize{5pt}{6pt}\selectfont  $p=0$, adap., est.}
\psfrag{p=0, adap., err.}[l][l]{\fontsize{5pt}{6pt}\selectfont  $p=0$, adap., err.}
\psfrag{p=1, unif., est.}[l][l]{\fontsize{5pt}{6pt}\selectfont $p=1$, unif., est.}
\psfrag{p=1, unif., err.}[l][l]{\fontsize{5pt}{6pt}\selectfont $p=1$, unif., err.}
\psfrag{p=1, adap., est.}[l][l]{\fontsize{5pt}{6pt}\selectfont  $p=1$, adap., est.}
\psfrag{p=1, adap., err.}[l][l]{\fontsize{5pt}{6pt}\selectfont  $p=1$, adap., err.}
\psfrag{p=2, unif., est.}[l][l]{\fontsize{5pt}{6pt}\selectfont $p=2$, unif., est.}
\psfrag{p=2, unif., err.}[l][l]{\fontsize{5pt}{6pt}\selectfont $p=2$, unif., err.}
\psfrag{p=2, adap., est.}[l][l]{\fontsize{5pt}{6pt}\selectfont  $p=2$, adap., est.}
\psfrag{p=2, adap., err.}[l][l]{\fontsize{5pt}{6pt}\selectfont  $p=2$, adap., err.}
\psfrag{p=3, unif., est.}[l][l]{\fontsize{5pt}{6pt}\selectfont $p=3$, unif., est.}
\psfrag{p=3, unif., err.}[l][l]{\fontsize{5pt}{6pt}\selectfont $p=3$, unif., err.}
\psfrag{p=3, adap., est.}[l][l]{\fontsize{5pt}{6pt}\selectfont  $p=3$, adap., est.}
\psfrag{p=3, adap., err.}[l][l]{\fontsize{5pt}{6pt}\selectfont  $p=3$, adap., err.}
\psfrag{error and estimator}[c][c]{\fontsize{5pt}{6pt}\selectfont error and estimator}
\psfrag{number of elements}[c][c]{\fontsize{5pt}{6pt}\selectfont number of elements}
\psfrag{O(2.5)}[r][r]{\fontsize{5pt}{6pt}\selectfont $\mathcal{O}(N^{-5/2})$}
\psfrag{O(rh)}[l][l]{\fontsize{5pt}{6pt}\selectfont $\mathcal{O}(N^{-1})$}
\psfrag{O(1)}[l][l]{\fontsize{5pt}{6pt}\selectfont $\mathcal{O}(N^{-1})$}
\psfrag{O(1.5)}[r][r]{\fontsize{5pt}{6pt}\selectfont $\mathcal{O}(N^{-5/2})$}
\psfrag{O(1.5)}[r][r]{\fontsize{5pt}{6pt}\selectfont $\mathcal{O}(N^{-3/2})$}
\psfrag{O(2)}[r][r]{\fontsize{5pt}{6pt}\selectfont $\mathcal{O}(N^{-2})$}
\psfrag{O(34)}[r][r]{\fontsize{5pt}{6pt}\selectfont $\mathcal{O}(N^{-3/4})$}
\psfrag{O(54)}[r][r]{\fontsize{5pt}{6pt}\selectfont $\mathcal{O}(N^{-5/4})$}
\psfrag{O(74)}[c][c]{\fontsize{5pt}{6pt}\selectfont $\mathcal{O}(N^{-7/4})$}
\psfrag{O(92)}[r][r]{\fontsize{5pt}{6pt}\selectfont $\mathcal{O}(N^{-9/2})$}
\psfrag{O(1)}[l][l]{\fontsize{5pt}{6pt}\selectfont $\mathcal{O}(N^{-1})$}
\begin{center}
\includegraphics[width=0.32\textwidth,clip=true]{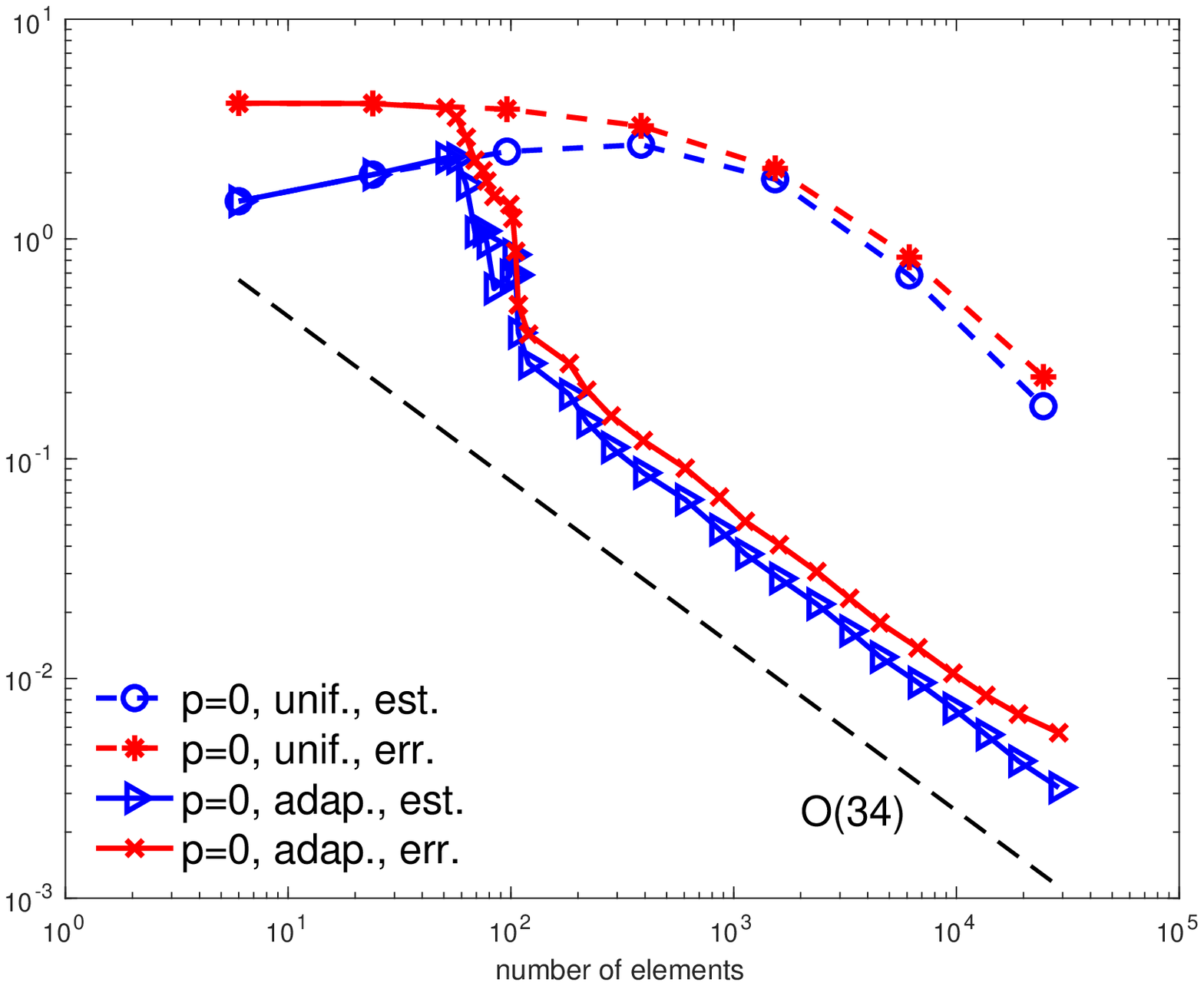}
%
%\vspace{2mm}
%
\includegraphics[width=0.32\textwidth,clip=true]{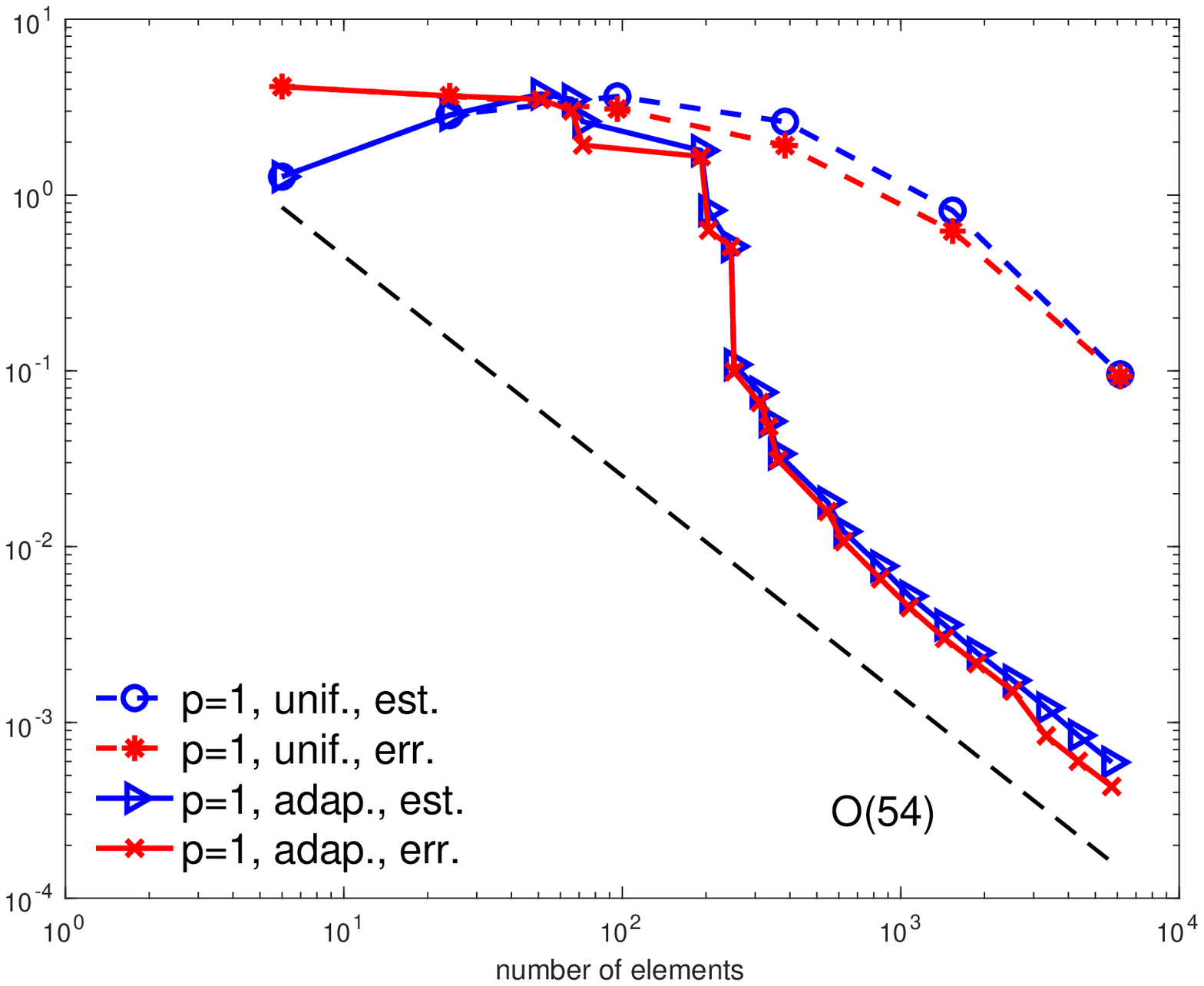}
%
%\vspace{2mm}
%
\includegraphics[width=0.32\textwidth,clip=true]{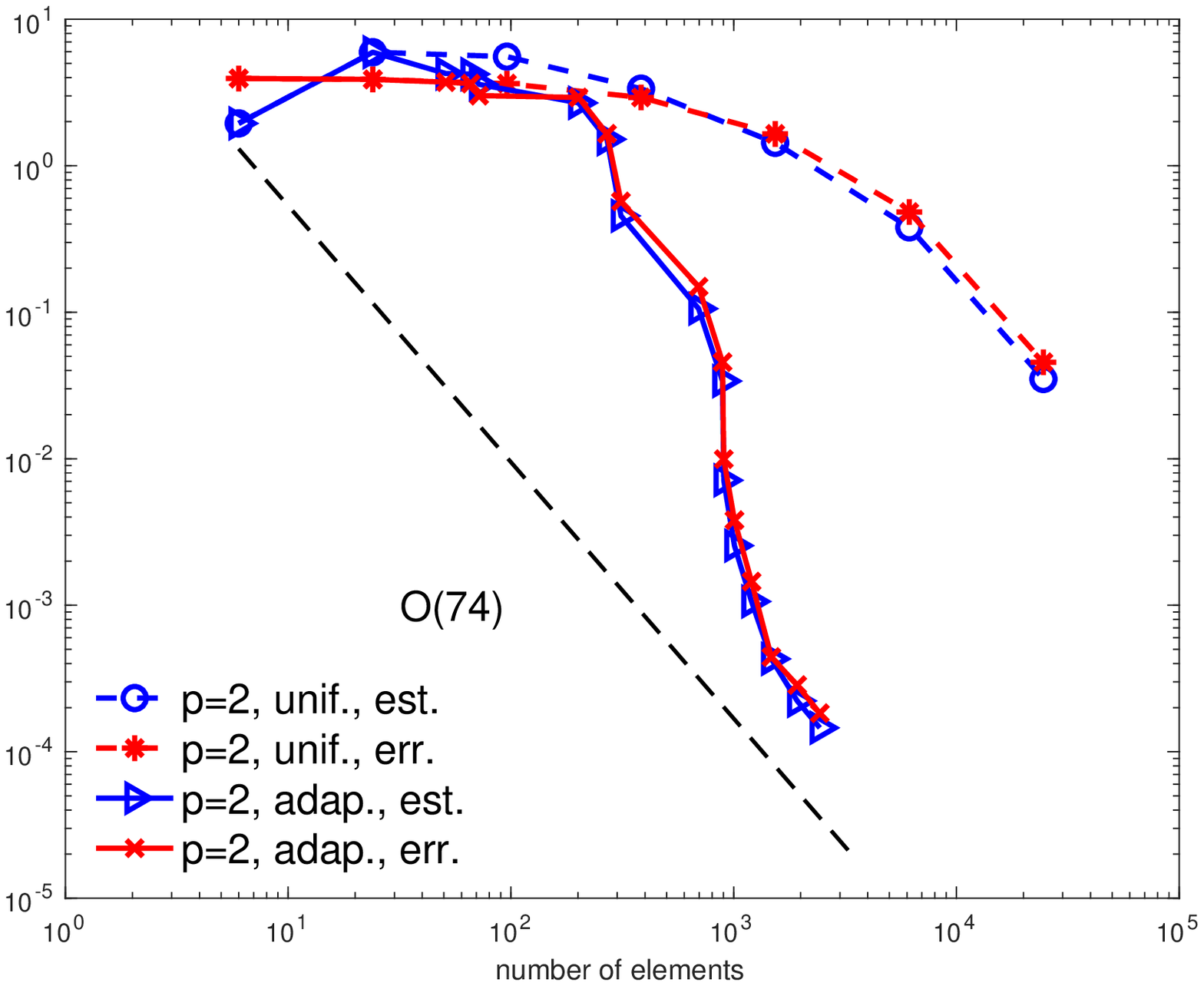}
%\qquad
%\includegraphics[width=0.45\textwidth]{FEM_2D/bend-Shape_p5}
\end{center}
\caption{Experiment with nearly singular solution on quarter pipe of  Section~\ref{subsec:bendcube bem1}. 
Energy error $\norm{\phi-\Phi_\ell}{\mathfrak{V}}$ and estimator $\eta_\ell$ of Algorithm~\ref{alg:bem algorithm} for hierarchical splines of degree $p\in\{0,1,2\}$ are plotted versus the number of elements $\#\TT_\ell$.
Uniform and adaptive ($\theta=0.5$) refinement is considered.
}
\label{fig:bendcube_p} 
\end{figure}

We consider polynomial degrees $p\in\{0,1,2\}$.
For the initial ansatz space with spline degree $p_{1,m}:=p_{2,m}:=p$ for all $m\in\{1,\dots 6\}$, we choose the initial knot vectors $\widehat\KK_{1(0,m)}:=\widehat\KK_{2(0,m)}:=(0,\dots 0,1,\dots,1)$ for all $m\in\{1,\dots,6\}$, where the multiplicity of $0$ and $1$ is $p+1$.
%However, we raise the multiplicity of the inner knot $0.5$ in $\widehat\KK^0_1$ to the maximal multiplicity, that is $p_1$. 
%As a consequence, the ansatz functions are only continuous at $\gamma(\{0.5\}\times[0,1])$,  but not continuously differentiable.
We choose the parameters of Algorithm~\ref{alg:bem algorithm} as $\theta=0.5$ and $\const{min}=1$, where we use the refinement strategy of Remark~\ref{rem:rational main bem} {\rm(c)} in the lowest-order case $p=0$.
For comparison, we also consider uniform refinement, where we mark all elements in each step, i.e., $\MM_\ell=\TT_\ell$ for all $\ell\in\N_0$. 
%This leads to uniform bisection of all elements.
%We compare uniform  ($\theta=1$) and adaptive ($\theta=0.4$) mesh-refinement.
 In Figure~\ref{fig:bendcube_mesh}, one can see some adaptively generated hierarchical meshes.
 In Figure~\ref{fig:bendcube_p} and Figure~\ref{fig:bendcube_pcomp}, 
we plot   the energy error $\norm{\phi-\Phi_\ell}{\mathfrak{V}}$ and the error estimator $\eta_\ell$ against the  number of elements $\#\TT_\ell$.  
All values are plotted in a double-logarithmic scale and the experimental convergence rates are visible as the slope of the corresponding curves.
 In all cases, the lines of the error and the error estimator are parallel, which numerically indicates reliability 
 and efficiency.
Since the solution $\phi$ is smooth, the uniform and the adaptive approach both lead to the optimal asymptotic convergence rate $\mathcal{O}((\#\TT_\ell)^{-3/4-p/2})$. 
However, $\phi$ is nearly singular at $\widetilde y_0$, which is why adaptivity yields a much better multiplicative constant.
%The uniform approach leads to the suboptimal convergence rate $\mathcal{O}((\#\TT_\ell)^{-1/3})$, since the reentrant corner at  $(0.5,0.5)$ causes a generic singularity of the solution $u$.
% However, the adaptive strategy recovers the optimal convergence rate $\mathcal{O}((\# \TT_\ell)^{-p/2})$ with $p:=p_1=p_2$.

\begin{figure}[h!]%Quarter ring convergence 2
\psfrag{p=0, unif., est.}[l][l]{\fontsize{5pt}{6pt}\selectfont  $p=0$, unif., est.}
\psfrag{p=0, unif., err.}[l][l]{\fontsize{5pt}{6pt}\selectfont  $p=0$, unif., err.}
\psfrag{p=0, adap., est.}[l][l]{\fontsize{5pt}{6pt}\selectfont   $p=0$, adap., est.}
\psfrag{p=0, adap., err.}[l][l]{\fontsize{5pt}{6pt}\selectfont   $p=0$, adap., err.}
\psfrag{p=1, unif., est.}[l][l]{\fontsize{5pt}{6pt}\selectfont  $p=1$, unif., est.}
\psfrag{p=1, unif., err.}[l][l]{\fontsize{5pt}{6pt}\selectfont  $p=1$, unif., err.}
\psfrag{p=1, adap., est.}[l][l]{\fontsize{5pt}{6pt}\selectfont   $p=1$, adap., est.}
\psfrag{p=1, adap., err.}[l][l]{\fontsize{5pt}{6pt}\selectfont   $p=1$, adap., err.}
\psfrag{p=2, unif., est.}[l][l]{\fontsize{5pt}{6pt}\selectfont  $p=2$, unif., est.}
\psfrag{p=2, unif., err.}[l][l]{\fontsize{5pt}{6pt}\selectfont  $p=2$, unif., err.}
\psfrag{p=2, adap., est.}[l][l]{\fontsize{5pt}{6pt}\selectfont   $p=2$, adap., est.}
\psfrag{p=2, adap., err.}[l][l]{\fontsize{5pt}{6pt}\selectfont   $p=2$, adap., err.}
\psfrag{p=3, unif., est.}[l][l]{\fontsize{5pt}{6pt}\selectfont  $p=3$, unif., est.}
\psfrag{p=3, unif., err.}[l][l]{\fontsize{5pt}{6pt}\selectfont  $p=3$, unif., err.}
\psfrag{p=3, adap., est.}[l][l]{\fontsize{5pt}{6pt}\selectfont   $p=3$, adap., est.}
\psfrag{p=3, adap., err.}[l][l]{\fontsize{5pt}{6pt}\selectfont   $p=3$, adap., err.}
\psfrag{error and estimator}[c][c]{\fontsize{5pt}{6pt}\selectfont  error and estimator}
\psfrag{number of elements}[c][c]{\fontsize{5pt}{6pt}\selectfont  number of elements}
\psfrag{O(2.5)}[r][r]{\fontsize{5pt}{6pt}\selectfont $\mathcal{O}(N^{-5/2})$}
\psfrag{O(rh)}[l][l]{\fontsize{5pt}{6pt}\selectfont $\mathcal{O}(N^{-1})$}
\psfrag{O(1)}[l][l]{\fontsize{5pt}{6pt}\selectfont $\mathcal{O}(N^{-1})$}
\psfrag{O(1.5)}[r][r]{\fontsize{5pt}{6pt}\selectfont $\mathcal{O}(N^{-5/2})$}
\psfrag{O(1.5)}[r][r]{\fontsize{5pt}{6pt}\selectfont $\mathcal{O}(N^{-3/2})$}
\psfrag{O(2)}[r][r]{\fontsize{5pt}{6pt}\selectfont $\mathcal{O}(N^{-2})$}
\psfrag{O(34)}[c][tc]{\fontsize{5pt}{6pt}\selectfont $\mathcal{O}(N^{-3/4})$}
\psfrag{O(54)}[l][l]{\fontsize{5pt}{6pt}\selectfont $\mathcal{O}(N^{-5/4})$}
\psfrag{O(74)}[c][c]{\fontsize{5pt}{6pt}\selectfont $\mathcal{O}(N^{-7/4})$}
\psfrag{O(92)}[r][r]{\fontsize{5pt}{6pt}\selectfont $\mathcal{O}(N^{-9/2})$}
\psfrag{O(1)}[l][l]{\fontsize{5pt}{6pt}\selectfont $\mathcal{O}(N^{-1})$}
\centering 
\includegraphics[width=0.32\textwidth]{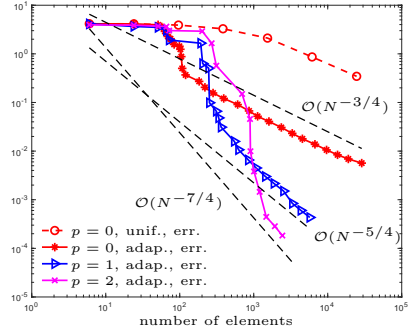}
\caption{Experiment with nearly singular solution on quarter pipe of Section~\ref{subsec:bendcube bem1}. 
The energy errors $\norm{\phi- \Phi_\ell}{\mathfrak{V}}$ of Algorithm~\ref{alg:bem algorithm} for hierarchical splines of degree $p\in\{0,1,2\}$ are plotted versus the number of elements $\#\TT_\ell$.
Uniform (for $p=2$) and adaptive ($\theta=0.5$ for $p\in\{0,1,2\}$) refinement is considered.}
\label{fig:bendcube_pcomp} 
\end{figure}

%% file: 03_axioms.tex
% !TEX encoding = MacOSRoman
% !TEX root = igabem3d.tex

\section{Proof of Theorem~\ref{thm:main bem}}\label{sec:axioms revisited bem}\label{sec:abstract setting bem}
In \cite[Section~3]{gp20}, we have identified the crucial properties of the underlying meshes \revision{(see~\eqref{M:patch bem}--\eqref{M:semi bem} below)}, the mesh-refinement \revision{(see~\eqref{R:sons bem}--\eqref{R:reduction bem} below)}, as well as the boundary element spaces \revision{(see~\eqref{S:inverse bem}--\eqref{S:stab bem} below)} which ensure that the residual error estimator fits into the general framework of \cite{axioms} and which hence guarantee \revision{reliability~\ref{eq:reliable bem} and linear} convergence \revision{at optimal rate~\eqref{eq:linear bem}--\eqref{eq:optimal bem}; see \cite[Theorem~3.4]{gp20}.}
The goal of this section is to recapitulate these properties and \revision{subsequently verify them for the present isogeometric setting};
% the corresponding main result \cite[Theorem~3.4]{gp20}
 see also \cite[Section~5.2--5.3]{diss} for more details. 
\revision{In Section~\ref {sec:rational main proof bem}, we further briefly comment on how to verify the mentioned properties for rational hierarchical splines; see Remark~\ref{rem:rational main bem}.}

%\newpage

\subsection{Meshes}\label{subsec:boundary discrete bem}
\revision{In \cite{gp20}, we even considered general \textit{meshes}} $\TT_\coarse$ of the boundary $\Gamma=\partial\Omega$ of the bounded Lipschitz domain $\Omega\subset\R^d$ in the following sense:
\begin{enumerate}[\rm (i)]
\item $\TT_\coarse$ is a finite set of compact Lipschitz domains on $\Gamma$, i.e., each element $T$ has the form $T=\gamma_T({\widehat{T}})$, where ${\widehat T}$ is a compact\footnote{A compact Lipschitz domain is the closure of a bounded Lipschitz domain.
For $d=2$, it is the finite union of compact intervals with non-empty interior.} Lipschitz domain in $\R^{d-1}$ and $\gamma_T:{ \widehat T}\to T$ is bi-Lipschitz.
\item  $\TT_\coarse$ covers $\Gamma$, i.e., $\Gamma = \bigcup_{T\in\TT_\coarse}{T}$.
\item For all $T,T'\in\TT_\coarse$ with $T\neq T'$, the intersection $T\cap T'$ has $(d-1)$-dimensional Hausdorff measure  zero.
\end{enumerate}
\revision{Clearly, any admissible hierarchical mesh $\TT_\coarse\in\T$ is a mesh in this sense, where 
\begin{align}
\widehat T:=\gamma_m^{-1}(T)\quad \text{and}\quad \gamma_T:=\gamma_m|_{\widehat T} \quad \text{for } T\in\TT_{\coarse,m} \text{ with }m\in\{1,\dots,M\}. 
\end{align}}
In order to ease notation, we introduce for \revision{\emph{any} mesh} $\TT_\coarse$ the corresponding \textit{mesh-width function} 
\begin{align}
h_\coarse\in L^\infty(\Gamma)\quad\text{with}\quad h_\coarse|_T=h_T:=|T|^{1/(d-1)}\text{ for all }T\in\TT_\coarse.
\end{align}
 For  $\omega\subseteq \Gamma$, we define the patches of order $q\in\N_0$ inductively by
\begin{align}
 \pi_\bullet^0(\omega) := \omega,
 \quad 
 \pi_\bullet^q(\omega) := \bigcup\set{T\in\TT_\bullet}{ {T}\cap \pi_\bullet^{q-1}(\omega)\neq \emptyset}.
\end{align}
The corresponding set of elements is
\begin{align}\label{eq:patch defined}
 \Pi_\bullet^q(\omega) := \set{T\in\TT_\bullet}{ {T} \subseteq \pi_\bullet^q(\omega)},
 \quad\text{i.e.,}\quad
 \pi_\bullet^q(\omega) = \bigcup\Pi_\bullet^q(\omega) \text{ for } q>0.
\end{align}
To abbreviate notation, we set $\pi_\bullet(\omega) := \pi_\bullet^1(\omega)$ and $\Pi_\bullet(\omega) := \Pi_\bullet^1(\omega)$.
If $\omega=\{z\}$ for some $z\in\Gamma$, we write  $\pi^q_\bullet(z):=\pi^q_\bullet(\{z\})$ and $\Pi_\bullet^q(z) := \Pi_\bullet^q(\{z\})$, where we skip the index  for $q=1$ as before.
  For $\SS\subseteq\TT_\bullet$, we define $\pi_\bullet^q(\SS):=\pi_\bullet^q(\bigcup\SS)$
and $\Pi_\bullet^q(\SS):=\Pi_\bullet^q(\bigcup\SS)$, and the superscript is omitted for $q=1$.

\revision{In the following Sections~\eqref{sec:patch bem}--\eqref{sec:semi bem},} we  \revision{show} the existence of constants $\const{patch},$ $\const{locuni},$ $\const{shape},$ $\const{cent},$ $\const{semi}>0$ such that the following \revision{properties} are satisfied for all \revision{admissible hierarchical meshes} $\TT_\bullet\in\T$:
\begin{enumerate}[(i)]
\renewcommand{\theenumi}{M\arabic{enumi}}
\bf\item\rm\label{M:patch bem}
{\bf Bounded   element patch:} 
For all $T\in\TT_\bullet$, there holds that 
\begin{align*}
\#\Pi_\bullet(T)\le \const{patch},
\end{align*}
i.e., the number of elements in a patch is uniformly bounded.
\bf\item\rm\label{M:locuni bem}
{\bf Local quasi-uniformity:}
For all $T\in\TT_\coarse$, there holds  that
\begin{align*}
{\diam(T)}/{\diam(T')}\le \const{locuni} \quad\text{for all }T'\in\Pi_\coarse(T),
\end{align*}
i.e., neighboring elements have comparable diameter.
\bf\item\rm\label{M:shape bem}
{\bf Shape-regularity:} 
For all  $T\in\TT_\bullet$, there holds that
\begin{align*}
{\diam(T)}/{h_T}\le \const{shape}.
\end{align*}
\bf\item\rm\label{M:cent bem}
{\bf Patch centered elements:} 
For all $T\in\TT_\bullet$, there holds\footnote{We use the convention $\dist(T,\emptyset):=\diam(\Gamma)$.} that
\begin{align*}
\diam(T)\le C_{\rm cent} \, \dist(T,\Gamma\setminus\pi_\bullet(T)),
\end{align*}
 i.e., each element lies essentially in the center of its patch.
\bf\item\rm\label{M:semi bem}
{\bf Local seminorm estimate:}
For all $z\in\Gamma$ and $v\in H^1(\Gamma)$, there holds that
 \begin{align*}
|v|_{H^{1/2}(\pi_{\bullet}(z))}\le \const{semi} \,\diam(\pi_{\bullet}(z))^{1/2}| v|_{H^1(\pi_{\bullet}(z))}.
\end{align*}
\end{enumerate}

The following proposition shows that \eqref{M:semi bem} is actually always satisfied.
However, in general the  multiplicative constant  depends  on the shape of the point patches.
The proof is inspired by \cite[Proposition 2.2]{hitchhiker}, where an analogous assertion for norms instead of seminorms is found.
For $d=2$,  we have already shown the assertion in the recent own work \cite[Lemma~4.5]{resigabem}.
For polyhedral domains $\Omega$ with triangular meshes, it is  proved in \cite[Proposition~3.3]{faerconv} via interpolation techniques. 
Adapting the arguments of \cite[Lemma~4.5]{resigabem}, a detailed proof for our setting is found in \cite[Proposition~5.2.2]{diss}.

\begin{proposition}\label{prop:semi estimate}
Let $\widehat \omega\subset\R^{d-1}$ be a bounded and  connected Lipschitz domain and  $\gamma_\omega:\widehat \omega \to\omega\subseteq\Gamma$ bi-Lipschitz.
In particular, there exists a constant $\const{lipref}>0$  
such that
\begin{align}\label{eq:Clipref}
C_{\rm lipref}^{-1} |s-t|\le \frac{|\gamma_\omega(s)-\gamma_\omega(t)|}{\diam(\omega)}\le C_{\rm lipref} |s-t|\quad\text{for all }s,t\in\widehat\omega.
\end{align}
Then, there exists a  constant $\const{semi}(\widehat\omega)>0$ such that 
\begin{align}\label{eq:local estimate}
|v|_{H^{1/2}(\omega)}\le \const{semi}(\widehat\omega) \,\diam(\omega)^{1/2}| v|_{H^1(\omega)}\quad\text{for all }v\in H^1(\Gamma).
\end{align}
The constant $\const{semi}(\widehat\omega)>0$  depends only on the dimension $d$, the set $\widehat\omega$, and $\const{lipref}$. \hfill$\square$
\end{proposition}

\subsubsection{Verification of  (\ref{M:patch bem})}\label{sec:patch bem}
Let $\TT_\bullet\in\T$ and $T\in\TT_\bullet$.
We split the patch as follows
\begin{align*}
\Pi_\bullet(T)&=\bigcup_{m=1}^M (\Pi_\bullet(T)\cap \TT_{\bullet,m})
\\
&\subseteq\bigcup_{m=1}^M\set{\Pi_\bullet(E)\cap \TT_{\bullet,m}}{E \text{ \revision{is lower-dimensional hyperrectangle} of }T}.
\end{align*}
If $\Pi_\bullet(E)\cap \TT_{\bullet,m}\neq \emptyset$, there exists $T'\in \TT_{\bullet,m}$ with $\emptyset\neq E\cap T'\subseteq T\cap T'$.
By admissibility, this implies 
that $E$ is even a common (transformed) lower-dimensional hyperrectangle of $T$ and $T'$.
In particular, this leads to $\Pi_\bullet(E)\cap \TT_{\bullet,m}\subseteq\Pi_\bullet(T')\cap\TT_{\bullet,m}=\Pi_{\bullet,m}(T')$, where $\Pi_{\bullet,m}(T'):=\set{\gamma_m(\widehat T)}{\widehat T\in\Pi_{\coarse,m}(\widehat T')}$.
Since $\TT_{\coarse,m}\in \T_{m}$, Remark~\ref{rem:Nu for p0} gives that $\#\Pi_{\bullet,m}(T')\lesssim 1$.
Altogether, we derive that  $\#\Pi_\bullet(T)\le \const{patch}$ with a constant $\const{patch}$ which depends only on $d$ and $M$.

\subsubsection{Verification of (\ref{M:locuni bem})}
Let $\TT_\bullet\in\T$ and $T,T'\in\TT_\bullet$ with $T\cap T'\neq\emptyset$.
Let $m,m'\in\{1,\dots,M\}$ with $m\neq m'$, $T\subseteq\Gamma_m$ and $T'\subseteq \Gamma_{m'}$. 
If $m=m'$, then  
 Remark~\ref{rem:Nu for p0} gives that   $\diam (\widehat T)\simeq\diam(\widehat T')$ for $\widehat T=\gamma_m^{-1}(T)$ and $\widehat T'=\gamma_m^{-1}(T')$. 
Bi-Lipschitz continuity of $\gamma_m$ yields that $\diam(T)\simeq\diam(T')$.
If $m\neq m'$, then admissibility of $\TT_\coarse$ and Remark~\ref{rem:admissible bem} give that $\level( T)=\level( T')$.
This  implies that $\diam( \widehat T)\simeq\diam( \widehat T')$ for $\widehat T=\gamma_{m'}^{-1}(T)$ and $\widehat T'=\gamma_m^{-1}(T')$,  and thus $\diam(T)\simeq\diam(T')$ due to bi-Lipschitz continuity of $\gamma_m$ and $\gamma_{m'}$.
This concludes local quasi-uniformity \eqref{M:locuni bem}, where the constant $\const{locuni}$ depends only the dimension $d$, the constant $\const{\gamma}$, and the initial meshes $\widehat \TT_{0,m}$ for $m\in\{1,\dots,M\}$.

\subsubsection{Verification of (\ref{M:shape bem})}
Let $\TT_\bullet\in\T$, $T\in\TT_\bullet$, and $m\in\{1,\dots,M\}$ with $T\subseteq\Gamma_m$.
We abbreviate $\widehat T:=\gamma_m^{-1}(T)$.
As the refinement procedure $\refine$ relies on uniform bisection of an element in the parameter domain, we see that $\diam(\widehat T)^{d-1}\simeq|\widehat T|$, where the hidden constants depend only on the dimension $d$ and the initial mesh $\widehat \TT_{0,m}$.
Since $\gamma_m$ is bi-Lipschitz, 
we see that $\diam(\widehat T)\simeq\diam(T)$.
Moreover, \eqref{eq:gram bound} shows that $|\widehat T|\simeq |T|$.
Altogether, we conclude that $\diam(T)^{d-1}\simeq|T|$, where the hidden constants depend only on $d$, $\const{\gamma}$, and $\widehat\TT_{0,m}$.

\subsubsection{Verification of (\ref{M:cent bem})}
\label{subsec:M cent bem}
Let $\TT_\coarse\in\T$, $T\in\TT_\coarse$, and $m\in\{1,\dots,M\}$ with $T\subseteq\Gamma_m$.
We show that  there exists $r\simeq\diam(T)$ with $B_r(T)\cap \Gamma\subset \pi_\coarse(T)$, which verifies \eqref{M:cent bem}.

\noindent \textbf{Step 1:}
According to  Remark~\ref{rem:Nu for p0} and Remark  \ref{rem:admissible bem}, admissibility $\TT_{\coarse}\in\T$ shows  that $|\level(T)-\level( T')|\le 1$ for all $ T'\in \Pi_\coarse(T)$.
Since we only use dyadic bisection, there exists an upper bound of the number of possible configurations of $T$ and $\pi_\bullet(T)$ depending only on the initial meshes $\widehat\TT_{0,m'}$ 
 and (an upper bound of) $\level(T)$.
In particular, this implies that $\diam(T)\lesssim\dist(T,\Gamma\setminus\pi_\bullet(T))$, 
but the hidden constant still depends on (an upper bound for) 
$\level(T)$.
We see that it only remains to consider small elements $T$ with high level.

\noindent \textbf{Step 2:}
In this step, we show that there exists \revision{a node} 
 $z\in\NN_\gamma$  \revision{(see~\eqref{eq:Ngamma})} and a generic constant $C>0$ such that $\pi_\coarse(T)\subseteq \pi_\gamma(z)$ and $B_{C}(T)\cap\Gamma\subseteq \pi_\gamma(z)$ if $\level(T)$ is sufficiently high.
Without loss of generality, we assume that $T\cap\gamma_m([0,1/2]^{d-1})\neq\emptyset$ and set $z:=\gamma_m(0)$.
Note that $\gamma_m([0,1/2]^{d-1})\cap\NN_\gamma=\{z\}$.
Since we assume that the surfaces have no hanging nodes,  $z\not\in\Gamma_{m'}$ implies that $\Gamma_{m'}\cap\gamma_m([0,1/2]^{d-1})=\emptyset$ for all $m'\in\{1,\dots,M\}$.
We abbreviate
\begin{align*}
C_m:=\min_{\substack{m'\in\{1,\dots,M\}\\z\not\in\Gamma_{m'}}} \dist\big(\Gamma_{m'},\gamma_m([0,1/2]^{d-1})\big)=\dist(\Gamma\setminus\pi_\gamma(z),\gamma_m([0,1/2]^{d-1})>0.
\end{align*}
Let $\level(T)$ be sufficiently high such that $\diam(\pi_\coarse(T))<C_m$, which is possible due to \eqref{M:patch bem}--\eqref{M:locuni bem}.
Note that this choice depends only on the dimension $d$, the  constant $\const{\gamma}$, and the initial meshes.
With the assumption that $T\cap\gamma_m([0,1/2]^{d-1})\neq\emptyset$, we derive that  $\pi_\coarse(T)\subseteq \pi_\gamma(z)$. 
The same argument proves that $B_{C_m/2}(T)\cap\Gamma\subseteq \pi_\gamma(z)$ if $\diam(T)<C_m/2$.

\noindent\textbf{Step 3:}
Due to Step~2, we may consider the set $\gamma_z^{-1}(\pi_\coarse(T))\subseteq\overline\pi_\gamma(z)$ provided that $\level(T)$ is sufficiently high. 
Recall that  $|\level( T)-\level( T')|\le 1$ for all $ T'\in \Pi_\coarse(T)$; see Step~1.
With the assumptions for the mapping $\gamma_{z}$ of Section~\ref{subsec:gamma bem}, and 
since we only use dyadic bisection, we see that the number of possible shapes of $\gamma_z^{-1}(T)$ and $\gamma_{z}^{-1}(\pi_{\coarse}(T))$  is uniformly bounded.
In particular, there exists $\overline r\simeq \diam(\gamma_z^{-1}(T))$ with $B_{\overline r}(\gamma_z^{-1}(T))\subseteq \gamma_{z}^{-1}(\pi_{\coarse}(T))$.
Bi-Lipschitz continuity of $\gamma_z$ and Step~2 yield the existence of $r\simeq\diam(T)$ with $r< C_m/2$ such that
\begin{align*}
B_r(T)\cap\Gamma= B_{r}(T)\cap\pi_\gamma(z)\subset \pi_\coarse(T).
\end{align*}
Together with Step~1, this concludes \eqref{M:cent bem}, where the constant $\const{cent}$ depends only on the dimension $d$, the parametrizations $\gamma_{m}$ and $\gamma_{z}$, and the initial meshes $\widehat\TT_{0,m}$ for $m\in\{1,\dots,M\}$ and $z\in\NN_{\gamma}$.

\subsubsection{Verification of (\ref{M:semi bem})}\label{sec:semi bem}
We show that there are only finitely many reference point patches.
Then, Proposition~\ref{prop:semi estimate} will conclude \eqref{M:semi bem}.
Let $\TT_\coarse\in\T$ and $z\in\Gamma$.
According to  Remark~\ref{rem:Nu for p0} and Remark  \ref{rem:admissible bem}, admissibility $\TT_{\coarse}\in\T$ shows  that $|\level(T)-\level( T')|\le 1$ for all $ T'\in \Pi_\coarse(z)$.
Since there are no hanging nodes in $\NN_\gamma$ \revision{(see~\eqref{eq:Ngamma})}, there exists $z'\in\NN_\gamma$ such that $\pi_{\coarse}(z)\subseteq \pi_\gamma(z')$.
With the assumptions for the mapping $\gamma_{z'}$ of Section~\ref{subsec:gamma bem}, and 
since we only use dyadic bisection, we see that the number of possible shapes of $\overline\pi_\coarse(z):=\gamma_{z'}^{-1}(\pi_{\coarse}(z))\subseteq\overline\pi_\gamma(z')$  is uniformly bounded.
More precisely, there exists  a finite set $\set{\widehat\omega_j}{j\in\{1,\dots, J\}}$ of connected subsets $\widehat\omega_j\subset \R^{d-1}$ such that for arbitrary $z\in\Gamma$ and corresponding $z'\in\NN_\gamma$ there exist $j\in\{1,\dots,J\}$ and
an affine bijection $\gamma_{\overline \pi_{\coarse}(z)}:\widehat\omega_j\to \overline\pi_{\coarse}(z)$ with
\begin{align}
\frac{|\gamma_{\overline \pi_{\coarse}( z)}(s)-\gamma_{\overline\pi_{\coarse}( z)}(t)|}{\diam(\overline\pi_{\coarse}( z))}\simeq |s-t| \quad \text{for all }s,t\in \widehat\omega_j.
\end{align}
Since $\gamma_{z'}$ is bi-Lipschitz, there holds that $\diam(\overline\pi_{\coarse}( z))\simeq\diam(\pi_\coarse(z))$, and we see for the mapping $\gamma_{\pi_\coarse(z)}:=\gamma_{z'}\circ\gamma_{\overline\pi_{\coarse}( z)}$ that 
\begin{align}
\frac{|\gamma_{\pi_{\coarse}( z)}(s)-\gamma_{\pi_{\coarse}( z)}(t)|}{\diam(\pi_{\coarse}( z))}\simeq |s-t| \quad \text{for all }s,t\in \widehat\omega_j.
\end{align}
Thus, the application of Proposition~\ref{prop:semi estimate} (on the interior of $\widehat\omega_j$) concludes \eqref{M:semi bem}.
The constant $\const{semi}$ depends only on    the dimension $d$,  the parametrizations $\gamma_{m}$ and $\gamma_{z'}$, and the initial meshes $\widehat\TT_{0,m}$ for $m\in\{1,\dots,M\}$ and $z'\in\NN_{\gamma}$.

%%%%%%%%%%%%%%%%%%%%%%%%%%%%%%%%%%%%%%%%%%%%%%%%%%%%%%%%%%%%%%%%%%%%%%%%%%%%%%%%%%%%%%%%%%%%%
\subsection{Mesh-refinement}
\label{subsec:general refinement bem}
%%%%%%%%%%%%%%%%%%%%%%%%%%%%%%%%%%%%%%%%%%%%%%%%%%%%%%%%%%%%%%%%%%%%%%%%%%%%%%%%%%%%%%%%%%%%%
%For $\TT_\coarse\in\T$ and an arbitrary set of marked elements $\MM_\coarse\subseteq\TT_\coarse$, we associate a corresponding \textit{refinement} $\TT_\fine:=\refine(\TT_\coarse,\MM_\coarse) \in\T$ with $\MM_\coarse\subseteq\TT_\coarse\setminus\TT_\fine$, i.e., at least the marked elements are  refined.
%Moreover, we suppose for the cardinalities that $\#\TT_\coarse<\#\TT_\fine$ if $\MM_\coarse\neq\emptyset$ and  $\TT_\fine=\TT_\coarse$ else.
%We define $\refine(\TT_\coarse)$ as the set of all $\TT_\fine$ such that there exist  meshes $\TT_{(0)},\dots,\TT_{(J)}$ and marked elements $\MM_{(0)},\dots,\MM_{(J-1)}$ with $\TT_\fine=\TT_{(J)}=\refine(\TT_{(J-1)},\MM_{(J-1)}),\dots,\TT_{(1)}=\refine(\TT_{(0)},\MM_{(0)})$ and $\TT_{(0)}=\TT_\coarse$. 
%We assume that there exists a fixed initial mesh $\TT_0\in\T$ with $\T=\refine(\TT_0)$.
%
\revision{In the following Section~\ref{sec:sons bem},} we \revision{show} that there exist $\const{child}\ge2$ and $0<\ro{child}<1$ such that for all \revision{admissible hierarchical} meshes $\TT_\coarse\in\T$ and refinements $\TT_\fine:=\refine(\TT_\coarse,\MM_\coarse)$ with arbitrary marked elements $\MM_\coarse\subseteq\TT_\coarse$, the following elementary properties~\eqref{R:sons bem}--\eqref{R:reduction bem} are satisfied:
\begin{enumerate}[(i)]
\renewcommand{\theenumi}{R\arabic{enumi}}
\bf\item\rm\label{R:sons bem}
\textbf{Child estimate:}
It holds that 
\begin{align*}\#\TT_\fine \le \const{child}\,\#\TT_\coarse,
\end{align*} i.e., one step of refinement leads to a bounded increase of elements.
\bf\item\rm\label{R:union bem}
\textbf{Parent is union of children:}
For all $T\in\TT_\coarse$, it holds that 
\begin{align*}
T=\bigcup\set{{T'}\in\TT_\fine}{T'\subseteq T},
\end{align*}
 i.e., each element $T$ is the union of its successors.
\bf\item\rm\label{R:reduction bem}
\textbf{Reduction of children:}
For all $T\in\TT_\coarse$, it holds that 
\begin{align*}
|T'| \le \ro{child}\,|T|\quad\text{for all }T'\in\TT_\fine\text{ with } T'\subsetneqq T,
\end{align*}
 i.e., successors are uniformly smaller than their parent.
\end{enumerate}
By induction and the definition of $\refine(\TT_\coarse)$, one easily sees that \eqref{R:union bem}--\eqref{R:reduction bem} remain valid if $\TT_\fine$ is an arbitrary mesh in $\refine(\TT_\coarse)$.
In particular, \eqref{R:union bem}--\eqref{R:reduction bem} imply that each refined element $T\in\TT_\coarse\setminus\TT_\fine$ is split into at least two children, which yields that 
\begin{align}\label{eq:R:refine bem}
\#(\TT_\coarse\setminus\TT_\fine)\le \#\TT_\fine-\#\TT_\coarse\quad\text{for all }\TT_\fine\in\refine(\TT_\coarse).
\end{align}
Besides~\eqref{R:sons bem}--\eqref{R:reduction bem}, we \revision{show in Section~\ref{subsec:R closure bem}--\ref{subsec:R overlay bem}} the following less trivial requirements~\eqref{R:closure bem}--\eqref{R:overlay bem} with  constants $\const{clos},\const{over}>0$:
\begin{enumerate}[(i)]
\renewcommand{\theenumi}{R\arabic{enumi}}
\setcounter{enumi}{3}
\bf\item\rm\label{R:closure bem}
\textbf{Closure estimate:}
Let $(\TT_\ell)_{\ell\in\N_0}$ be an arbitrary sequence in $\T$ such that  $\TT_{\ell+1}=\refine(\TT_\ell,\MM_\ell)$ with some $\MM_\ell\subseteq \TT_\ell$ for all $\ell\in\N_0$.
Then, for all $\ell\in\N_0$, it holds  that 
\begin{align*}
\# \TT_\ell-\#\TT_0\le \const{clos}\sum_{j=0}^{\ell-1}\#\MM_j.
\end{align*}
\bf\item\rm\label{R:overlay bem}
\textbf{Overlay property:}
For all $\TT_\coarse,\TT_\star\in\T$, there exists a common refinement $\TT_\fine\in\refine(\TT_\coarse)\cap\refine(\TT_\star)$ which satisfies the overlay estimate
\begin{align*}
\#\TT_\fine \le \const{over}(\#\TT_\star - \#\TT_0)+\#\TT_\coarse.
\end{align*}
\end{enumerate}

\subsubsection{Verification of (\ref{R:sons bem})--(\ref{R:reduction bem})}\label{sec:sons bem}
The child estimate \eqref{R:sons bem} is trivially satisfied with $\const{child}=2^{d-1}$ since each refined element is split into exactly $2^{d-1}$ children.
The union of children property \eqref{R:union bem} holds by definition.
Finally, the proof of \eqref{R:reduction bem} works just as for IGAFEM in \cite[Section~5.3]{igafem}, where the assertion is stated for $\Gamma=\Gamma_1\subset\R^{d-1}$. 
The constant $\ro{child}$ depends only on $d$ and~$\const{\gamma}$.

\subsubsection{Verification of (\ref{R:closure bem})}\label{subsec:R closure bem}
The proof of the closure estimate \eqref{R:closure bem} works as for IGAFEM in \cite[Section~5.4]{igafem}, where the assertion is stated for $\Gamma=\Gamma_1\subset\R^{d-1}$; see also \cite[Section~5.5.7]{diss} for further details.
The constant $\const{clos}$ depends only on $d$, $\const{\gamma}$, $\widehat \TT_{0,m}$, and $(p_{1,m},\dots,p_{d-1,m})$ for $m\in\{1,\dots,M\}$.

\subsubsection{Verification of (\ref{R:overlay bem})}\label{subsec:R overlay bem}
Let $\TT_\bullet,\TT_\star\in\T$.
For each $m\in\{1,\dots,M\}$, let {$\widehat \TT_{\circ,m}$} be the overlay of {$\widehat\TT_{\bullet,m}$} and {$\widehat\TT_{\star,m}$}, i.e.,
\begin{align}
\widehat \TT_{\circ,m}:=\set{\widehat T\in\widehat\TT_{\coarse,m}}{\exists \widehat T'\in\widehat\TT_{\star,m}\quad \widehat T\subseteq\widehat T'}
\cup\set{\widehat T\in\widehat \TT_{\star,m}}{\exists \widehat T'\in\widehat\TT_{\coarse,m}\quad\widehat T\subseteq\widehat T'}
\end{align}

We have proved in \cite[Section~5.5]{igafem} 
that $\TT_{\bullet,m}\in\T_m$, and that 
\begin{align}
\#\TT_{\circ,m}\le \#\TT_{\bullet,m}+\#\TT_{\star,m}-\#\TT_{0,m}.
\end{align}
Summing  all components gives the overlay estimate
\begin{align}
\#\TT_\circ\le \#\TT_\bullet+\#\TT_\star-\#\TT_0.
\end{align}
It remains to show that $\TT_\circ$ is a refinement of $\TT_\bullet$ and $\TT_\star$, i.e., $\TT_\circ\in\refine(\TT_\bullet)\cap\refine(\TT_\star)$.
Clearly, $\TT_\circ$ is finer than these meshes.
Since Proposition~\ref{prop:refineT subset T bem} states that any finer admissible mesh can be reached via $\refine$,
we just have to verify admissibility of $\TT_\circ$.
As each $\TT_{\coarse,m}$ is admissible, we only have to show that  $T\cap T'$ is a common transformed lower-dimensional hyperrectangle for all $T,T'\in\TT_\circ$ with non-empty intersection and  $T\subseteq\Gamma_m$ as well as $T'\subseteq\Gamma_{m'}$ for some $m,m'\in\{1,\dots,M\}$ with $m\neq m'$.
Without loss of generality, we assume that $T\in\TT_\bullet$ and $T'\in\TT_\star$. Further, we may assume that $\dim(T\cap T')>0$.
Then, by definition of $\TT_\circ$, there exist  $T_\star\in\TT_\star$ and $T_\bullet'\in\TT_\bullet$ with $T\subseteq T_\star\subseteq\Gamma_m$ and  $T'\subseteq T_\bullet'\subseteq\Gamma_{m'}$.
Obviously, $T$ and $T_\bullet'$ have non-empty intersection too.
Thus, admissibility of $\TT_\bullet$ shows that $T\cap T_\bullet'$ is a common transformed  hyperrectangle.
We suppose that $T'$ is obtained from $T_\bullet'$ via iterative bisections, i.e., $T'\subsetneqq T_\bullet'$, and lead this to a contradiction. 
The intersection $T\cap T'$ is only a proper subset of a transformed  hyperrectangle of $T$.
Since $T_\star\supseteq T$, the same holds for $T_\star$  instead of $T$, i.e., $T_\star\cap T'$ is only a proper subset of a  transformed hyperrectangle of $T_\star$.
Thus, admissibility of $\TT_\star$ leads to a contradiction, and we see that $T'=T_\bullet'$ shares a common transformed hyperrecangle with $T$.

%%%%%%%%%%%%%%%%%%%%%%%%%%%%%%%%%%%%%%%%%%%%%%%%%%%%%%%%%%%%%%%%%%%%%%%%%%%%%%%%%%%%%%%%%%%%%
\subsection{Boundary element space}\label{subsec:ansatz}
%%%%%%%%%%%%%%%%%%%%%%%%%%%%%%%%%%%%%%%%%%%%%%%%%%%%%%%%%%%%%%%%%%%%%%%%%%%%%%%%%%%%%%%%%%%%%
%With each $\TT_\coarse\in\T$, we associate a finite-dimensional space of vector valued functions
%\begin{align}
%\XX_\coarse \subset L^2(\Gamma)^D\subset H^{-1/2}(\Gamma)^D.
%\end{align}
%Let $\Phi_\coarse\in\XX_\coarse$ be the corresponding Galerkin approximation to the solution $\phi\in H^{-1/2}(\Gamma)^D$ of \eqref{eq:strong}, i.e.,
%\begin{align}\label{eq:pregalerkin bem}
% \dual{\mathfrak{V}\Phi_\coarse}{\Psi_\coarse} = \dual{f}{\Psi_\coarse}
% \quad\text{for all }\Psi_\coarse\in\XX_\coarse.
%\end{align}
%
\revision{In the following Sections~\ref{subsec:S inverse bem}--\ref{subsec:unity bem},} we \revision{show} the existence of constants $\const{inv}>0$,  $\q{loc},\q{proj},\q{supp} \in\N_0$, and $0<\ro{unity}<1$,   such that the following properties~\eqref{S:inverse bem}--\eqref{S:unity bem} hold for all \revision{admissible hierarchical meshes} $\TT_\coarse\in\T$ \revision{with associated hierarchical spline space $\XX_\coarse$}:
\begin{enumerate}[(i)]
\renewcommand{\theenumi}{S\arabic{enumi}}
\bf\item\rm\label{S:inverse bem}
\textbf{Inverse inequality:}
For  all  $\Psi_\coarse\in\XX_\coarse$, it holds that 
\begin{align*}
\norm{h_\bullet^{1/2}\Psi_\coarse}{L^2(\Gamma)}\le \const{inv} \, \norm{\Psi_\coarse}{H^{-1/2}(\Gamma)}.
\end{align*}
\bf\item\rm\label{S:nestedness bem}
\textbf{Nestedness:}
For  all $\TT_\fine\in\refine(\TT_\coarse)$, it holds that 
\begin{align*}
\XX_\coarse\subseteq\XX_\fine.
\end{align*}
\bf\item\rm\label{S:local bem}
\textbf{Local domain of definition:}
For all $\TT_\fine\in\refine(\TT_\coarse)$, all
 $T\in\TT_\coarse\setminus \Pi_\coarse^{\q{loc}}( \TT_\coarse\setminus\TT_\fine)\subseteq\TT_\coarse\cap\TT_\fine$, and $\Psi_\fine\in\XX_\fine$, it holds that 
\begin{align*}\Psi_\fine|_{\pi_\coarse^{\q{proj}}(T)} \in \set{\Psi_\coarse|_{\pi_\coarse^{\q{proj}}(T)}}{\Psi_\coarse\in\XX_\coarse}.
\end{align*}
\bf\item\rm\label{S:unity bem}{\bf Componentwise local approximation of unity:} For all $T\in\TT_\bullet$ and all $j\in\{1,\dots,\D\}$, 
there exists some $\Psi_{\bullet,T,j}\in\XX_\bullet$ with  
\begin{align*}
T\subseteq \supp (\Psi_{\bullet,T,j})\subseteq \pi_\bullet^{\q{supp}}(T) 
\end{align*}
 such that only the $j$-th component  does not vanish, i.e., 
 \begin{align*}
 (\Psi_{\bullet,T,j})_{j'}=0\quad\text{for}\quad j'\neq j,
 \end{align*}
 and
\begin{align*}
\norm{1-(\Psi_{\bullet,T,j})_j}{L^2(\supp(\Psi_{\bullet,T,j}))} \le \ro{unity}{|\supp(\Psi_{\bullet,T,j})|}^{1/2}.
\end{align*}
\end{enumerate}
\begin{remark}\label{rem:approx one}
Clearly, \eqref{S:unity bem} is in particular satisfied if \revision{the considered ansatz space} $\XX_\bullet$ is a product space, i.e., $\XX_\bullet=\prod_{j=1}^{\D}(\XX_{\bullet})_j$, and each component $(\XX_{\bullet})_j\subset L^2(\Gamma)$ satisfies \eqref{S:unity bem}.
\end{remark}

Besides \eqref{S:inverse bem}--\eqref{S:unity bem}, we \revision{show in Section~\ref{subsec:scottbem}} that there exist constants $\const{sz}>0$ as well as $\q{sz}\in\N_0$ such that for all \revision{admissible hierarchical meshes}  $\TT_\coarse\in\T$ \revision{with associated hierarchical spline space $\XX_\coarse$} and  $\SS\subseteq \TT_\bullet$, there exists a linear operator  $\mathcal{J}_{\coarse,\SS}:L^2(\Gamma)^{\D}\to\set{\Psi_\coarse\in\XX_\coarse}{\Psi_\coarse|_{\bigcup(\TT_\coarse\setminus\SS)}=0}$ with the following properties~\eqref{S:proj bem}--\eqref{S:stab bem}:
\begin{enumerate}[(i)]
\renewcommand{\theenumi}{S\arabic{enumi}}
\setcounter{enumi}{4}
\bf\item\rm\label{S:proj bem}
\textbf{Local projection property.}
Let $\q{loc}, \q{proj}\in\N_0$ from \eqref{S:local bem}.
For all $\psi\in L^2(\Gamma)^{\D}$ and $T\in\TT_\coarse$ with $\Pi_\coarse^{\q{loc}}(T)\subseteq\SS$, it holds  that 
\begin{align*}
(\mathcal{J}_{\coarse,\SS} \psi)|_T = \psi|_T, \quad\text{if }\psi|_{\pi_\coarse^{\q{proj}}(T)} \in \set{\Psi_\coarse|_{\pi_\coarse^{{\q{proj}}}(T)}}{\Psi_\coarse\in\XX_\coarse}.
\end{align*}
\bf\item\rm\label{S:stab bem}
\textbf{Local $\boldsymbol{L^2}$-stability.}
For all   $\psi\in L^2(\Gamma)^{\D}$ and $T\in\TT_\coarse$, it holds that 
\begin{align*}
\norm{ \mathcal{J}_{\coarse,\SS} \psi}{L^2(T)}\le \const{sz} \norm{\psi}{L^2(\pi_\coarse^{\q{sz}}(T))}.
\end{align*}
\end{enumerate}

\subsubsection{Verification of (\ref{S:inverse bem})}\label{subsec:S inverse bem}
For piecewise constants and piecewise affine functions on a  triangulation of the boundary of a polyhedral domain $\Omega$, the inverse estimate \eqref{S:inverse bem} is already found in \cite[Theorem~4.7]{inversefaermann}. 
 \cite[Theorem 3.6]{inversesauter} and \cite[Theorem 3.9]{georgoulis} generalized the result to arbitrary piecewise polynomials on curvilinear triangulations.
 In the recent own work~\cite[Proposition~4.1]{resigaconv} and based on the ideas of \cite{inversefaermann}, we proved \eqref{S:inverse bem} for non-rational splines on a one-dimensional piecewise smooth boundary $\Gamma$.
In the proof, we derived the following abstract  criterion for the ansatz functions which is sufficient for  the inverse inequality \eqref{S:inverse bem}.
Although, \cite{resigaconv} considered only $d=2$, the proof works verbatim for arbitrary dimension $d\ge 2$.

\begin{proposition}\label{prop:abs invest}
Let $\TT_\bullet\in\T$ be a general mesh as in Section~\ref{subsec:boundary discrete bem} which satisfies \eqref{M:patch bem}--\eqref{M:semi bem}.
We assume  that the Lipschitz constants of the mappings $\gamma_T:\widehat T\to T$ are uniformly bounded, i.e., there exists a constant $\const{lip}>0$ such that
\begin{align}\label{eq:lipinv}
\const{lip}^{-1}\le \frac{|\gamma_T(s)-\gamma_T(t)|}{|s-t|}\le \const{lip}\quad\text{for all }s,t\in\widehat T\text{ with } T\in\TT_\bullet.
\end{align}
Moreover, let $\psi\in L^2(\Gamma)$ satisfy the following assumption:
There exists a constant $\ro{inf} \in (0,1)$, such that for all $T\in\TT_\bullet$ there exists a hyperrectangular subset  $R_T$ of the interior $T^\circ$  (i.e., $R_T$ has the form $R_T=\gamma_T(\widehat R_T)$ with $\widehat R_T=\prod_{i=1}^{d-1} [a_{T,i},b_{T,i}]\subset \widehat T^\circ$ for some real numbers $a_{T,i}<b_{T,i}$) such that $|R_T|\ge \ro{inf}|T|$,  $\psi$ does not change its sign on $R_T$, and 
\begin{align}\label{eq:min-max}
\operatorname*{essinf}\limits_{x\in R_T}|\psi(x)|\geq \ro{inf} \norm{\psi}{L^\infty(T)},
\end{align} 
where ${\rm essinf}$ denotes the essential infimum.
We further assume that the shape-regularity constants of the sets $\widehat R_T$ are uniformly bounded, i.e., there exists a constant $\const{rec}>0$ such that
\begin{align}\label{eq:Crec}
\max\Big\{\frac{b_{T,i}-a_{T,i}}{b_{T,i'}-a_{T,i'}}:T\in\TT_\bullet\wedge i,i'\in\{1,\dots,d-1\}\Big\}\le \const{rec}\quad\text{for all }T\in\TT_\bullet.
\end{align}
Then, there exists a constant $\const{inv}>0$ such that
\begin{align}
\|h_\bullet^{{1/2}} \psi \|_{L^2(\Gamma)} \leq \const{inv} \|\psi\|_{H^{-{1/2}}(\Gamma)}.
\end{align}
The constant $\const{inv}$ depends only on $d$, $\const{lip}$, $\ro{inf}$,  $\const{rec}$, and  \eqref{M:patch bem}--\eqref{M:semi bem}.\hfill $\square$
  \end{proposition}

To apply Proposition~\ref{prop:abs invest} to hierarchical splines, we need the next elementary lemma which was proved in the recent own work \cite[Proposition~4.1]{resigaconv}. 
\begin{lemma}\label{lem:poly minmax}
Let $p_{}\in\N_0$ be a fixed polynomial degree. 
Let $I\subset \R$ be a compact interval with positive length $|I|>0$. 
Then, there exists a constant $\rho\in(0,1)$ such that for all polynomials $P$ of degree $p$ on $I$, there exists some interval $[a,b]\subset I^\circ$ of length $(b-a)\ge \rho|I|$ such that $P$ does not change its sign on $[a,b]$ and 
\begin{align}
\min_{t\in [a,b]}|P(t)|\geq \rho \,\norm{P}{L^\infty(I)}.
\end{align}
The constant $\rho$ depends only on ${p_{}}$. \hfill$\square$
\end{lemma}

Finally, we come to the proof of the inverse inequality \eqref{S:inverse bem}.
Let $\TT_\bullet\in\T$ be an admissible hierarchical mesh on $\Gamma$.
We recall that $\XX_\bullet$ is a product space of transformed scalar-valued hierarchical splines.
Without loss of generality, we assume that we are in the scalar case $\D=1$.
We show that all $\Psi_\bullet\in\XX_\bullet\subset L^2(\Gamma)$ satisfy the assumptions of Proposition~\ref{prop:abs invest} and hence conclude that $\norm{h_\bullet^{1/2}\Psi_\bullet}{L^2(\Gamma)}\lesssim \norm{\Psi_\bullet}{H^{-1/2}(\Gamma)}$.
We have already seen  that \eqref{M:patch bem}--\eqref{M:semi bem} are satisfied.
Moreover, \eqref{eq:lipinv} is trivially satisfied since each $\gamma_T$ is just the restriction of some $\gamma_m$ to $\widehat T=\gamma_m^{-1}(T)$, where $m\in\{1,\dots,M\}$.
For $T\in\TT_\coarse$, we abbreviate $\widehat \Psi_\bullet:=\Psi_\bullet\circ \gamma_T$.
Due to the regularity \eqref{eq:gram bound} of the parametrizations $\gamma_m$, it is sufficient to find a uniform constant $\widehat \rho_{\inf}\in (0,1)$ and a shape-regular hyperrectangle $\widehat R_T\subset \widehat{T}^\circ$ such that $|\widehat R_T|\ge \widehat\rho_{\inf} |\widehat{T}|$, $\widehat\Psi_\bullet$ does not change sign on $\widehat R_T$, and
\begin{align}\label{eq:parameter spline ass} 
\inf_{t\in \widehat R_T}|\widehat\Psi_\bullet(t)|\ge \widehat\rho_{\inf} \norm{\widehat\Psi_\bullet}{L^\infty(\widehat T)}.
\end{align} 
Indeed,  one easily shows that $|\widehat R_T|\ge \widehat\rho_{\inf} |\widehat{T}|$ implies that $| R_T|\ge \rho_{\inf} |{T}|$ for some uniform constant $ \rho_{\inf}\in (0,1)$; see, e.g., \cite[Section~5.3]{igafem}.
Recall that  $\widehat\Psi_\bullet$ coincides with a tensor-product polynomial $P$.
Hence, there exist polynomials $P_i$ of degree $p_{\max}$  such that $P(t)=\prod_{i=1}^{d-1} P_i(t_i)$.
With the notation $\widehat T=\prod_{i=1}^{d-1} \widehat T_i$ and $\widehat R_T=\prod_{i=1}^{d-1} (\widehat R_T)_i$, we see that the latter inequality is satisfied if 
\begin{align}\label{eq:Pi minmax}
\inf_{t_i\in (\widehat R_T)_i}|P_i(t_i)|\geq \widehat \rho_{\inf}^{\,1/(d-1)} \, \norm{P_i}{L^\infty(\widehat T_i)} \quad\text{for all } i\in\{1,\dots,d-1\}.
\end{align}
We define $(\widehat R_T)_i$ as the interval of Lemma~\ref{lem:poly minmax} corresponding to the polynomial $P_i$ on the interval $I=\widehat T_i$.
With the constant $\rho$ of Lemma~\ref{lem:poly minmax}, we set  $\widehat \rho_{\inf}:=\rho^{d-1}$.
Then, \eqref{eq:Pi minmax}, and therefore \eqref{eq:parameter spline ass} is satisfied.
Moreover, one sees that $|\widehat R_T|\ge  \widehat \rho_{\inf}|\widehat T|$,  and that $\widehat\Psi_\bullet$ doesn't change its sign on $\widehat R_T\subset \widehat T^\circ$.
It remains to prove shape-regularity \eqref{eq:Crec}.
Since, the refinement procedure $\refine$ only uses uniform bisection of elements, the element $\widehat T$ is shape-regular in the sense that $|\widehat T_i|\simeq|\widehat T_{i'}|$ for all $i,i'\in\{1,\dots,d-1\}$.
Together with $|(\widehat R_T)_i|\ge \rho |\widehat T_i|$, this proves  \eqref{eq:Crec}.
Altogether, we conclude \eqref{S:inverse bem}, where the constant $\const{inv}$ depends only on the dimensions $d,D$, the (fixed) number $M$ of boundary parts $\Gamma_m$, the parametrizations $\gamma_m$ and $\gamma_z$, the initial meshes 
 $\widehat \TT_{0,m}$, and the polynomial orders $(p_{1,m},\dots,p_{d-1,m})$ for $m\in\{1,\dots,M\}$ and $z\in\NN_\gamma$.

\subsubsection{Verification of (\ref{S:nestedness bem})}\label{sec:basis on boundary}
Let $\TT_\bullet\in\T$ and $\TT_\circ\in\refine(\TT_\bullet)$.
The nestedness $\XX_\bullet\subseteq\XX_\circ$ was already stated in  \eqref{eq:nested bem}.

\subsubsection{Bases of hierarchical splines on the  boundary}\label{section:basis bem}
In this section, we give two bases for $\XX_\coarse$.
Since $\XX_\coarse$ is a product space of transformed scalar-valued hierarchical splines, it is sufficient to consider $D=1$.
For $m\in\{1,\dots,M\}$,  recall the definition of the set of hierarchical B-splines $\widehat\BB_{\coarse,m}$ from Section~\ref{subsec:parameter hsplines bem}, which is a basis of $\widehat\XX_{\coarse,m}$.
Moreover, let 
\begin{align}
\Trunc_{\coarse,m}:\widehat\XX_{\coarse,m}\to\widehat\XX_{\coarse,m}
\end{align}
 be the well-known linear \textit{truncation procedure}; see, e.g, \cite[Section~5.9]{igafem}. 
According to \cite[Theorem~6 and 9]{juttler2}, the set of \textit{truncated hierarchical B-splines} $\set{\Trunc_{\coarse,m}(\widehat\beta)}{\widehat\beta\in\widehat\BB_{\coarse,m}}$ are an alternate basis of $\widehat\XX_{\coarse,m}$
In general, truncated hierarchical B-splines have a smaller support than the corresponding hierarchical B-splines.
Indeed, one even has the pointwise estimate
\begin{align}\label{eq:bounds for Trunc}
0\le \Trunc_{\coarse,m}(\widehat\beta)\le\widehat\beta\quad\text{for all }\widehat\beta\in\widehat\BB_{\coarse,m};
\end{align}
see, e.g., \cite[Section~5.9]{igafem}.
In contrast to hierarchical B-splines, truncated hierarchical B-splines form a partition of unity, i.e., 
\begin{align}\label{eq:truncated partition}
\sum_{\widehat\beta\in\widehat\BB_{\coarse,m}}\Trunc_{\coarse,m}(\widehat\beta)=1; 
\end{align}
see \cite[Theorem~10]{juttler2}. 
Consequently, there holds that
\begin{align}
\XX_{\coarse,m}={\rm span}(\BB_{\coarse,m})
\quad \text{with basis}\quad \BB_{\coarse,m}:=\set{\widehat\beta\circ\gamma_m^{-1}}{\widehat\beta\in\widehat\BB_{\coarse,m}}. 
\end{align}
For $\beta\in\BB_{\coarse,m}$, let $\Trunc_{\coarse}(\beta):=\Trunc_{\coarse,m}(\beta):=\Trunc_{\coarse,m}(\widehat\beta\circ\gamma_m)\circ\gamma_m^{-1}$ denote the truncation transformed onto $\Gamma_m$.
Then, an alternate basis of $\XX_{\coarse,m}$ is given by
$\set{\Trunc_{\coarse,m}(\beta)}{\beta\in\BB_{\coarse,m}}$.
If we identify functions in $ L^2(\Gamma_m)$ with their extension (by zero) in $L^2(\Gamma)$,  we see that
\begin{align}\label{eq:hierarchical basis bem}
\XX_\coarse={\rm span}(\BB_\coarse)\quad\text{with basis}\quad \BB_\coarse:=\bigcup_{m=1}^M \BB_{\coarse,m}. 
\end{align}
An alternate basis of $\XX_\coarse$ is given by
${\rm span}\set{\Trunc_{\coarse}(\beta)}{\beta\in\BB_{\coarse}}$.

\subsubsection{Verification of (\ref{S:local bem})}
\label{subsec:S local bem}
With Remark~\ref{rem:connected}(a) and \eqref{eq:hierarchical basis bem}, the proof of \eqref{S:local bem} with $\q{loc}:=\q{proj}+2(p_{\max}+1)$ works as for IGAFEM in \cite[Section~5.8]{igafem}, where the assertion is stated for $\Gamma=\Gamma_1\subset\R^{d-1}$; see \cite[Proposition~5.5.12]{diss} for further details.

\subsubsection{Verification of (\ref{S:unity bem})}\label{subsec:unity bem}
Let $\TT_\bullet\in\T$.
First, we recall that $\XX_\bullet$ is a product space of transformed scalar-valued hierarchical splines. 
Without loss of generality, we assume that $\D=1$; see Remark~\ref{rem:approx one}. 
 \cite[Lemma~2.6]{faer2} resp. \cite[Lemma~3.5]{faer3} prove a similar version of \eqref{S:unity bem} for splines on a one-dimensional boundary $\Gamma$ resp. for certain piecewise polynomials of degree $0,1,5,$ and $6$ on  curvilinear triangulations of a two-dimensional boundary $\Gamma$.
There, the proof follows from direct calculations, where  \cite[Lemma~2.6]{faer2} actually only provides a proof of \eqref{S:unity bem} for splines of degree $2$.
In contrast, we will make use of the following abstract result.

\begin{proposition}\label{prop:abs trunc}
Let $\D=1$. 
Let $\TT_\bullet$ be a general mesh with associated space $\XX_\bullet\revision{\subset L^2(\Gamma)}$ as in Section \ref{subsec:boundary discrete bem} which satisfies \eqref{M:patch bem}--\eqref{M:shape bem}.
Assume that  there exists a finite subset $\overline\BB_\bullet\subset \XX_{\bullet}$ which satisfies the following three properties:
\begin{enumerate}[\rm(i)]
\item\label{item:nonnegativity S4} Non-negativity: Each $\overline\beta\in\overline\BB_\bullet$ is non-negative.
\item\label{item:locality S4} Locality:  There is some $\q{supp}'\in\N_0$ such that for all $\overline\beta\in\overline\BB_\bullet$ there exists an element $T_{\overline\beta}\in\TT_\bullet$ with $\supp (\overline\beta)\subseteq \pi_\bullet^{\q{supp}'}(T_{\overline\beta})$.
\item\label{item:partition S4} Partition of unity: It holds  that $\sum_{\overline\beta\in\overline\BB_\bullet} \overline\beta=1$.
\end{enumerate}
Then, \eqref{S:unity bem} is satisfied with $\q{supp}=2\q{supp}'$, and the constant $\ro{unity}$ depends only on \eqref{M:patch bem}--\eqref{M:shape bem} and  $\q{supp}'$.
\end{proposition}
\begin{proof}
Let $T\in\TT_\bullet$.
We set
\begin{align*}
\Psi_{\coarse,T,1}:=\Psi_{\coarse,T}:=\sum_{\substack{\overline\beta\in\overline\BB_\coarse\\ T\subseteq\supp(\overline\beta)}}\overline\beta.
\end{align*}
This implies that $0\le \Psi_{\bullet,T}\le 1$ and $\Psi_{\bullet,T}|_T=1$, which is why we have that $T\subseteq \supp(\Psi_{\bullet,T})$.
Note that $T\subseteq\supp(\overline\beta)$ implies that
  $T\subseteq\supp(\overline\beta)\subseteq \pi_\bullet^{\q{supp}'}(T_{\overline\beta})$.
In particular, we obtain that $T\in\Pi_\bullet^{\q{supp}'}(T_{\overline\beta})$ and hence $\pi_\bullet^{\q{supp}'}(T_{\overline\beta})\subseteq\pi_\bullet^{\q{supp}}(T)$ with $\q{supp}:=2\q{supp}'$.
We conclude that $\supp(\Psi_{\coarse,T})\subseteq \pi_\coarse^{\q{supp}}(T)$.
Finally, there holds that
\begin{align*}
\int_{\supp(\Psi_{\bullet,T})}(1-\Psi_{\bullet,T})^2\,dx\revision{\le}|\supp(\Psi_{\bullet,T})|-|T|
=\Big(1-\frac{|T|}{|\supp(\Psi_{\bullet,T})|}\Big)|\supp(\Psi_{\bullet,T})|\\
\le\Big(1-\frac{|T|}{|\pi_\coarse^{q_{\rm supp}}(T)|}\Big)|\supp(\Psi_{\bullet,T})|\le \ro{unity}^2 |\supp(\Psi_{\bullet,T})|,
\end{align*}
where $0<\ro{unity}<1$ depends only on \eqref{M:patch bem}--\eqref{M:shape bem} and $\q{supp}'$.
\end{proof}
We choose $\overline\BB_\bullet=\set{\Trunc_\bullet(\beta)}{\beta\in\BB_\bullet}$ in Proposition~\ref{prop:abs trunc}.
Then, \eqref{item:nonnegativity S4} follows from \eqref{eq:bounds for Trunc}, \eqref{item:locality S4} with $\q{supp}'=2(p_{\max}+1)$ follows from Remark~\ref{rem:connected}{\rm(a)}, and \eqref{item:partition S4} follows from  \eqref{eq:truncated partition}.
This concludes the proof of \eqref{S:unity bem} with $\q{supp}=4(p_{\max}+1)$, and $\ro{unity}$ depends only on the dimension $d$,  the number $M$ of boundary parts $\Gamma_m$, the constant $\const{\gamma}$, the initial meshes $\widehat\TT_{0,m}$, and $(p_{1,m},\dots,p_{d-1,m})$ for $m\in\{1,\dots,M\}$.

\subsubsection{Verification of (\ref{S:proj bem})--(\ref{S:stab bem})}\label{subsec:scottbem}
Let $\TT_\bullet$ and $\SS\subseteq\TT_\bullet$.
Since $\XX_\coarse$ is a product space of transformed scalar-valued hierarchical splines, we may assume without loss of generality that $D=1$.
For $m\in\{1,\dots,M\}$, we set $\widehat\SS_m:=\set{\gamma_m^{-1}( T)}{T\in\SS\cap\TT_{\bullet,m}}$.
For any hierarchical B-spline $\widehat\beta\in\widehat\BB_{\coarse,m}$, let $\widehat T_{\widehat\beta}\in\widehat\TT_{\coarse,m}$ with $\widehat T_{\widehat\beta}\subseteq\supp(\widehat\beta)$ and   $\level(\widehat T_{\widehat\beta})=\level(\widehat\beta)$, and let $\widehat\beta^*$ be the corresponding dual basis function on $\widehat T_{\widehat\beta}$ 
 from \cite[Section~5.10]{igafem}, i.e.,
\begin{align}
\int_{\widehat T_{\widehat\beta}}\widehat\beta^*\widehat\beta'=\delta_{\widehat\beta\widehat\beta'}\quad\text{for all }\widehat\beta'\in\widehat\BB_{\uni{\level(\widehat\beta)}}.
\end{align}
It is shown in \cite{igafem} that
\begin{align}\label{eq:dual bound}
\norm{\widehat\beta^*}{L^\infty(\widehat T_{\widehat\beta})}\lesssim |\widehat T_{\widehat\beta}|^{-1}.
\end{align} 
We define the operator  $\mathcal{J}_{\bullet,\SS}:L^2(\Gamma)\to \set{\Psi_\bullet\in\XX_\bullet}{\Psi_\bullet|_{\bigcup(\TT_\bullet\setminus\SS)}=0}$  via 
\begin{align}
(\mathcal{J}_{\bullet,\SS}\psi)\circ\gamma_m:= \widehat {\mathcal{J}}_{\bullet,m,\widehat\SS_m}(\psi\circ\gamma_m)\quad\text{for all }m\in\{1,\dots,M\},
\end{align}
where 
\begin{align}\label{eq:hatJSm}
\widehat {\mathcal{J}}_{\bullet,m,\widehat\SS_m}:L^2(\widehat\Gamma_m)\to\widehat\XX_{\bullet,m}, \quad
\widehat\psi\mapsto\sum_{\substack{\widehat\beta\in\widehat\BB_{\bullet,m}\\ \supp(\widehat\beta)\subseteq\bigcup\widehat\SS_m}}\int_{\widehat T_{\widehat\beta}}\widehat\beta^*\widehat\psi\,dx\,\Trunc_{\bullet,m}(\widehat\beta).
\end{align}
Recall that $0\le \Trunc_{\bullet,m}(\widehat\beta)\le\widehat\beta$ (see \eqref{eq:bounds for Trunc}), which is why ${\mathcal{J}}_{\bullet,\SS}$ clearly maps into the desired space $\set{\Psi_\bullet\in\XX_\bullet}{\Psi_\bullet|_{\bigcup(\TT_\bullet\setminus\SS)}=0}$.

We come to the verification of the  properties \eqref{S:proj bem}--\eqref{S:stab bem}.
Let $\q{proj}:=2(p_{\max}+1)$ and $\q{loc}=\q{proj}+2(p_{\max}+1)$ of Section~\ref{subsec:S local bem}.
Moreover, let $T\in \TT_\bullet$ with $\Pi_\bullet^{\q{loc}}(T)\subseteq\SS$ and $m\in\{1,\dots,M\}$ with $T\subseteq\Gamma_m$.
Recall the notation $\widehat T=\gamma_m^{-1}(T)$.
For all $\psi\in L^2(\Gamma)$, there holds with the abbreviation    $\widehat\psi:=\psi\circ\gamma_m$ that 
\begin{align*}
({\mathcal{J}}_{\bullet,\SS}\psi)\circ\gamma_m|_{\widehat T}=(\widehat {\mathcal{J}}_{\bullet,m,\widehat\SS_m}\widehat\psi)|_{\widehat T}=\sum_{\substack{\widehat\beta\in\widehat\BB_{\bullet,m}\\ \supp(\widehat\beta)\subseteq\bigcup\widehat\SS_m}}\int_{\widehat T_{\widehat\beta}}\widehat\beta^*\widehat\psi\,dx\,\Trunc_{\bullet,m}(\widehat\beta)|_{\widehat T}.
\end{align*}
Note that $\Trunc_{\bullet,m}(\widehat\beta)|_{\widehat T}\neq0$  implies that $|\supp(\widehat\beta)\cap\widehat T|=0$.
Due to Remark~\ref{rem:connected}(a), $|\supp(\widehat\beta)\cap\widehat T|>0$ implies that $\supp(\widehat\beta)\subseteq  \Pi^{\q{proj}}_{\bullet,m}(\widehat T)\subseteq \Pi^{\q{loc}}_{\bullet,m}(\widehat T)$, 
\revision{where $\Pi_{\coarse,m}$ denotes the patch with respect to the mesh $\widehat\TT_{\coarse,m}$ defined analogously as in \eqref{eq:patch defined}}. 
We abbreviate $\Pi^{\q{loc}}_{\bullet,m}( T):=\set{\gamma(\widehat T')}{\widehat T'\in \Pi^{\q{loc}}_{\bullet,m}( \widehat T)}$ and  note that $\Pi^{\q{loc}}_{\bullet,m}( T)\subseteq \Pi^{\q{loc}}_{\bullet}( T)\cap \TT_{\coarse,m}\subseteq\SS\cap \TT_{\coarse,m}$.
This yields that $\Pi^{\q{loc}}_{\bullet,m}(\widehat T)\subseteq\widehat\SS_m$.
Hence, we see that
\begin{align*}
({\mathcal{J}}_{\bullet,\SS}\psi)\circ\gamma_m|_{\widehat T}=\sum_{\substack{\widehat\beta\in\widehat\BB_{\bullet,m}} }\int_{\widehat T_{\widehat\beta}}\widehat\beta^*\widehat\psi\,dx\,\Trunc_{\bullet,m}(\widehat\beta)|_{\widehat T}.
\end{align*}
If  $\psi$ satisfies  that $\psi|_{\pi_\bullet^{\q{proj}}(T)}\in\set{\Psi_\bullet|_{\pi_\bullet^{\q{proj}}(T)}}{\Psi_\bullet\in\XX_\bullet}$, 
then there exists $\widehat\Psi_{\bullet,m}\in\widehat \XX_{\bullet,m}$ with $\widehat\psi|_{\pi_{\bullet,m}^{\q{proj}}(\widehat T)}=\widehat\Psi_{\bullet,m}|_{\pi_{\bullet,m}^{\q{proj}}(\widehat T)}$.
With $\widehat T_{\widehat\beta}\subseteq\supp(\widehat \beta)$, we see as before that $\Trunc_{\bullet,m}(\widehat\beta)|_{\widehat T}\neq0$ implies that $\widehat T_{\widehat\beta}\subseteq\Pi_{\coarse,m}^{\q{proj}}(\widehat T)$, which leads to
\begin{align*}
(\mathcal{J}_{\bullet,\SS}\psi)\circ\gamma_m|_{\widehat T}=\sum_{\substack{\widehat\beta\in\widehat\BB_{\bullet,m}} }\int_{\widehat T_{\widehat\beta}}\widehat\beta^*\widehat\Psi_{\coarse,m}\,dx\,\Trunc_{\bullet,m}(\widehat\beta)|_{\widehat T}.
\end{align*}
The right-hand side coincides with the projection operator from \cite[Theorem~4]{speleers} corresponding to the mesh $\widehat\TT_{\bullet,m}$ applied to $\widehat \Psi_{\coarse,m}$. 
Since $\widehat\Psi_{\coarse,m}\in\widehat\XX_{\coarse,m}$, this yields that
\begin{align*}
(\mathcal{J}_{\bullet,\SS}\psi)\circ\gamma_m|_{\widehat T}=\widehat \Psi_{\coarse,m}|_{\widehat T}=\widehat\psi|_{\widehat T}.
\end{align*}
This proves the local projection property \eqref{S:proj bem}.

Finally, we prove local $L^2$-stability \eqref{S:stab bem}.
Let again $T\in\TT_\bullet$ and $m\in\{1,\dots,M\}$ with $T\subseteq\Gamma_m$. 
With the notation from before, the boundedness of the Gram determinant \eqref{eq:gram bound}
shows that
\begin{align*}
\norm{\mathcal{J}_{\bullet,\SS}\psi}{L^2(T)}\simeq\norm{\widehat {\mathcal{J}}_{\bullet,m,\widehat\SS_m}\widehat\psi}{L^2(\widehat T)}.
\end{align*}
With \eqref{eq:dual bound}, it follows as for IGAFEM in \cite[Lemma~5.6]{igafem} (which proves local $L^2$-stability for a similar operator mapping into $\widehat\XX_\coarse\cap H_0^1(\widehat\Gamma_m)$) that
\begin{align*}
\norm{\widehat {\mathcal{J}}_{\bullet,m,\widehat\SS_m}\widehat\psi}{L^2(\widehat T)}\lesssim\norm{\widehat\psi}{L^2(\pi_{\bullet,m}^{\q{loc}}(\widehat T))}.
\end{align*}
With boundedness of the Gram determinant \eqref{eq:gram bound} and $\pi_{\bullet,m}^{\q{loc}}( T):=\gamma_m(\pi_{\bullet,m}^{\q{loc}}( \widehat T))$, we derive that 
\begin{align*}
\norm{\widehat\psi}{L^2(\pi_{\bullet,m}^{\q{loc}}(\widehat T))}\simeq \norm{\psi}{L^2(\pi_{\bullet,m}^{\q{loc}}(T))}\le  \norm{\psi}{L^2(\pi_{\bullet}^{\q{loc}}(T))}.
\end{align*}
This concludes the proof of \eqref{S:stab bem}.
The constant $\const{sz}$ depends only on the dimension $d$, the constant $\const{\gamma}$, the initial meshes  $\widehat\TT_{\bullet,m}$,  and the polynomial orders $(p_{1,m},\dots,p_{d-1,m})$ for $m\in\{1,\dots, M\}$.

\subsection{Proof of Theorem~\ref{thm:main bem} for rational hierarchical splines}
\label{sec:rational main proof bem}
As mentioned in Remark~\ref{rem:rational main bem},  Theorem~\ref{thm:main bem} is still valid if one replaces the ansatz space $\XX_\bullet$ for $\TT_\coarse\in\T$ by rational hierarchical splines, i.e., by the set
\begin{align}
\XX_\bullet^{W_0}=\Big\{{W_0^{-1}\Psi_\bullet}:\Psi_\bullet\in\XX_\bullet\Big\},
\end{align}
where $\widehat W_{0,m}=W_0\circ\gamma_m\in\widehat\SS^{(p_{1,m},\dots,p_{d-1,m})}(\widehat\KK_{0,m})$ is a fixed positive weight function in the initial space of hierarchical splines  for all $m\in\{1,\dots,M\}$, where we additionally assume the representation \eqref{eq:W0 representation}.
Indeed, the mesh properties \eqref{M:patch bem}--\eqref{M:semi bem} as well as the refinement properties \eqref{R:sons bem}--\eqref{R:overlay bem} of Section~\ref{sec:abstract setting bem} are independent of the discrete spaces.
To verify the validity of Theorem~\ref{thm:main bem} in the rational setting, it thus only remains to verify the properties \eqref{S:inverse bem}--\eqref{S:stab bem} for the rational boundary element spaces.

To see the inverse estimate \eqref{S:inverse bem}, it is again sufficient to consider $D=1$.
In Section~\ref{subsec:S inverse bem}, we proved \eqref{S:inverse bem} for $\XX_\coarse$ by applying Proposition~\ref{prop:abs invest} for all $\Psi_\coarse\in\XX_\coarse$. 
With the notation from Section~\ref{subsec:S inverse bem}, we showed that
\begin{align*}
\inf_{x\in R_T}|\Psi_\coarse(x)|\geq \ro{inf} \norm{\Psi_\coarse}{L^\infty(T)}\quad\text{for all }T\in\TT_\coarse, \Psi_\coarse\in\XX_\coarse,
\end{align*} 
where $\Psi_\coarse$ does not change its sign on $R_T$.
With $0<w_{\min}:=\inf_{x\in\Gamma}W_0(x)$, $w_{\max}:=\sup_{x\in\Gamma}W_0(x)$, and $\widetilde\rho_{\inf}:=\ro{inf}w_{\min}/ w_{\max}$, this yields that, for all $\Psi_\coarse\in\XX_\coarse$, 
\begin{align*}
\widetilde\rho_{\inf} \norm{W_0^{-1}\Psi_\coarse}{L^\infty(T)}\le \frac{\ro{inf}}{w_{\max}}\norm{\Psi_\coarse}{L^\infty(T)}\le   \frac{1}{w_{\max}}\inf_{x\in R_T}|\Psi_\coarse(x)|\le  \inf_{x\in R_T}|W_0^{-1}\Psi_\coarse(x)|.
\end{align*}
In particular, the conditions for Proposition~\ref{prop:abs invest} are also satisfied for the functions in $\XX_\coarse^{W_0}$, which concludes \eqref{S:inverse bem}.

The properties \eqref{S:nestedness bem}--\eqref{S:local bem} depend only on the numerator of  the rational hierarchical splines   and thus transfer.

For the proof of \eqref{S:unity bem}, we exploit the representation \eqref{eq:W0 representation} to verify  the conditions of the abstract Proposition~\ref{prop:abs trunc}.
Again, we may assume that $D=1$.
Let $\TT_\coarse\in\T$.
Note that $\widehat W_{0,m}$ is also an element of the standard tensor-product spline space $\widehat\SS^{(p_{1,m},\dots,p_{d-1,m})}(\widehat\KK_{\uni{k},m})$ for all $m\in\{1,\dots,M\}$ and  $k\in\N_0$.
In particular, it can  be written as  linear combination of B-splines in $\widehat\BB_{\uni{k},m}$. 
The representation \eqref{eq:W0 representation} and the two-scale relation with only non-negative coefficients between bases of consecutive levels of \cite[Section~11]{boor} yields that the corresponding coefficients are  non-negative.
Therefore, the preservation of coefficients property of \cite[Theorem~1]{speleers} or \cite[Theorem~12]{strongly} implies that also the coefficients of the linear combination of $\widehat W_{0,m}$ in $\set{\Trunc_{\coarse,m}(\widehat\beta)}{\widehat\beta\in\widehat\BB_{\coarse,m}}$ are non-negative, i.e., 
\begin{align}
\widehat W_{0,m}= \sum_{\widehat\beta\in \widehat\BB_{\coarse,m}} \widetilde  w_{\coarse,m,\widehat\beta}\,\Trunc_{\coarse,m}(\widehat \beta)\quad\text{with }\widetilde  w_{\coarse,m,\widehat\beta}\ge 0.
\end{align}
If we identify functions in $L^2(\Gamma_m)$ with their extension (by zero) in $L^2(\Gamma)$, we can choose
\begin{align*}
\overline\BB_\coarse:=\bigcup_{m=1}^M \left\{\Big(\frac{\widetilde  w_{\coarse,m,\widehat\beta}}{\widehat W_{0,m}}\,\Trunc_{\coarse,m}(\widehat\beta)\Big)\circ\gamma_m^{-1}:\widehat\beta\in\widehat\BB_{\coarse,m}\right\}\subset\XX_\coarse^{W_0}.
\end{align*}
As in Section~\ref{subsec:unity bem}, one sees that this choice satisfies the assumptions of Proposition~\ref{prop:abs trunc}.

To see \eqref{S:proj bem} and \eqref{S:stab bem}, we  define the corresponding projection operator
\begin{align}
{\mathcal{J}}_{\bullet,\SS}^{W_0} :L^2(\Gamma)^{\D}\to\set{\Psi_\coarse\in\XX_\coarse}{\Psi_\coarse|_{\bigcup(\TT_\coarse\setminus\SS)}=0}, \quad\psi\mapsto W_0^{-1}{\mathcal{J}}_{\bullet,\SS}(W_0\psi).
\end{align}
The desired properties transfer immediately from the non-rational case.